\documentclass[a4paper,10pt]{amsart}

\usepackage[T1]{fontenc}
\usepackage[utf8]{inputenc}
\usepackage[english]{babel}

\usepackage{xcolor}
\usepackage{amsmath}
\usepackage{amsthm}
\usepackage{amsxtra}
\usepackage{amsfonts,amssymb}
\usepackage{mathtools}
\usepackage{mathdots}

\usepackage{hyperref}

\usepackage{dsfont}

\usepackage{caption}

\usepackage{young}
\usepackage{ytableau}

\newtheorem{Th}{Theorem}[section]
\newtheorem{Prop}[Th]{Proposition}
\newtheorem{Lm}[Th]{Lemma}

\theoremstyle{definition}
\newtheorem{Def}[Th]{Definition}

\newtheorem{Rem}{Remark}[section]
\newtheorem{Exa}{Example}[section]

\numberwithin{equation}{section}

\newcommand{\Stab}{\operatorname{Stab}}
\newcommand{\Ind}{\operatorname{Ind}}

\newcommand{\diag}{\operatorname{diag}}

\newcommand{\Tab}{\operatorname{Tab}}
\newcommand{\Aut}{\operatorname{Aut}}
\newcommand{\supp}{\operatorname{supp}}
\newcommand{\nat}{\mathrm{nat}}
\newcommand{\sym}{\mathrm{sym}}
\newcommand{\sgn}{\operatorname{sgn}}
\newcommand{\reg}{\mathrm{reg}}
\newcommand{\triv}{\mathrm{triv}}
\newcommand{\Tr}{\operatorname{Tr}}
\newcommand{\tr}{\operatorname{tr}}

\newcommand{\ex}{\operatorname{ex}}
\newcommand{\Char}{\operatorname{Char}}
\newcommand{\exchar}{\ex\Char}

\newcommand{\Fix}{\operatorname{Fix}}

\title[Characters of the infinite wreath product]{Indecomposable characters of the wreath product of a torus by the infinite block-diagonal symmetric group
} % $\mathbb{T}\wr\mathfrak{S}_{\infty}$
\author{Nikolay Nessonov} 
\address{B.~Verkin Institute for Low Temperature Physics and Engineering of the National Academy of Sciences of Ukraine, 47 Nauky Ave., Kharkiv, 61103, Ukraine}
\email{n.nessonov@gmail.com}

\author{Nhok Tkhai Shon Ngo} 
\address{Institute of Science and Technology Austria (ISTA),
Am Campus 1,
3400 Klosterneuburg,
Austria}
\email{nhoktkhaishon.ngo@ista.ac.at}

\date{\today}

\begin{document}

\begin{abstract}
In this paper we obtain a complete description of all indecomposable characters (central positive-definite functions) of ``block-diagonal'' inductive limits of semidirect products $\mathbb{T}^N\rtimes\mathfrak{S}_{N}$ of the compact $N$-dimensional torus $\mathbb{T}^{N}$ with the symmetric group $\mathfrak{S}_{N}$. The classification is proved by means of a variant of the ergodic method of Vershik and Kerov, namely by studying the weak limits of characters of finite-dimensional subgroups. As an application of our results, we give another proof of the classification of indecomposable characters of the infinite unitary group with block diagonal embeddings.

\medskip
		
\emph{Keywords:} character, factor representation, infinite-dimensional group.
\end{abstract}

\maketitle

\setcounter{tocdepth}{1}

\tableofcontents

\section{Introduction}\label{sect:intro}

\subsection{Wreath products of an abelian group by the symmetric groups.}
Fix a second-countable compact abelian group $A$. For a positive integer $N$ denote by $\mathfrak{S}_{N}$ the group of all permutations of the set $\{1,2,\ldots,N\}$. Then, an element $s\in\mathfrak{S}_{N}$ naturally acts on the $N$-fold Cartesian product $A^{N}$ by permuting the coordinates $A^N\ni \left( a_1,a_2,\ldots,a_N\right)\mapsto\left( a_{s^{-1}(1)},a_{s^{-1}(2)},\ldots, a_{s^{-1}(N)}\right)\in A^N$. Using this action one constructs the semidirect product $G_N(A)=A^{N}\rtimes\mathfrak{S}_{N}$. The resulting group $G_N(A)$ is called the \emph{wreath product} of $A$ by $\mathfrak{S}_N$ and often denoted by $A\wr\mathfrak{S}_N$. It is clear that the group $G_N(A)$ is compact and hence its representation theory is well understood. In particular, the irreducible representations of $G_N(A)$ are classified by the functions $\lambda\colon\widehat{A}\to\mathbb{Y}$ with $\sum_{\alpha\in\widehat{A}}|\lambda(\alpha)|=N$, where $\widehat{A}$ is the \emph{Pontryagin dual group} of $A$ and $\mathbb{Y}=\bigcup_{k\ge 0}\mathbb{Y}_k$ is the set of all integer partitions (or, equivalently, \emph{Young diagrams}, see Appendix \ref{sect:appendix} for more details).
Moreover, using the formula of Frobenius, one can calculate the corresponding irreducible characters of $G_N(A)$ in terms of $\widehat{A}$ and characters of symmetric groups.

\subsubsection{Infinite-dimensional versions.}
The problem of classifying the indecomposable characters becomes more difficult if one considers infinite-dimensional analogues of the group $G_N(A)$.
For instance, one can consider the inductive limit $\bigcup_{N\ge 1}G_N(A)$ of the sequence
\begin{equation}\label{eq:standard_ind_limit}
G_{1}(A)\to G_2(A)\to\ldots\to G_{N}(A)\to G_{N+1}(A)\to\ldots\to\bigcup_{N\ge 1}G_N(A),
\end{equation}
where we use the natural embeddings $A^{N}\to A^{N+1}$ and $\mathfrak{S}_{N}\to\mathfrak{S}_{N+1}$ given by
\begin{equation*}
(a_1,\ldots,a_N)\in A^N\mapsto (a_1,\ldots,a_N,1_A)\in A^{N+1},~\text{where}~ 1_A ~\text{is identity of}~ A,
\end{equation*}
and $\sigma\in\mathfrak{S}_N$ is considered to be an permutation  $\tilde{\sigma}\in\mathfrak{S}_{N+1}$ such that $\tilde{\sigma}(k)=\sigma(k)$ if $k\leq N$ and $\tilde{\sigma}(N+1)=N+1$.

It is a remarkable fact that character theory can be built for $\mathfrak{S}_\infty=\bigcup\limits_{N=1}^\infty\mathfrak{S}_N$. The corresponding results were discovered by E. Thoma \cite{Thoma}. In this case the indecomposable characters are parameterized by countable sets of positive numbers (called \emph{Thoma parameters}) satisfying certain inequalities.
It turns out that the classification of indecomposable characters for the inductive limit \eqref{eq:standard_ind_limit} is very similar to those obtained by Thoma.

Namely, the first author and Dudko obtained a Thoma-type classification of the indecomposable characters in this case (see \cite[Theorem 9]{Dudko_Nessonov}). In fact, they considered a more general case of $\bigcup_{N\ge 1}(\Gamma\wr\mathfrak{S}_N)$, where $\Gamma$ is any (possibly non-abelian) separable topological group. The argument in \cite{Dudko_Nessonov} relied on the multiplicativity property of indecomposable characters and operator-algebraic methods.
Besides that, Hirai, Hirai and Hora in \cite{Hirai_Hirai_Hora_1} and \cite{Hirai_Hirai_Hora_2} described all weakly converging sequences of characters of finite-dimensional groups $G_{N}(T)$ for an arbitrary (not necessarily abelian) compact group $T$. Using this and the fact that any indecomposable character on the inductive limit can be approximated by a sequence of characters of groups $G_N(T)$, the authors found all indecomposable characters of the group $\bigcup_{N\ge 1} G_N(T)$. The results of Hirai, Hirai and Hora extend a result of Boyer \cite{Boyer_2005} who considered the case of groups $G_N(T)$ for a finite group $T$.
Finally, let us note that that the idea of approximating an indecomposable character of an ``infinite-dimensional group'' by the characters of finite-dimensional subgroups is essentially a variant of the \emph{ergodic method} originally due to Vershik and Kerov \cite{Vershik_Kerov}.

We note, however, that the works of Boyer \cite{Boyer_2005}, Hirai, Hirai, Hora \cite{Hirai_Hirai}, \cite{Hirai_Hirai_Hora_1}, \cite{Hirai_Hirai_Hora_2} and the first author with Dudko \cite{Dudko_Nessonov} on inductive limits of wreath products with symmetric groups concern only the natural embeddings of the form $A^{N}\rtimes\mathfrak{S}_{N}\hookrightarrow A^{N+1}\rtimes\mathfrak{S}_{N+1}$.
In the present paper we study the indecomposable characters of inductive limits of groups $G_N(A)$ with respect to the \emph{block-diagonal embeddings} which we define in the next section.

\subsection{Limits with respect to block-diagonal embeddings and UHF $\mathbb{C}^*$-algebras}\label{subsect:block_diag_unitary}
Choose a sequence of positive integers $\widehat{\mathbf{n}}=\{n_k\}_{k=1}^{\infty}$, where $n_k>1$, and set $N_k=n_1\ldots n_k$. Then, the uniformly hyperfinite (UHF) $\mathbb{C}^*$-algebra associated to $\widehat{\mathbf{n}}$ is the norm closure of the union of an increasing chain $\operatorname{Mat}_{N_1}\stackrel{\mathfrak{i}_{1,2}}{\rightarrow} \operatorname{Mat}_{N_2}\stackrel{\mathfrak{i}_{2,3}}{\rightarrow}\ldots$ of full matrix algebras, where $\operatorname{Mat}_{N_i}$ is the algebra of $N_i\times N_i$ complex matrices, $\operatorname{Mat}_{N_i}\ni a\rightarrow\mathfrak{i}_{i,i+1}(a)=  a\otimes I_{n_i}\in \operatorname{Mat}_{N_{i+1}}$ and $ I_{n_i}$  is the identity in the $\operatorname{Mat}_{n_i}$, see Glimm \cite{Glimm}.

Motivated by this idea, one can construct similar embeddings for groups and form the corresponding inductive limits. For example, by considering the unitary subgroups $\operatorname{U}_{N_i}$ of $\operatorname{Mat}_{N_i}$, we obtain a sequence of embeddings
\begin{equation}\label{eq:ind_lim_unitary}
\operatorname{U}_{N_1}\to\operatorname{U}_{N_2}\to\ldots\to\operatorname{U}_{{N_i}}\to\ldots
\end{equation}
These embeddings are examples of block-diagonal embeddings for unitary groups. More generally, we have embeddings $\operatorname{U}_{N}\hookrightarrow\operatorname{U}_{MN}$ which on the level of matrices\footnote{This matrix interpretation also justifies the adjective ``block-diagonal''.} look like as follows:
\begin{equation*}
\operatorname{U}_{N}\ni u\operatorname\mapsto
\begin{pmatrix}
u & & &
\\
& u & &
\\
& & \ddots &
\\
& & & u
\end{pmatrix}\in\operatorname{U}_{MN}.
\end{equation*}
Generalizing this construction for other groups yields new examples of infinite-dimensional or \emph{big} groups. For example, one can consider the subgroup of diagonal matrices or the subgroup of permutation matrices in $\operatorname{U}_{N_i}$ and use the same embeddings. The corresponding inductive limits generate inside $\varinjlim \operatorname{U}_{N_i}$ the wreath product of a torus $\mathbb{T}$ by the infinite block-diagonal symmetric group. This is the main object of interest in this paper. We discuss next this construction and its generalizations in more detail.

\subsubsection{Block-diagonal embeddings.}\label{Block-diagonal embeddings}
Recall that $A$ is a compact abelian group. Let us define the block-diagonal embeddings for groups $G_{N}(A)=A^{N}\rtimes\mathfrak{S}_{N}$. For any positive integers $M$ and $N$ the block-diagonal embedding $G_{N}(A)\to G_{MN}(A)$ is induced by the following embeddings $A^{N}\to A^{MN}$ and $\mathfrak{S}_{N}\to\mathfrak{S}_{MN}$:
\begin{equation*}
\begin{split}
&A^N\ni(a_1,\ldots,a_N)
\mapsto(\underbrace{a_1,\ldots,a_N,\ldots,a_1,\ldots,a_N}_{MN})\in A^{MN},
\\
&\mathfrak{S}_N\ni s\mapsto \tilde s\in\mathfrak{S}_{MN},~\text{where}~\tilde s(jN+k)=jN+ s(k),\quad j=0,\ldots,M-1,~k=1,\ldots,N.
\end{split}
\end{equation*}
Note that these embeddings are compatible with each other in the following sense: for any positive integers $N_1$, $N_2$, $N_3$ such that $N_1$ divides $N_2$ and $N_2$ divides $N_3$, the block-diagonal embedding $G_{N_1}(A)\to G_{N_3}(A)$ is a composition of block-diagonal embeddings $G_{N_1}(A)\to G_{N_2}(A)$ and $G_{N_2}(A)\to G_{N_3}(A)$.

To construct the inductive limit, we again fix a sequence positive integers $\widehat{\mathbf{n}}=\{n_k\}_{k=1}^{\infty}$ with $n_k>1$ and put $N_k=n_1\ldots n_k$. The discussion above implies that there are compatible sequences of embeddings
\begin{equation*}
\begin{array}{rcccccccc}
A^{N_1}
&\hookrightarrow&
A^{N_2}
&\hookrightarrow& 
\cdots 
&\hookrightarrow&
A^{N_k}
&\hookrightarrow&
\cdots,
\\[0.35em]
\mathfrak{S}_{N_1}
&\hookrightarrow&
\mathfrak{S}_{N_2}
&\hookrightarrow& 
\cdots 
&\hookrightarrow&
\mathfrak{S}_{N_k}
&\hookrightarrow&
\cdots,
\\[0.35em]
G_{N_1}(A)
&\hookrightarrow&
G_{N_2}(A)
&\hookrightarrow& 
\cdots 
&\hookrightarrow&
G_{N_k}(A)
&\hookrightarrow&
\cdots.
\end{array}
\end{equation*}
We denote the corresponding inductive limits by $A^{\widehat{\mathbf{n}}}$, $\mathfrak{S}_{\widehat{\mathbf{n}}}$ and $G_{\widehat{\mathbf{n}}}(A)$, respectively. We will regard $A^{\widehat{\mathbf{n}}}$ and $\mathfrak{S}_{\widehat{\mathbf{n}}}$ as subgroups of $G_{\widehat{\mathbf{n}}}(A)$. 
The topological group $G_{\widehat{\mathbf{n}}}(A)$ is not locally compact, hence its representation theory is not covered by the standard theory (see Section \ref{sect:prelim} for more details). Note also that the center of $G_{\widehat{\mathbf{n}}}(A)$ is isomorphic to $A$ since the same holds for $G_N(A)$ (central elements are given by elements $(a,a,\ldots,a)\in A^N$). We can naturally identify an element $\mathfrak{t}=(a_1,a_2,\ldots, a_{N_k})\in A^{N_k}$ with the function $\mathfrak{t}$ on $\mathbb{X}_N=\{1,2,\ldots,N\}$ defined by $\mathfrak{t}(i)=a_i$, $i\in\mathbb{X}_N$. Thus, $A^{\widehat{\mathbf{n}}}$ can be identified with the abelian group of $A$-valued functions on the compact space $\mathbb{X}_{\widehat{\mathbf{n}}}=\prod_{j\ge 1}\mathbb{X}_{n_j}$. In turn, the elements of the group $\mathfrak{S}_{\widehat{\mathbf{n}}}$ can be viewed as automorphisms of the compact space $\mathbb{X}_{\widehat{\mathbf{n}}}$. Namely, $\sigma \in\mathfrak{S}_{N_k}\subset\mathfrak{S}_{\widehat{\mathbf{n}}}$ acts on $x=(x_{N_k},y)\in \mathbb{X}_{N_k}\times(\mathbb{X}_{n_{k+1}}\times\mathbb{X}_{n_{k+2}}\times\ldots)=\mathbb{X}_{\widehat{\mathbf{n}}}$ by $x=(x_{N_k},y)\mapsto x=(\sigma(x_{N_k}),y)$ (see more details in Section \ref{sect:measure_space_interpret} for case $A=\mathbb{T}$). 

The goal of the present paper is to classify all $\mathrm{II}_1$ factor-representations of the resulting group $G_{\widehat{\mathbf{n}}}(A)$ in the case $A=\mathbb{T}$ (however, we believe that our methods can be extended to more general cases, see Section \ref{sect:gen_ab_group}). This is essentially equivalent to the classification of indecomposable characters (normalized central positive-definite functions, see Section \ref{sect:intro_indec_har_sph_rep} below) on this group. Our approach uses a version of the \emph{ergodic method} due to Vershik and Kerov, namely, the approximation by characters of the pre-limit finite-dimensional subgroups.

\subsubsection{The case of one-dimensional torus.}\label{subsubsect:group_realizations_torus}
Now let us discuss the case when $A=\mathbb{T}=\{z\in\mathbb{C}:|z|=1\}$ in a more detail. In this case the Pontryagin dual of $A$ is $\widehat{\mathbb{T}}\simeq\mathbb{Z}$. The group $G_N(\mathbb{T})=\mathbb{T}^{N}\rtimes\mathfrak{S}_{N}$ can be realized as the group of $N\times N$ matrices which have exactly one non-zero entry in each row and column and whose all non-zero entries belong to $\mathbb{T}$. In other words, $G_N(\mathbb{T})$ is isomorphic to the normalizer of the subgroup of diagonal matrices inside the unitary group $\operatorname{U}_{N}$.
Note that under this interpretation the block-diagonal embedding $\mathbb{T}^{N}\rtimes\mathfrak{S}_{N}\hookrightarrow\mathbb{T}^{MN}\rtimes\mathfrak{S}_{MN}$ is induced from that of $\operatorname{U}_{N}\hookrightarrow\operatorname{U}_{MN}$ defined in the beginning of Section~\ref{subsect:block_diag_unitary}.

There is also an alternative realization of the group $G_{\widehat{\mathbf{n}}}(\mathbb{T})$. Namely, $G_{\widehat{\mathbf{n}}}(\mathbb{T})$ can be identified with a group of certain unitary operators acting on a Hilbert space $L^2(X,\nu)$, where $X=[0,1)$ is equipped with the standard Lebesgue measure $\nu$. Indeed, for an element $\mathfrak{t}=(t_1,\ldots,t_N)\in\mathbb{T}^{N}$ one considers the following piecewise constant function $f_{\mathfrak{t}}\colon X\to\mathbb{T}$:
\begin{equation*}
f_{\mathfrak{t}}(x)=
\begin{cases}
t_1,& x\in[0,\frac{1}{N}),
\\
t_2,& x\in[\frac{1}{N},\frac{2}{N}),
\\
\ldots
\\
t_{N},& x\in[\frac{N-1}{N},1).
\end{cases}
\end{equation*}
Then, an element $\mathfrak{t}\in\mathbb{T}^{N}$ can be identified with the operator of multiplication by $f_{\mathfrak{t}}$ on $L^2(X,\nu)$ (this is a unitary operator since $|f_{\mathfrak{t}}|\equiv 1$). On the other hand, the symmetric group $\mathfrak{S}_{N}$ acts naturally on $X=[0,1)$ via the rearrangements of half-intervals $[\frac{k}{N},\frac{k+1}{N})$, $k=0,1,\ldots,N-1$. The induced action of $\mathfrak{S}_{N}$ on $L^2(X,\nu)$ realizes $\mathfrak{S}_N$ as a subgroup of unitary operators on $L^2(X,\nu)$. Thus, we obtain a realization of $G_N(\mathbb{T})$ in the unitary group of $L^2(X,\nu)$. It is not difficult to check that these realizations for different $N$ are compatible with respect to the block-diagonal embeddings defined above. In fact, this interpretation of $G_{\widehat{\mathbf{n}}}(\mathbb{T})$ is closely related to the construction with the product space $\mathbb{X}_{\widehat{\mathbf{n}}}$ discussed in previous section.

\subsection{Indecomposable characters of $G_{\widehat{\mathbf{n}}}(\mathbb{T})$ and associated spherical representations of the Gelfand pair $(G_{\widehat{\mathbf{n}}}(\mathbb{T}), \mathfrak{S}_{\widehat{\mathbf{n}}})$.}\label{sect:intro_indec_har_sph_rep}

Let $G$ be a topological group. Recall that a positive definite function $\chi\colon G\to\mathbb{C}$ is called a \emph{(normalized) character} of $G$ if the following properties hold:
\begin{itemize}
\item 
$\chi$ is {\it central}, i.e., $\chi(gh)=\chi(hg)$ for all $g,h\in G$;

\item 
$\chi(1_G)=1$, where $1_G$ is the identity element of $G$.
\end{itemize}
A character $\chi$ is called {\it indecomposable} if it cannot be expressed as $\alpha\chi_1+(1-\alpha)\chi_2$ for some $\alpha\in(0,1)$ and two distinct  characters $\chi_1$ and $\chi_2$. Equivalently, $\chi$ is indecomposable if and only if the corresponding \emph{GNS (Gelfand--Naimark--Segal) representation} $(\pi_\chi, \mathcal{H}_\chi,\xi_\chi)$ is a \emph{factor representation} (see Appendix \ref{sect:characters_spherical_func} for more details).

Consider the diagonal subgroup $\diag(G)=\{(g,g):g\in G\}$ of $G\times G$. It is well-known that under some suitable conditions on $G$ the pair $(G\times G,\diag(G))$ is a \emph{Gelfand pair} (see \cite{Borodin_Olshanski} or Appendix \ref{sect:appendix_spherical_func_approx} for precise definitions). The characters of the group $G$ are directly related to spherical functions of $(G\times G,\diag(G))$. Namely, we have the following general fact (for the proof we refer to Appendix \ref{sect:characters_spherical_func}).

\begin{Prop}\label{prop:char_spherical_double}
Let $\chi$ be an indecomposable character of the group $G$ and let $(\pi_\chi, \mathcal{H}_\chi,\xi_\chi)$ be the associated Gelfand--Naimark--Segal representation. Then, there 
exists an antilinear involutive isometry $S\colon\mathcal{H}_{\chi}\to\mathcal{H}_{\chi}$ such that $\pi_{\chi}(G)'=S\pi_{\chi}(G)''S$, where $\pi_{\chi}(G)''$ denotes the von 
Neumann algebra generated by the operators $\pi_{\chi}(g)$, $g\in G$, and $\pi_{\chi}(G)'$ denotes its commutant. Moreover, the formula 
$\pi_{\chi}^{(2)}(g_1,g_2)=\pi_{\chi}(g_1)S\pi_{\chi}(g_2)S$ defines in $\mathcal{H}_{\chi}$ an irreducible spherical representation of the Gefland pair $(G\times 
G,\diag(G))$ with spherical vector $\xi_{\chi}$. The spherical function of $\pi_{\chi}^{(2)}$ is given by the formula 
$\varphi(g_1,g_2)=\langle\pi_{\chi}^{(2)}(g_1,g_2)\xi_{\chi},\xi_{\chi}\rangle=\chi(g_1g_2^{-1})$.
\end{Prop}

The following proposition shows that indecomposable characters of the group $G_{\widehat{\mathbf{n}}}(\mathbb{T})$ are strongly related to spherical functions of $(G_{\widehat{\mathbf{n}}}(\mathbb{T}),\mathfrak{S}_{\widehat{\mathbf{n}}})$. However, we do not use this fact in the proof of main results, and we refer the interested reader to Appendix \ref{sect:characters_spherical_func} for the proof.

\begin{Prop}\label{prop:ind_char_and_spher_func}
Let $\chi$ be an indecomposable character on $G_{\widehat{\mathbf{n}}}(\mathbb{T})$ and let $(\pi_{\chi},\mathcal{H}_{\chi},\xi_{\chi})$ be the associated GNS representation.
Define an injective homomorphism $\Gamma\colon G_{\widehat{\mathbf{n}}}(\mathbb{T})\to G_{\widehat{\mathbf{n}}}(\mathbb{T})\times G_{\widehat{\mathbf{n}}}(\mathbb{T})$ by the formula $\Gamma(\mathfrak{t}s)=(\mathfrak{t}s,s)$ for $\mathfrak{t}\in\mathbb{T}^{\widehat{\mathbf{n}}}$ and $s\in\mathfrak{S}_{\widehat{\mathbf{n}}}$. 
Then, in the notation of Proposition \ref{prop:char_spherical_double}, the restriction of representation $\pi=\pi_{\chi}^{(2)}\circ\Gamma$ to the subspace\footnote{For a subset $D$ of vectors in a Hilbert space, we denote by $[D]$ the closed linear span of $D$.} $[\pi_{\chi}(\mathbb{T}^{\widehat{\mathbf{n}}})\xi]$ is an irreducible spherical representation of $G_{\widehat{\mathbf{n}}}(\mathbb{T})$. Moreover, the restrictions to $\mathbb{T}^{\widehat{\mathbf{n}}}$ of $\chi$ and the spherical function of $\pi$ coincide.
\end{Prop}

\begin{Rem}
This proposition together with the classification of spherical functions of $(G_{\widehat{\mathbf{n}}}(\mathbb{T}),\mathfrak{S}_{\widehat{\mathbf{n}}})$ (see below) already restricts the list of possible options for $\chi$. However, it is unclear whether this fact can be used for the effective classification of indecomposable characters of $G_{\widehat{\mathbf{n}}}(\mathbb{T})$.
\end{Rem}

\subsection{Main results.}
This paper can be regarded as a sequel to our previous work in \cite{Nessonov_Ngo}. There we gave a classification of indecomposable characters for the group $\mathfrak{S}_{\widehat{\mathbf{n}}}$~-- an analogue of infinite symmetric group for block-diagonal embeddings defined as the inductive limit of the sequence $\{\mathfrak{S}_{N_k}\}_{k\ge 1}$ with block-diagonal embeddings. Since the characters of $\mathfrak{S}_{\widehat{\mathbf{n}}}$ appear in the classification for\footnote{From now on the compact abelian group $A$ is assumed to be $\mathbb{T}$ unless stated otherwise.} $G_{\widehat{\mathbf{n}}}=G_{\widehat{\mathbf{n}}}(\mathbb{T})$ we recall the classification from \cite{Nessonov_Ngo} below.

As in \cite{Nessonov_Ngo}, denote by $\chi_{\nat}$ and $\sgn_{\infty}$ the \emph{natural} and the \emph{sign} characters of $\mathfrak{S}_{\widehat{\mathbf{n}}}$, respectively. These are defined as follows: for any $\sigma\in\mathfrak{S}_{N_k}$ we put
\begin{equation*}
\chi_{\nat}(\sigma)=\frac{1}{N_k}\#\Fix_{N_k}(\sigma),\quad\sgn_{\infty}(\sigma)=\lim_{k\to\infty}\sgn_{N_k}(\sigma).
\end{equation*}
Here $\Fix_{N_k}(\sigma)\subset\mathbb{X}_{N_k}=\{1,2,\ldots,N_k\}$ is the fixed-point set of permutation $\sigma\in\mathfrak{S}_{N_k}$, and $\sgn_{N_k}(\sigma)$ is the sign of the image of $\sigma$ in $\mathfrak{S}_{N_k}$.
Note that $\sgn_{\infty}$ is a non-trivial (multiplicative) character if and only if the sequence $\widehat{\mathbf{n}}$ has only finitely many even elements.
The main result of \cite{Nessonov_Ngo} can be stated as follows:
\begin{Prop}[{\cite[Theorem 1.2]{Nessonov_Ngo}}]\label{prop:inf_sym_indec_char}
The set $\exchar(\mathfrak{S}_{\widehat{\mathbf{n}}})$ of indecomposable characters of the group $\mathfrak{S}_{\widehat{\mathbf{n}}}$ is 
\begin{equation*}
\exchar(\mathfrak{S}_{\widehat{\mathbf{n}}})=\big\{\chi_{\nat}^{r},\,\sgn_{\infty}\cdot\chi_{\nat}^{r}:r\in\mathbb{Z}_{\ge 0}\cup\{\infty\}\big\}.
\end{equation*}
Here, the character $\chi_{\nat}^{\infty}$ is the regular character $\chi_{\reg}$ on $\mathfrak{S}_{\widehat{\mathbf{n}}}$.
\end{Prop}

To state the classification theorem for indecomposable characters of $G_{\widehat{\mathbf{n}}}$ we need to introduce more notation. For any $M\in\mathbb{Z}$ and $g=(\mathfrak{t},\sigma)\in G_{N_k}\subset G_{\widehat{\mathbf{n}}}$ define (see also \eqref{eq:phi_M_formula})
\begin{equation*}
\phi_M(g)=\frac{1}{N_k}\sum_{\substack{x\in\mathbb{X}_{N_k}\\\sigma(x)=x}}\mathfrak{t}(x)^{M}.
\end{equation*}
Note that for $M=0$ we get $\phi_0(g)=\chi_{\nat}(\sigma)$. Now we can state the main result of the present paper.

\begin{Th}[{Theorem \ref{thm:classification_char}}]\label{thm:main_th_intro}
Indecomposable characters of the infinite wreath product $G_{\widehat{\mathbf{n}}}$ are parameterized by a (possibly empty) unordered collection of non-zero integers $M_1,\ldots,M_l$ and an indecomposable character $\psi$ of $\mathfrak{S}_{\widehat{\mathbf{n}}}$.
More precisely, if $\chi$ is an indecomposable character on $G_{\widehat{\mathbf{n}}}$, then there exist $M_1,\ldots,M_l\in\mathbb{Z}\setminus\{0\}$ and $\psi\in\exchar(\mathfrak{S}_{\widehat{\mathbf{n}}})$ such that for any $g=(\mathfrak{t},\sigma)\in G_{\widehat{\mathbf{n}}}$ we have
\begin{equation*}
\chi(g)=\chi_{\mathbf{m},\psi}(g):=\psi(\sigma)\phi_{M_1}(g)\cdots\phi_{M_l}(g),
\end{equation*}
where $\mathbf{m}=(M_1,\ldots,M_l)$.
\end{Th}

\begin{Rem}
Using the realization of the group $G_{\widehat{\mathbf{n}}}$ inside $\operatorname{U}(L^2(X,\nu))$ (see Section \ref{subsubsect:group_realizations_torus}), we may regard an element $g=(f,\sigma)\in G_{\widehat{\mathbf{n}}}$ as a pair consisting of a function $f\colon X\to\mathbb{T}$ and a transformation $\sigma\colon X\to X$. Then, the formula above can be written as follows:
\begin{equation*}
\chi_{\mathbf{m},\psi}(g)=\psi(\sigma)\prod_{i=1}^{l}\left(\int_{\Fix(\sigma;X)}f(x)^{M_i}d\nu(x)\right).
\end{equation*}
In particular, if $\mathbf{m}$ is empty, then $\chi_{\mathbf{m},\psi}(g)=\psi(\sigma)$. The central elements of $G_{\widehat{\mathbf{n}}}$ correspond to constant functions $t\cdot\mathds{1}_{X}$, $t\in\mathbb{T}$ and the restriction of $\chi_{\mathbf{m},\psi}$ to the center is $\chi_{\mathbf{m},\psi}(t\cdot\mathds{1}_{X})=t^{M}$, where $M=M_1+\ldots+M_l$.
\end{Rem}

In the course of proving the classification of characters, we also describe all spherical functions of the generalized Gelfand pair $(G_{\widehat{\mathbf{n}}},\mathfrak{S}_{\widehat{\mathbf{n}}})$ (see Section \ref{sect:prelim}  and Appendix \ref{sect:appendix_spherical_func_approx} for precise definitions). 

\begin{Th}[Theorem \ref{thm:spherical_func}]\label{thm:sph_th_intro}
Spherical functions of the Gelfand pair $(G_{\widehat{\mathbf{n}}},\mathfrak{S}_{\widehat{\mathbf{n}}})$ are parameterized by unordered (possibly empty) collections of non-zero integers.
More precisely, for any spherical function $\varphi$ of $(G_{\widehat{\mathbf{n}}},\mathfrak{S}_{\widehat{\mathbf{n}}})$ there exist $M_1,\ldots,M_l\in\mathbb{Z}\setminus\{0\}$ such that for any $g=(\mathfrak{t},\sigma)\in G_{\widehat{\mathbf{n}}}$ we have
\begin{equation*}
\varphi(g)=\phi_{\mathbf{m}}(\mathfrak{t})=\phi_{M_1}(\mathfrak{t})\cdots\phi_{M_l}(\mathfrak{t}),
\end{equation*}
where  $\mathbf{m}=(M_1,\ldots,M_l)$.
\end{Th}

Finally, let us mention that another motivation for classifying the indecomposable characters of the group $G_{\widehat{\mathbf{n}}}$ was our desire to obtain a new proof of the classification of indecomposable characters of the infinite block-diagonal unitary group $\operatorname{U}_{\widehat{\mathbf{n}}}=\varinjlim\operatorname{U}_{N_i}$ (see \eqref{eq:ind_lim_unitary}). This problem was addressed previously by Boyer \cite{Boyer_1993} in the case $n_i=2$ for all $i$, and by Enomoto--Izumi \cite{Enomoto_Izumi} who consider a more general setting. We recover this classification as a consequence of our results in Section \ref{sect:ind_char_unitary} using the observation that $G_{\widehat{\mathbf{n}}}$ is the normalizer of $\mathbb{T}^{\widehat{\mathbf{n}}}$ in $\operatorname{U}_{\widehat{\mathbf{n}}}$.

\begin{Th}[{Theorem \ref{thm:indec_char_inf_unitary}, see also \cite[Proposition 1]{Boyer_1993}, \cite[Theorems 1.2 and 1.5]{Enomoto_Izumi}}]
Let $\tr$ be the unique tracial state on the inductive limit of matrix algebras $\varinjlim\operatorname{Mat}_{N_i}$ and regard $\operatorname{U}_{\widehat{\mathbf{n}}}=\varinjlim\operatorname{U}_{N_i}$ as the unitary subgroup of $\varinjlim\operatorname{Mat}_{N_i}$. Then, the indecomposable characters of the infinite block-diagonal unitary group $\operatorname{U}_{\widehat{\mathbf{n}}}$ are of the form $\tr^p\overline{\tr}^q$ for non-negative integers $p,q$.
\end{Th}

\subsection{Methods.} 
The classification of indecomposable characters for standard inductive limits of wreath products was achieved either by the character approximation (as in \cite{Boyer_2005} and \cite{Hirai_Hirai}) together with the careful study of the branching rules, or by using the operator-algebraic methods (as in \cite{Dudko_Nessonov}).
In our work we also use the character approximation (the Vershik--Kerov approach). However, we avoid the use of branching rules because for the block-diagonal embeddings the corresponding combinatorics seems to be rather involved. Instead, we analyze the converging sequences of finite-dimensional characters and calculate directly the limits for special elements of the group $G_{\widehat{\mathbf{n}}}(\mathbb{T})$. In the course of our computations we also classify weak limits of spherical functions of Gelfand pairs $(G_{N_k},\mathfrak{S}_{N_k})$ and, as a byproduct, obtain the list of spherical functions of $(G_{\widehat{\mathbf{n}}},\mathfrak{S}_{\widehat{\mathbf{n}}})$. The approximation theorem for characters and spherical functions holds in general settings and we give a proof of a version of it in the Appendix.

\subsection{Contents.} The remainder of the paper is organized as follows.

In Section 2 we discuss in detail the construction of the group $G_{\widehat{\mathbf{n}}}(\mathbb{T})$ and how it can be realized in the unitary group of the Hilbert space of a certain measure space $(\mathbb{X}_{\widehat{\mathbf{n}}},\nu_{\widehat{\mathbf{n}}})$.

In Section 3 we compute the possible weak limits of spherical functions of Gelfand pairs $(G_{N},\mathfrak{S}_{N})$. These computations allow us to classify all spherical functions of the generalized Gelfand pair $(G_{\widehat{\mathbf{n}}},\mathfrak{S}_{\widehat{\mathbf{n}}})$ and they are also used in the proof of the main theorem.

In Section 4 we give an overview of the representation theory of semidirect products of a compact abelian group and a finite group and, in particular, we obtain the character formulas for the group $G_N(\mathbb{T})$.

Section 5 is devoted to the main technical part of the proof of Theorem \ref{thm:classification_char}: we calculate the possible weak limits of irreducible characters of groups $G_{N_k}$ and thus obtain the list of all potential indecomposable characters.

In Section 6 we complete the proof of the classification of indecomposable characters of $G_{\widehat{\mathbf{n}}}$.

In Section 7 we discuss a construction of a $\mathrm{II}_{1}$ factor representation of $G_{\widehat{\mathbf{n}}}$ that realizes the character $\phi_1(g)$. Using this construction we show how the characters $\chi_{\mathbf{m},\psi}$ can be realized as matrix elements of certain unitary representations of $G_{\widehat{\mathbf{n}}}$.

In Section 8 we show how to derive from our main theorem the classification of indecomposable characters of $\operatorname{U}_{\widehat{\mathbf{n}}}=\varinjlim\operatorname{U}_{N_i}$ originally due to Boyer \cite{Boyer_1993}.

In Section 9 we give some speculations for the case when $A$ is an arbitrary compact abelian group. 
We hope to complete the details in the future work.

Finally, in the Appendix we prove several technical lemmas used in the proofs. In particular, we explain the calculation of limits of characters of symmetric groups~-- the argument is essentially a minor modification of the one from \cite[Section 5]{Nessonov_Ngo}, but we included it for completeness. Besides that, we record proofs of general approximation theorems for characters and spherical functions that apply to inductive limits of arbitrary compact groups (although these results seem to be known to experts, we were not able to find these facts in the form we needed in the literature.)

\subsection{Acknowledgements.} The authors were partially supported by the ``Long-term program of support of the Ukrainian
research teams at the Polish Academy of Sciences carried out in collaboration with the U.S.
National Academy of Sciences with the financial support of external partners''. 
The second author thanks Artem Dudko for discussions and hospitality during his visit to IM PAN in August 2025.
During the work on this project the second author
was also partially supported by the Austrian Science Fund (FWF) grant ``Geometry of the tip of the global
nilpotent cone'' no. 10.55776/P35847 and by the Austrian Academy of Sciences DOC Fellowship ``Big algebras in classical types''. For open access purposes, the authors have applied a CC BY public copyright license
to any author-accepted manuscript version arising from this submission.

\section{Preliminaries}\label{sect:prelim}

\subsection{Notation.}  Denote by $\mathfrak{S}_{N}$ the group of all permutations of the set $\mathbb{X}_{N}=\{1,2,\ldots,N\}$. Let $\mathcal{F}_{N}(\mathbb{Z})$ and $\mathcal{F}_{N}(\mathbb{T})$ be the abelian groups of functions on $\mathbb{X}_{N}$ with values in $\mathbb{Z}$ and $\mathbb{T}=\{z\in\mathbb{C}:|z|=1\}$, respectively (the group operation is defined pointwise). Equip the multiplicative group $\mathcal{F}_{N}(\mathbb{T})$ with $N$-dimensional torus topology.

Note that groups $\mathcal{F}_{N}(\mathbb{T})\simeq\mathbb{T}^N$ and $\mathcal{F}_{N}(\mathbb{Z})\simeq\mathbb{Z}^N$ are \emph{Pontryagin dual} to each other. Namely, any $\mathfrak{m}\in\mathcal{F}_{N}(\mathbb{Z})$ determines a multiplicative character $R_{\mathfrak{m}}$ on $\mathcal{F}_{N}(\mathbb{T})$:
\begin{equation*}
R_{\mathfrak{m}}(\mathfrak{t})=\prod_{x\in\mathbb{X}_{N}}\mathfrak{t}(x)^{\mathfrak{m}(x)},~\mathfrak{t}\in\mathcal{F}_{N}(\mathbb{T}).
\end{equation*}
\subsection{The group $G_N$ and the Gelfand pair $(G_N,\mathfrak{S}_N)$.}
The natural action of $\mathfrak{S}_{N}$ on $\mathbb{X}_{N}$ induces a natural $\mathfrak{S}_{N}$-action on $\mathcal{F}_{N}(\mathbb{T})$ and $\mathcal{F}_{N}(\mathbb{Z})$: for any $\sigma\in\mathfrak{S}_{N}$, $\mathfrak{t}\in\mathcal{F}_{N}(\mathbb{T})$ and $\mathfrak{m}\in\mathcal{F}_{N}(\mathbb{Z})$ define
\begin{equation*}
\sigma\cdot\mathfrak{t}:=\mathfrak{t}\circ\sigma^{-1},~\sigma\cdot\mathfrak{m}:=\mathfrak{m}\circ\sigma^{-1}.
\end{equation*}
Denote by $G_{N}=\mathcal{F}_{N}(\mathbb{T})\rtimes\mathfrak{S}_{N}$ the semidirect product of $\mathcal{F}_{N}(\mathbb{T})$ and $\mathfrak{S}_{N}$. More explicitly, the group $G_{N}$ is defined as follows:
\begin{itemize}
\item
as a topological space, $G_{N}$ is the Cartesian product $\mathcal{F}_{N}(\mathbb{T})\times\mathfrak{S}_{N}$,

\item
the group operation is defined via the formula $(\mathfrak{t}_1,\sigma_1)\cdot(\mathfrak{t}_2,\sigma_2):=(\mathfrak{t}_1(\sigma_1\cdot\mathfrak{t}_2),\sigma_1\sigma_2)$.
\end{itemize}
Alternatively, one can regard $G_N$ as the group of $N\times N$ matrices which have exactly one non-zero element in each row and each column and whose non-zero entries belong to $\mathbb{T}$.

We identify $\mathfrak{S}_{N}$ and $\mathcal{F}_{N}(\mathbb{T})$ with the subgroups $\{(1_\mathbb{T},\sigma):\sigma\in\mathfrak{S}_{N}\}$ and $\{(\mathfrak{t},1_{\mathfrak{S}_N}):\mathfrak{t}\in\mathcal{F}_{N}(\mathbb{T})\}$ of the group $G_{N}$.
Denote by $\mathbf{C}^*[G_N]$ the group $\mathbf{C}^*$-algebra of $G_N$. Set $E=\frac{1}{N!}\sum\limits_{s\in\mathfrak{S}_N}s$. It is not difficult to see that $E$ is a self-adjoint idempotent in $\mathbf{C}^*[G_N]$.

\begin{Prop}\label{prop:g_n_sym_n_gelfand_pair}
The $*$-algebra $E\,\mathbf{C}^*[G_N]\,E$ is abelian, i.e., the pair $(G_{N},\mathfrak{S}_{N})$ is a Gelfand pair.
\end{Prop}
\begin{proof}
Since any element of $G_{N}$ can be decomposed as a product $f\sigma$ with $f\in\mathcal{F}_{N}(\mathbb{T})$ and $\sigma\in\mathfrak{S}_{N}$, to prove the proposition it suffices to show that for any $g_1,g_2\in\mathcal{F}_N(\mathbb{T})$ the elements $Eg_1E$ and $Eg_2E$ commute, i.e., the equality
\begin{equation*}
Eg_1Eg_2E=Eg_2Eg_1E ~\text{ holds}.
\end{equation*}
Let us first introduce the elements $\mathcal{R}(g_i)=\frac{1}{N!}\sum\limits_{s\in\mathfrak{S}_N}sg_is^{-1}$. Then,
\begin{equation*}\label{sect:equalities_E}
s\mathcal{R}(g_i)=\mathcal{R}(g_i)s ~\text{for all}~ s\in\mathfrak{S}_N ~\text{and}~ Eg_iE=E\mathcal{R}(g_i)E=E\mathcal{R}(g_i)=\mathcal{R}(g_i)E.
\end{equation*}
Since $g_1,g_2\in\mathcal{F}_N(\mathbb{T})$, it follows from the definition of the group $G_N$ that
\begin{equation*}\label{sect: commute}
\mathcal{R}(g_1)\mathcal{R}(g_2)=\mathcal{R}(g_2)\mathcal{R}(g_1).
\end{equation*}
Now we obtain from \eqref{sect:equalities_E} that
\begin{equation*}
Eg_1Eg_2E=E\mathcal{R}(g_1)E\cdot E\mathcal{R}(g_2)E=E\mathcal{R}(g_1)\mathcal{R}(g_2)E\stackrel{\eqref{sect: commute}}{=}E\mathcal{R}(g_2)\mathcal{R}(g_1)E=Eg_2Eg_1E,
\end{equation*}
which finishes the proof of the proposition.
\end{proof}

Let $\Pi$ be an irreducible unitary representation of the group $G_N$ in a Hilbert space $\mathcal{H}$. Recall that $\Pi$ is called $\mathfrak{S}_{N}$-\emph{spherical} if the subspace $\mathcal{H}^{\mathfrak{S}_N}=\left\{\eta\in\mathcal{H}:\Pi(s)\eta=\eta~~\text{for all}~s\in\mathfrak{S}_N \right\}$ is one-dimensional. In this case the function $\varphi$ on $G_N$ defined by $\varphi(g)=\langle\Pi(g)\xi,\xi\rangle_{\mathcal{H}}$, where $\xi\in\mathcal{H}^{\mathfrak{S}_N}$ and $\|\xi\|=1$, is called a {\it spherical function}.
\begin{Prop}
$\mathfrak{S}_{N}$-spherical functions of the Gelfand pair $(G_N,\mathfrak{S}_{N})$ are of the form
\begin{equation}\label{spherical_finite}
\varphi_{\mathfrak{m}}(\mathfrak{t})=\frac{1}{N!}\sum_{\sigma\in\mathfrak{S}_{N}}R_{\sigma\cdot\mathfrak{m}}(\mathfrak{t})=\frac{1}{N!}\sum_{\sigma\in\mathfrak{S}_{N}}\prod_{x\in\mathbb{X}_{N}}\mathfrak{t}(x)^{\mathfrak{m}(\sigma(x))}=\frac{1}{N!}\sum_{\sigma\in\mathfrak{S}_{N}}\prod_{x\in\mathbb{X}_{N}}\mathfrak{t}(\sigma(x))^{\mathfrak{m}(x)}.
\end{equation}
\end{Prop}

\begin{proof}
An irreducible representation of $G_N$ with a $\mathfrak{S}_N$-spherical vector is isomorphic to a representation of the form $\Ind_{\mathcal{F}_N(\mathbb{T})}^{G_N}(R_{\mathfrak{m}}\otimes\triv_{\mathfrak{S}_N})$, where $R_{\mathfrak{m}}$, $\mathfrak{m}\in\mathcal{F}_N(\mathbb{Z})$, is a one-dimensional representation of $\mathcal{F}_N(\mathbb{T})\simeq\mathbb{T}^N$, and $\triv_{\mathfrak{S}_n}$ is the trivial representation of $\mathfrak{S}_N$ (see also Section \ref{sect:irred_char_wreath_prod}). The formula for the spherical function follows.
\end{proof}

\subsection{Diagonal embeddings and inductive limit.}\label{diagonal_emb_X}

For any positive integers $N$ and $M$ define the map  $\mathfrak{p}_{_{MN,M}}\colon\mathbb{X}_{MN}\to\mathbb{X}_{M}$ as follows:
\begin{equation*}
\mathfrak{p}_{_{MN,M}}(Mq+r)=r,~\text{where}~0\le q\le N-1,~1\le r\le M.
\end{equation*}
For any $s\in\mathfrak{S}_M$ define $\mathfrak{i}_{_{M,MN}}(s)\in\mathfrak{S}_{MN}$ as
\begin{equation*}
(\mathfrak{i}_{_{M,MN}}(s))(Mq+r)=Mq+s(r).
\end{equation*}
The map $\mathfrak{i}_{_{M,MN}}\colon\mathfrak{S}_M\to\mathfrak{S}_{MN}$ is called the \emph{diagonal embedding} of the group $\mathfrak{S}_M$ into the group $\mathfrak{S}_{MN}$.
\begin{Rem}\label{rem:x_mn_as_product}
Note that if one identifies the set $\mathbb{X}_{M}\times\mathbb{X}_{N}$ with $\mathbb{X}_{MN}$ via the map $(x,y)\mapsto x+M(y-1)$, then $\mathfrak{p}_{_{MN,M}}\colon\mathbb{X}_{M}\times\mathbb{X}_{N}\to\mathbb{X}_{M}$ is the projection on the first coordinate.
\end{Rem}

Fix an arbitrary sequence $\widehat{\mathbf{n}}=\{n_{j}\}_{j=1}^{\infty}$ of positive integers that are greater than $1$. Set $N_k=n_1n_2\ldots n_k$. We use notation $\mathfrak{p}^{(k)}:=\mathfrak{p}_{_{N_{k+1},N_k}}$ and $\mathfrak{i}^{(k)}:=\mathfrak{i}_{_{N_{k},N_{k+1}}}$ for brevity.
Consider the sequence of natural maps
\begin{equation*}
\ldots\stackrel{\mathfrak{p}^{(k)}}{\longrightarrow}\mathbb{X}_{N_k}\stackrel{\mathfrak{p}^{(k-1)}}{\longrightarrow}\mathbb{X}_{N_{k-1}}
\stackrel{\mathfrak{p}^{(k-2)}}{\longrightarrow}\ldots\stackrel{\mathfrak{p}^{(1)}}{\longrightarrow}\mathbb{X}_{N_1}.
\end{equation*}
It gives the chain of the corresponding dual maps
\begin{equation*}
\mathcal{F}_{N_1}(\mathbb{T})\stackrel{\iota^{(1)}}{\longrightarrow}\ldots\stackrel{\iota^{(k-1)}}{\longrightarrow}\mathcal{F}_{N_{k-1}}(\mathbb{T})\stackrel{\iota^{(k-1)}}{\longrightarrow}\mathcal{F}_{N_k}(\mathbb{T})\stackrel{\iota^{(k+1)}}{\longrightarrow}\ldots,
\end{equation*}
where $(\iota^{(k)}f)(x)=f(\mathfrak{p}^{(k)}(x))$ for any $f\in\mathcal{F}_{N_k}(\mathbb{T})$ and $x\in\mathbb{X}_{N_{k+1}}$.
On the other hand, we also have the following sequence of embeddings of symmetric groups:
\begin{equation*}
\mathfrak{S}_{N_1}\stackrel{\mathfrak{i}^{(1)}}{\longrightarrow}\ldots
\stackrel{\mathfrak{i}^{(k-2)}}{\longrightarrow}\mathfrak{S}_{N_{k-1}}
\stackrel{\mathfrak{i}^{(k-1)}}{\longrightarrow}
\mathfrak{S}_{N_k}\stackrel{\mathfrak{i}^{(k)}}{\longrightarrow}\ldots.
\end{equation*}
The maps above define the chain of the embeddings
\begin{equation*}
G_{N_1}\to\ldots\to G_{N_{k-1}}\to G_{N_k}\to\ldots.
\end{equation*}
We denote by $\mathcal{F}_{\widehat{\mathbf{n}}}(\mathbb{T})$, $\mathfrak{S}_{\widehat{\mathbf{n}}}$ and $G_{\widehat{\mathbf{n}}}$ the inductive limits of the corresponding chains with these diagonal embeddings. Further we identify $\mathcal{F}_{N_k}(\mathbb{T})$, $\mathfrak{S}_{N_k}$ and $G_{N_k}$ with their images in $\mathcal{F}_{\widehat{\mathbf{n}}}(\mathbb{T})$, $\mathfrak{S}_{\widehat{\mathbf{n}}}$ and $G_{\widehat{\mathbf{n}}}$, respectively. Note that $\mathcal{F}_{\widehat{\mathbf{n}}}(\mathbb{T})$, $\mathfrak{S}_{\widehat{\mathbf{n}}}$ and $G_{\widehat{\mathbf{n}}}$ are topological groups with the inductive limit topology. In particular, $\mathfrak{S}_{\widehat{\mathbf{n}}}$ is a discrete group.

\subsection{Action of $\mathfrak{S}_{\widehat{\mathbf{n}}}$ on $(\mathbb{X}_{\widehat{\mathbf{n}}},\nu_{\widehat{\mathbf{n}}})$.}\label{sect:measure_space_interpret}
For each $m$ endow the finite set $\mathbb{X}_m$ with the uniform probability measure $\nu_m$. Let $(\mathbb{X}_{\widehat{\mathbf{n}}},\nu_{\widehat{\mathbf{n}}})$ be the product measure space of spaces $(\mathbb{X}_{n_k},\nu_{n_k})$.
Then, the group $\mathfrak{S}_{\widehat{\mathbf{n}}}$ naturally acts on $\mathbb{X}_{\widehat{\mathbf{n}}}$ and can be viewed as a certain subgroup of measure preserving transformations of the space $(\mathbb{X}_{\widehat{\mathbf{n}}},\nu_{\widehat{\mathbf{n}}})$. Indeed, we can regard the space $\mathbb{X}_{\widehat{\mathbf{n}}}=\prod\limits_{j=1}^{\infty}\mathbb{X}_{n_j}$ as the space of sequences $x=(x_1,x_2,\ldots)$, where $x_k\in\mathbb{X}_{n_k}$ for each $k$. Besides that, we identify $\mathbb{X}_{n_1}\times\ldots\times\mathbb{X}_{n_k}$ with $\mathbb{X}_{n_1n_2\ldots n_k}=\mathbb{X}_{N_k}$ (see Remark \ref{rem:x_mn_as_product}).
Under this identifications, an element $\sigma\in\mathfrak{S}_{N_r}\subset\mathfrak{S}_{\widehat{\mathbf{n}}}$ acts on $\mathbb{X}_{\widehat{\mathbf{n}}}$ as follows:
\begin{equation}\label{eq:sym_group_action}
\sigma((y,z))=(\sigma(y),z),~y\in\mathbb{X}_{N_r},~z\in\prod_{j=r+1}^{\infty}\mathbb{X}_{n_j}.
\end{equation}
See also Section \ref{sect:realizations} and \cite[Section 2]{Nessonov_Ngo} for more details about this construction.

In a similar way the group $\mathcal{F}_{\widehat{\mathbf{n}}}(\mathbb{T})$ can be identified with a subgroup of the group of all measurable functions $f\colon\mathbb{X}_{\widehat{\mathbf{n}}}\to\mathbb{T}$. The elements of the group $\mathcal{F}_{\widehat{\mathbf{n}}}(\mathbb{T})$ naturally act on the Hilbert space $L^2(\mathbb{X}_{\widehat{\mathbf{n}}},\nu_{\widehat{\mathbf{n}}})$ as multiplication operators. Together with the induced action of $\mathfrak{S}_{\widehat{\mathbf{n}}}$ on $L^2(\mathbb{X}_{\widehat{\mathbf{n}}},\nu_{\widehat{\mathbf{n}}})$, this yields a realization of $G_{\widehat{\mathbf{n}}}$ inside $U(L^2(\mathbb{X}_{\widehat{\mathbf{n}}},\nu_{\widehat{\mathbf{n}}}))$. In fact, this construction is basically equivalent to the one from Section \ref{subsubsect:group_realizations_torus}.

\section{Limits of spherical functions of $(G_{N_k},\mathfrak{S}_{N_k})$}\label{sect:spherical_func}

In this section we describe all weak (pointwise) limits of spherical functions of Gelfand pairs $(G_{N_k},\mathfrak{S}_{N_k})$. This computation will be used to classify all spherical functions of the Gelfand pair $(G_{\widehat{\mathbf{n}}},\mathfrak{S}_{\widehat{\mathbf{n}}})$ and later in a similar calculation for the characters of $G_{\widehat{\mathbf{n}}}$.

\subsection{Spherical representations.}
Consider a unitary representation $\Pi$ of the group $G_{\widehat{\mathbf{n}}}$ which acts in a Hilbert space $\mathcal{H}$. Recall that $\Pi$ is called $\mathfrak{S}_{\widehat{\mathbf{n}}}$-\emph{spherical} if representation $\Pi$ is irreducible and the subspace
\begin{equation*}
\mathcal{H}^{\mathfrak{S}_{\widehat{\mathbf{n}}}}=\left\{\eta\in\mathcal{H}: \Pi(s)\eta=\eta~\text{for all}~s\in\mathfrak{S}_{\widehat{\mathbf{n}}}\right\}
\end{equation*}
is one-dimensional. In this case the corresponding \emph{spherical function} $\varphi\colon G_{\widehat{\mathbf{n}}}\to\mathbb{C}$ is defined as
\begin{equation*}
\varphi(g)=\langle\Pi(g)\eta,\eta\rangle_{\mathcal{H}},~\text{where}~\eta\in\mathcal{H}^{\mathfrak{S}_{\widehat{\mathbf{n}}}}~\text{and}~\|\eta\|=1.
\end{equation*}
The spherical function $\varphi$ is $\mathfrak{S}_{\widehat{\mathbf{n}}}$-biinvariant, i.e., the equality $\varphi(s_1gs_2)=\varphi(g)$ holds for all $s_1,s_2\in \mathfrak{S}_{\widehat{\mathbf{n}}}$ and $g\in G_{\widehat{\mathbf{n}}}$. In particular, for any $g=(\mathfrak{t},s)\in G_{\widehat{\mathbf{n}}}$, where $\mathfrak{t}\in \mathcal{F}_{\widehat{\mathbf{n}}}(\mathbb{T})$ and $s\in \mathfrak{S}_{\widehat{\mathbf{n}}}$, we have $\varphi(g)=\varphi(\mathfrak{t})$.

The pair $(G_{\widehat{\mathbf{n}}},\mathfrak{S}_{\widehat{\mathbf{n}}})$ is a Gelfand pair as the inductive limit of Gelfand pairs $(G_{N_k},\mathfrak{S}_{N_k})$ (see \cite[Proposition 8.15]{Borodin_Olshanski} or Appendix \ref{sect:appendix_spherical_func_approx}). The classification of the corresponding spherical functions is given in the theorem below.

\begin{Th}\label{thm:spherical_func}
Spherical functions of $(G_{\widehat{\mathbf{n}}},\mathfrak{S}_{\widehat{\mathbf{n}}})$ are of the form
\begin{equation*}
\varphi(g)=\prod_{j=1}^{l}\left(\frac{1}{N_k}\sum_{x\in\mathbb{X}_{N_k}}\mathfrak{t}(x)^{M_j}\right),~g=(\mathfrak{t},\sigma)\in G_{N_k},~k\ge 1
\end{equation*}
for some (possible empty) collection of non-zero integers $M_1$, \ldots, $M_l$.
\end{Th}

We prove this classification using the approximation technique. Namely, the following proposition shows that any indecomposable spherical function of $(G_{\widehat{\mathbf{n}}},\mathfrak{S}_{\widehat{\mathbf{n}}})$ can be approximated by a sequence of spherical functions of $(G_{N_k},\mathfrak{S}_{N_k})$.

\begin{Prop}\label{prop:sph_func_approx}
Let $\Pi$ be an $\mathfrak{S}_{\widehat{\mathbf{n}}}$-spherical representation of the group $G_{\widehat{\mathbf{n}}}$ and let $\varphi$ be the corresponding spherical function.
Then, there exists an increasing sequence $\{k_l\}_{l=1}^{\infty}$ of positive integers and a sequence $\{ \mathfrak{m}_{k_l}\in\mathcal{F}_{N_{k_l}}(\mathbb{Z})\}_{l=1}^{\infty}$ such that $\varphi$ is the weak (pointwise) limit of the sequence of spherical functions $\{\varphi_{\mathfrak{m}_{k_l}}\}_{l=1}^{\infty}$ of groups $\{G_{N_{k_l}}\}_{k=1}^{\infty}$.

\end{Prop}
\begin{proof}
This is a consequence of Proposition \ref{prop:g_n_sym_n_gelfand_pair} and Proposition \ref{prop:general_approx_sph_func}.
\end{proof}

Therefore, Theorem \ref{thm:spherical_func} is reduced to analyzing all possible weak limits of sequences of spherical functions of $(G_{N_k},\mathfrak{S}_{N_k})$. The description of such weak limits (see Proposition \ref{prop:sph_func_conv_seq}) occupies the most part of Section~\ref{sect:spherical_func}.

\subsection{Auxiliary lemmas.} In this subsection we prove several technical estimates which are used in the proof of Proposition \ref{prop:sph_func_conv_seq}.

\begin{Lm}
For any non-zero real number $K$ and any $\delta>0$ the following inequality holds:
\begin{equation}\label{cos_integral}
\frac{1}{\delta}\int_{0}^{\delta}\lvert\cos K\theta\rvert\, d\theta\le\frac{2}{\pi}+\frac{1}{|K|\delta}.
\end{equation}
\end{Lm}
\begin{proof}
This follows from the direct calculation of the integral on the left.
\end{proof}

\begin{Lm}\label{t_powers_conv}
Let $\{q_k\}_{k=1}^{\infty}$ be a sequence of integers such that the sequence of functions $\{t^{q_k}\}_{k=1}^{\infty}$ of functions on $\mathbb{T}$ is pointwise convergent. Then, the sequence $\{q_k\}_{k=1}^{\infty}$ is eventually constant, i.e., there exists a positive integer $L$ such that $q_k=q_L$ for all $k>L$.
\end{Lm}
\begin{proof}
Let $\mu$ be the probability Haar measure on $\mathbb{T}$.
Denote by $f$ the pointwise limit of functions $\{t^{q_k}\}_{k=1}^{\infty}$ on $\mathbb{T}$. Assume that the sequence $\{q_k\}_{k=1}^{\infty}$ is not stationary. Then, there exist two subsequences $\{k_i\}_{i=1}^{\infty}$ and $\{k_{i}'\}_{i=1}^{
\infty}$ such that $q_{k_i}\neq q_{k_i'}$ for all $i$. Consider functions $f_i(t)=t^{q_{k_i}}$ and $g_i(t)=t^{q_{k_i'}}$ for each $i\ge 1$. Then, clearly both sequences of functions $\{f_i(t)\overline{g_i(t)}\}_{i=1}^{\infty}$ and $\{|f_i(t)|^2\}_{i=1}^{\infty}$ converge pointwise to $|f(t)|^2$. In this case, the Lebesgue's dominated convergence theorem implies
\begin{equation*}
\begin{split}
&\int_{\mathbb{T}}|f|^2d\mu(t)=\lim_{i\to\infty}\int_{\mathbb{T}}f_i(t)\overline{g_i(t)}d\mu(t)=\lim_{i\to\infty}\int_{\mathbb{T}}t^{q_{k_i}-q_{k_i'}}d\mu(t)=0,~\text{and}
\\
&\int_{\mathbb{T}}|f(t)|^2d\mu(t)=\lim_{i\to\infty}\int_{\mathbb{T}}|f_i(t)|^2d\mu(t)=\lim_{i\to\infty}\int_{\mathbb{T}}1\,d\mu(t)=1,
\end{split}
\end{equation*}
which is a contradiction. Therefore, $\{q_k\}_{k=1}^{\infty}$ must be eventually constant.
\end{proof}

\begin{Def}
For any positive integers $N$ and $K$ such that $1\le K\le N$ and any real number $\theta$ define an element $\mathfrak{t}^{\theta}_{N,K}\in\mathcal{F}_N(\mathbb{T})$ as follows:
\begin{equation}\label{t_N_K_function}
\mathfrak{t}^\theta_{N,K}(x)=
\begin{cases}
e^{i\theta},&~x\in\{1,\ldots,K\}\subset \mathbb{X}_N,
\\
1,&~x\in\{K+1,\ldots,N\}\subset \mathbb{X}_N.
\end{cases}
\end{equation}
\end{Def}

\begin{Lm}\label{spher_func_t_element}
Let $\mathfrak{m}\in \mathcal{F}_N(\mathbb{Z})$. Then, for $t=e^{i\theta}$ we have
\begin{equation*}
\varphi_{\mathfrak{m}}(\mathfrak{t}^{\theta}_{N,K})=\binom{N}{K}^{-1}\sum_{A\in\mathcal{S}(\mathbb{X}_{N},K)}
t^{^{\left({\sum\limits_{x\in A}\mathfrak{m}(x)}\right)}},
\end{equation*}
where $\mathcal{S}(\mathbb{X}_{N},K)$ is the set of all $K$-element subsets of $\mathbb{X}_{N}$.
In particular, $\varphi_{\mathfrak{m}}(\mathfrak{t}_{N,K})$ equals to the coefficient of $z^K$ term of the product
\begin{equation*}
\binom{N}{K}^{-1}\prod_{x\in\mathbb{X}_{N}}\left(1+zt^{\mathfrak{m}(x)}\right).
\end{equation*}
\end{Lm}

\begin{proof}
By definition, we have
\begin{equation*}
\begin{aligned}
\varphi_{\mathfrak{m}}(\mathfrak{t}^{\theta}_{N,K})
&=\frac{1}{N!}\sum_{\sigma\in\mathfrak{S}_{N}}\prod_{x\in\mathbb{X}_{N}}\mathfrak{t}^{\theta}_{N,K}(x)^{\mathfrak{m}(\sigma(x))}=
\frac{1}{N!}\sum_{\sigma\in\mathfrak{S}_{N}}\exp\left\{i\theta\cdot\sum_{x=1}^{K}\mathfrak{m}(\sigma(x))\right\}=
\\
&=\frac{K!\,(N-K)!}{N!}\cdot\sum_{A\in\mathcal{S}(\mathbb{X}_{N},K)}\exp\left\{i\theta\cdot\sum_{x\in A}\mathfrak{m}(x)\right\}=\binom{N}{K}^{-1}\sum_{A\in\mathcal{S}(\mathbb{X}_{N},K)}\exp\left\{i\theta\cdot\sum_{x\in A}\mathfrak{m}(x)\right\},
\end{aligned}
\end{equation*}
where we used that there are exactly $K!\cdot(N-K)!$ permutations in $\mathfrak{S}_{N}$ with prescribed set $\{\sigma(1),\ldots,\sigma(K)\}$. The second assertion of the lemma follows now from the formula above and the Vieta's formulas.
\end{proof}

\begin{Lm}\label{main_ineq_lemma}
Let $\mathfrak{m}\in\mathcal{F}_{N}(\mathbb{Z})$ and let positive integers $N$, $K$ and $L$ be such that $1\le K\le N$. Suppose that the inequality
$
\min\{\#\{x\in\mathbb{X}_{N}:\mathfrak{m}(x)=m\},\#\{x\in\mathbb{X}_{N}:\mathfrak{m}(x)=m'\}\}\ge L
$ holds for some distinct $m,m'\in\mathbb{Z}$.
If $t=e^{i\theta}\in\mathbb{T}$, then
\begin{equation*}
|\varphi_{\mathfrak{m}}(\mathfrak{t}^{\theta}_{N,K})|\le\binom{N}{K}^{-1}\sum_{\substack{l_1,l_2\ge 0,\\l_1+2l_2\le 2L}}\binom{L}{l_1,l_2,L-l_1-l_2}\binom{N-2L}{K-l_1-2l_2}\left|2\cos\frac{(m-m')\theta}{2}\right|^{l_1}.
\end{equation*}
\end{Lm}

\begin{proof}
In the proof we will use the following simple observation:

\medskip

\noindent
\textbf{Claim.}
{\em
Let $p(z)=p_mz^m+\ldots+p_0$ and $q(z)=q_nz^n+\ldots+q_0$ be two polynomials with complex coefficients. Then, for any $d$ the absolute value of the coefficient of $z^d$ term of $p(z)q(z)$ does not exceed the coefficient of $z^d$ term of the polynomial
$(|p_m|z^m+\ldots+|p_0|)(|q_n|z^n+\ldots+|q_0|)$.
}

\medskip

Choose any two $L$-element subsets $X$ and $X'$ in $\mathbb{X}_{N}$ such that for all $x\in X$ and $x'\in X'$ we have $\mathfrak{m}(x)=m$ and $\mathfrak{m}(x')=m'$. Then, by Lemma \ref{spher_func_t_element}, $\varphi_{\mathfrak{m}}(\mathfrak{t}^{\theta}_{N,K})$ is the coefficient of $z^{K}$ term of the product
\begin{equation*}
\begin{aligned}
\binom{N}{K}^{-1}\prod_{x\in\mathbb{X}_{N}}\left(1+t^{\mathfrak{m}(x)}z\right)
&=\binom{N}{K}^{-1}(1+t^{m}z)^{L}(1+t^{m'}z)^{L}\prod_{x\in\mathbb{X}_{N}\setminus (X\cup X')}(1+t^{\mathfrak{m}(x)}z)=
\\
&=\binom{N}{K}^{-1}(1+(t^{m}+t^{m'})z+t^{m+m'}z^2)^{L}\prod_{x\in\mathbb{X}_{N}\setminus (X\cup X')}(1+t^{\mathfrak{m}(x)}z).
\end{aligned}
\end{equation*}
The claim above implies that $|\varphi_{\mathfrak{m}}(\mathfrak{t}^{\theta}_{N,K})|$ does not exceed the coefficient $c_K$ of $z^K$ term of the polynomial
\begin{equation*}
\binom{N}{K}^{-1}(1+z\lvert t^{m}+t^{m'}\rvert+z^2)^{L}(1+z)^{N-2L}.
\end{equation*}
Since $\lvert t^{m}+t^{m'}\rvert=\lvert e^{im\theta}+e^{im'\theta}\rvert=2\lvert\cos\frac{1}{2}(m-m')\theta\rvert$, we have
\begin{equation*}
c_K=\binom{N}{K}^{-1}\sum_{\substack{l_1,l_2\ge 0,\\l_1+2l_2\le 2L}}\binom{L}{l_1,l_2,L-l_1-l_2}\binom{N-2L}{K-l_1-2l_2}\left\lvert2\cos\frac{(m-m')\theta}{2}\right\rvert^{l_1}.
\end{equation*}
This completes the proof of the lemma.
\end{proof}

\begin{Lm}\label{lim_ineq_lemma}
Let a sequence $\{\mathfrak{m}_k\}_{k=1}^{\infty}$, where $\mathfrak{m}_k\in\mathcal{F}_{N_k}(\mathbb{Z})$ for each $k$, be such that the sequence $\{\varphi_{\mathfrak{m}_k}\}_{k=1}^{\infty}$ of functions on $\mathcal{F}_{\widehat{\mathbf{n}}}(\mathbb{T})$ weakly converges, i.e., $\lim\limits_{k\to\infty}\varphi_{\mathfrak{m}_k}(f)$ exists for all $f\in \mathcal{F}_{\widehat{\mathbf{n}}}(\mathbb{T})=\bigcup\limits_{k}\mathcal{F}_{N_k}(\mathbb{T})$. Set $\varphi =\lim\limits_{k\to\infty}\varphi_{\mathfrak{m}_k}$. Fix an arbitrary positive integer $K_l<N_l$. If there exists a subsequence $\left\{ k_j\right\}_{j=0}^\infty$ such that $l<k_j<k_{j+1}$ for all $j$ and the inequality
\begin{equation}\label{m_occurences}
\min\{\#\{x\in\mathbb{X}_{N_{k_j}}:\mathfrak{m}_{k_j}(x)=m\},\#\{x\in\mathbb{X}_{N_{k_j}}:\mathfrak{m}_{k_j}(x)=m'\}\}\ge L ~\text{holds for all}~j=0,1,2,\ldots,
\end{equation}
then
$|\varphi(\mathfrak{t}^{\theta}_{N_{l},K_{l}})|\le\left(\alpha^2+2\alpha(1-\alpha)\left|\cos\frac{(m-m')\theta}{2}\right|+(1-\alpha)^2\right)^{L}$, where $\alpha=\frac{K_{l}}{N_{l}}$.
\end{Lm}
\begin{proof}
Fix $j$. Considering   $\mathfrak{t}^{\theta}_{N_{l},K_{l}}\in \mathcal{F}_{N_{l}}(\mathbb{T})$ as an element of the group $ \mathcal{F}_{N_{k_j}}(\mathbb{T})$, where $l<k_j$, we have
\begin{equation}
\#\left\{ x\in \mathbb{X}_{N_{k_j}}: \mathfrak{t}^{\theta}_{N_{l},K_{l}}(x)=e^{i\theta}\right\}=K_{l}\frac{N_{k_j}}{N_{l}}=\alpha N_{k_j}.
\end{equation}
It follows from Lemma \ref{main_ineq_lemma}  that
\begin{equation*}
|\varphi_{\mathfrak{m}_{k_j}}(\mathfrak{t}^{\theta}_{N_{l},K_{l}})|\le \binom{N_{k_j}}{\alpha N_{k_j}}^{-1}\sum_{\substack{l_1,l_2\ge 0,\\l_1+2l_2\le 2L}}\binom{L}{l_1,l_2,L-l_1-l_2}\binom{N_{k_j}-2L}{\alpha N_{k_j}-l_1-2l_2}\left|2\cos\frac{(m-m')\theta}{2}\right|^{l_1}.
\end{equation*}
By assumption, as $j\to\infty$ the left hand-side converges to $|\varphi(\mathfrak{t}^{\theta}_{N_{l},K_{l}})|$. For the right-hand side note that
\begin{multline*}
\lim_{j\to\infty}\binom{N_{k_j}}{\alpha N_{k_j}}^{-1}\sum_{\substack{l_1,l_2\ge 0,\\l_1+2l_2\le 2L}}\binom{L}{l_1,l_2,L-l_1-l_2}\binom{N_{k_j}-2L}{\alpha N_{k_j}-l_1-2l_2}\left|2\cos\frac{(m-m')\theta}{2}\right|^{l_1}=
\\
=\sum_{\substack{l_1,l_2\ge 0,\\l_1+2l_2\le 2L}}\binom{L}{l_1,l_2,L-l_1-l_2}\alpha^{l_1+2l_2}(1-\alpha)^{2L-l_1-2l_2}\left|2\cos\frac{(m-m')\theta}{2}\right|^{l_1}=
\\
=\left(\alpha^2+2\alpha(1-\alpha)\left|\cos\frac{(m-m')\theta}{2}\right|+(1-\alpha)^2\right)^{L},
\end{multline*}
and the lemma follows.
\end{proof}

\subsection{Classification of weakly converging sequences of spherical functions.}

Next lemma essentially describes the limiting spherical function for a particular choice of the sequence $\{\mathfrak{m}_k\}_{k=1}^{\infty}$ in Proposition \ref{prop:sph_func_approx}.

\begin{Lm}\label{lem:lim_sph_func_formula}
For arbitrary integers $M_1,\ldots,M_l$ define for all sufficiently large $k$ the element $\mathfrak{m}_k\in\mathcal{F}_{N_k}(\mathbb{Z})$ as
\begin{equation*}
\mathfrak{m}_k(x)=
\begin{cases}
M_x,&~x\in\{1,2,\ldots,l\}\subset\mathbb{X}_{N_k},
\\
0,&~x\in\{l+1,\ldots,N_k\}.
\end{cases}
\end{equation*}
Fix any $\mathfrak{t}\in\mathcal{F}_{N_r}(\mathbb{T})\subset\mathcal{F}_{\widehat{\mathbf{n}}}(\mathbb{T})$. Then,
\begin{equation*}
\lim_{k\to\infty}\varphi_{\mathfrak{m}_k}(\mathfrak{t})=\prod_{j=1}^{l}\int_{\mathbb{X}_{\widehat{\mathbf{n}}}}\mathfrak{t}(x)^{M_j}d\nu_{\widehat{\mathbf{n}}}(x).
\end{equation*}
\end{Lm}

\begin{proof}
For all $k>r$ we have
\begin{equation}\label{phi_m_t_formula}
\begin{aligned}
\varphi_{\mathfrak{m}_k}(\mathfrak{t})
&=\frac{1}{N_k!}\sum_{\sigma\in\mathfrak{S}_{N_k}}\prod_{x\in\mathbb{X}_{N_k}}\mathfrak{t}(\sigma(x))^{\mathfrak{m}_k(x)}
=\frac{1}{N_k!}\sum_{\sigma\in\mathfrak{S}_{N_k}}\mathfrak{t}(\sigma(1))^{M_1}\ldots\mathfrak{t}(\sigma(l))^{M_l}=
\\
&=\frac{(N_k-l)!}{N_k!}\sum_{\substack{\text{distinct}\\x_1,\ldots,x_l\in\mathbb{X}_{N_k}}}\mathfrak{t}(x_1)^{M_1}\ldots
\mathfrak{t}(x_l)^{M_l},
\end{aligned}
\end{equation}where the last summation is over all $l$-tuples $(x_1,\ldots,x_l)$ of distinct elements in $\mathbb{X}_{N_k}$.
 Let $i\in\mathbb{X}_{N_r}$ and $t_i=\mathfrak{t}(i)$, where $r<k$. Then,
\begin{equation*}
\varphi_{\mathfrak{m}_k}(\mathfrak{t})=\frac{(N_k-l)!}{N_k!}
\sum_{I=(i_1,\ldots,i_l)\in\mathbb{X}_{N_r}^l}\frac{(N_k/N_r)!}{(N_k/N_r-a_1(I))!}
\ldots\frac{(N_k/N_r)!}{(N_k/N_r-a_{N_r}(I))!}\,t_{i_1}^{M_1}\ldots t_{i_l}^{M_l},
\end{equation*}
where for a given $l$-tuple $I=(i_1,\ldots,i_l)$ we put $a_j(I)=\#\{s:i_s=j\}$ for all $j\in\mathbb{X}_{N_r}$ (note that $a_1(I)+\ldots+a_{N_r}(I)=l$). Indeed, formula for the coefficients follow from the fact that the multiset $\{\mathfrak{t}(x):x\in\mathbb{X}_{N_k}\}$ contains $N_k/N_r$ copies of each $t_i$, $i\in\mathbb{X}_{N_r}$.
As $k\to\infty$ the expression above converges to
\begin{equation*}
\lim_{k\to\infty}\varphi_{\mathfrak{m}_k}(\mathfrak{t}_k)=\frac{1}{N_r^l}\sum_{i_1,\ldots,i_l\in\mathbb{X}_{N_r}}t_{i_1}^{M_1}\ldots t_{i_l}^{M_l}=\prod_{j=1}^{l}\Bigg(\frac{1}{N_r}\sum_{x\in\mathbb{X}_{N_r}}t_x^{M_j}\Bigg)=\prod_{j=1}^{l}\int_{\mathbb{X}_{\widehat{\mathbf{n}}}}
\mathfrak{t}(x)^{M_j}d\nu_{\widehat{\mathbf{n}}}(x),
\end{equation*}
as claimed.
\end{proof}

The main goal of this subsection is to describe all weakly converging subsequences of spherical functions.

\begin{Prop}\label{prop:sph_func_conv_seq}
Consider a sequence $\{\mathfrak{m}_k\}_{k=1}^{\infty}$, where $\mathfrak{m}_k\in\mathcal{F}_{N_k}(\mathbb{Z})$. Denote by $\varphi_{\mathfrak{m}_k}$ the spherical function on $G_{N_k}$ that corresponds to $\mathfrak{m}_k\in\mathcal{F}_{N_k}(\mathbb{Z})$ (see \eqref{spherical_finite}).
Suppose that for any $f\in\mathcal{F}_{\widehat{\mathbf{n}}}(\mathbb{T})=\bigcup\limits_k\mathcal{F}_{N_k}(\mathbb{T})$ there exists $\lim\limits_{k\to \infty} \varphi_{\mathfrak{m}_k}(f)=\varphi(f)$.  If the function $\varphi\colon \mathcal{F}_{\widehat{\mathbf{n}}}(\mathbb{T})\to\mathbb{C}$ is continuous, then there exists a finite collection of non-zero integers $M_1,\ldots,M_l$ such that for all sufficiently large $k$ there exists an element $\sigma\in\mathfrak{S}_{N_k}$ satisfying
\begin{equation*}
\mathfrak{m}_k(\sigma(1))=M_1,\, \mathfrak{m}_k(\sigma(2))=M_2,\,\ldots,\,\mathfrak{m}_{k}(\sigma(l))=M_l,\,\ldots,\, \mathfrak{m}_{k}(\sigma(l+1))=0, \ldots,\,\mathfrak{m}_{k}(\sigma(N_k))=0.
\end{equation*}
\end{Prop}

\begin{proof}
By assumption $\varphi$ is continuous and it is clear that $\varphi(1)=1$.
We divide the proof of the proposition into several steps.

\medskip

\noindent
\textbf{Step 1.} \emph{There exists a constant $S\in\mathbb{Z}$ such that
\begin{equation}\label{estimate_of_sum_m_k}
\sum_{x\in\mathbb{X}_{N_k}}\mathfrak{m}_k(x)=S
\end{equation}
for all sufficiently large $k$.}

\medskip

Indeed, if $f_\theta(x) =e^{i\theta}$ for all $x\in \mathbb{X}_{N_k}$, then the sequence $\varphi_{\mathfrak{m}_k}(f_\theta)$ converges to $\varphi(f_\theta)$ for all $\theta$ as $k\to\infty$. However, by Lemma \ref{spher_func_t_element} we have
\begin{equation*}
\varphi_{\mathfrak{m}_k}(f_\theta)=\exp\left\{i\theta\cdot\sum_{x\in\mathbb{X}_{N_k}}\mathfrak{m}_k(x)\right\}
\end{equation*}
Applying Lemma \ref{t_powers_conv} yields that the sequence $\{\sum_{x\in\mathbb{X}_{N_k}}\mathfrak{m}_k(x)\}_{k=1}^{\infty}$ is eventually constant.

\medskip

\noindent
\textbf{Step 2.} \emph{There exists a constant $D>0$ such that
\begin{equation}\label{estimate_m-m'}
\sup_{k\ge 1}\max_{x,x'\in\mathbb{X}_{N_k}}|\mathfrak{m}_k(x)-\mathfrak{m}_k(x')|<D.
\end{equation}}

\medskip

Take any $k$ and any distinct $m,m'\in\mathfrak{m}_k(\mathbb{X}_{N_k})$ and choose an arbitrary positive integer $K_1\in(0,N_1)$. Applying Lemma \ref{main_ineq_lemma} for $N=N_k$, $K=N_k\cdot K_1/N_1$ and $\mathfrak{m}_k\in\mathcal{F}_{N_k}(\mathbb{Z})$ and $L=1$, we obtain the inequality\footnote{Recall that we identify an element $\mathfrak{t}^\theta_{N_1,K_1}\in\mathcal{F}_{N_1}(\mathbb{T})$ with the element $\mathfrak{t}^\theta_{N_k,K_1N_k/N_1}\in\mathcal{F}_{N_k}(\mathbb{T})$.}
\begin{equation*}
\begin{split}
|\varphi_{\mathfrak{m}_k}(\mathfrak{t}^\theta_{N_1,K_1})|
&=|\varphi_{\mathfrak{m}_k}(\mathfrak{t}^\theta_{N_k,K_1N_k/N_1})|\le
\\
&\le\binom{N_k}{K_1N_k/N_1}^{-1}\left(\binom{N_k-2}{K_1N_k/N_1}+\binom{N_k-2}{K_1N_k/N_1-1}\left\lvert2\cos\frac{(m-m')\theta}{2}\right\rvert+\binom{N_k-2}{K_1N_k/N_1-2}\right),
\end{split}
\end{equation*}
where $\mathfrak{t}^\theta_{N_1,K_1}$ is defined by \eqref{t_N_K_function} with $t=e^{i\theta}$.
Integrating the inequality above over $\theta\in[0,\delta]$ for a positive $\delta>0$ and using the estimate from Lemma \ref{cos_integral} yields
\begin{equation*}
\begin{split}
\frac{1}{\delta}\int_{0}^{\delta}|\varphi_{\mathfrak{m}_k}(\mathfrak{t}^\theta_{N_1,K_1})|d\theta
&\le\frac{\binom{N_k-2}{K_1N_k/N_1}}{\binom{N_k}{K_1N_k/N_1}}+\frac{\binom{N_k-2}{K_1N_k/N_1-2}}{\binom{N_k}{K_1N_k/N_1}}+\frac{\binom{N_k-2}{K_1N_k/N_1-1}}{\binom{N_k}{K_1N_k/N_1}}\cdot\frac{1}{\delta}\int_{0}^{\delta}\left\lvert\cos\frac{(m-m')\theta}{2}\right\rvert d\theta\le
\\
&\le\frac{\binom{N_k-2}{K_1N_k/N_1}}{\binom{N_k}{K_1N_k/N_1}}+\frac{\binom{N_k-2}{K_1N_k/N_1-2}}{\binom{N_k}{K_1N_k/N_1}}+\frac{\binom{N_k-2}{K_1N_k/N_1-1}}{\binom{N_k}{K_1N_k/N_1}}\cdot\left(\frac{2}{\pi}+\frac{1}{|m-m'|\delta}\right).
\end{split}
\end{equation*}
Put $\alpha=K_1/N_1\in(0,1)$.
Now we use the fact that $m$ and $m'$ were arbitrary distinct elements of $\mathfrak{m}_k(\mathbb{X}_{N_k})$ and take the lower limits of both sides as $k\to\infty$. We obtain the following inequality (a similar calculation is done in the proof of Lemma \ref{lim_ineq_lemma}):
\begin{equation*}
\liminf_{k\to\infty}\frac{1}{\delta}\int_{0}^{\delta}|\varphi_{\mathfrak{m}_k}(\mathfrak{t}^\theta_{N_1,K_1})|d\theta\le\alpha^2+(1-\alpha)^2+2\alpha(1-\alpha)\Bigg(\frac{2}{\pi}+\frac{1}{\delta}\cdot\frac{1}{\limsup\limits_{k\to\infty}\max\limits_{x,x'\in\mathbb{X}_{N_k}}|\mathfrak{m}_k(x)-\mathfrak{m}_k(x')|}\Bigg).
\end{equation*}
On the other hand, since $\{\varphi_{\mathfrak{m}_k}\}_{k}$ converges pointwise to $\varphi$, the dominated convergence theorem implies that
\begin{equation*}
\liminf_{k\to\infty}\frac{1}{\delta}\int_{0}^{\delta}|\varphi_{\mathfrak{m}_k}(\mathfrak{t}^\theta_{N_1,K_1})|d\theta=\lim_{k\to\infty}\frac{1}{\delta}\int_{0}^{\delta}|\varphi_{\mathfrak{m}_k}(\mathfrak{t}^\theta_{N_1,K_1})|d\theta=\frac{1}{\delta}\int_{0}^{\delta}|\varphi(\mathfrak{t}^\theta_{N_1,K_1})|d\theta.
\end{equation*}
Consequently, if we assume that
\begin{equation}\label{eq:sup_infinite}
\sup_{k\ge 1}\max_{x,x'\in\mathbb{X}_{N_k}}|\mathfrak{m}_k(x)-\mathfrak{m}_k(x')|=+\infty
\end{equation}
then for any $\delta>0$ we have
\begin{equation}\label{eq:inequality_to_infinity}
\frac{1}{\delta}\int_{0}^{\delta}|\varphi(\mathfrak{t}^\theta_{N_1,K_1})|d\theta\le\alpha^2+(1-\alpha)^2+2\alpha(1-\alpha)\cdot\frac{2}{\pi}=1-2\alpha(1-\alpha)(1-2/\pi).
\end{equation}
Hence, taking into account  the continuity of $\varphi$, we obtain the following contradiction:
\begin{equation}
1=\varphi(1)=\lim_{\delta\to 0}\frac{1}{\delta}\int_{0}^{\delta}|\varphi(\mathfrak{t}^\theta_{N_1,K_1})|d\theta=1-2\alpha(1-\alpha)(1-2/\pi)<1.
\end{equation}
Thus, the assumption \eqref{eq:sup_infinite} is false, and the inequality \eqref{estimate_m-m'} holds.

\medskip

\noindent
\textbf{Step 3.}
\emph{The union $\bigcup_{k\ge 1}\mathfrak{m}_k(\mathbb{X}_{N_k})$ is a finite set.}

\medskip

Applying \eqref{estimate_of_sum_m_k} and \eqref{estimate_m-m'}, we get
\begin{equation*}
|\mathfrak{m}_k(x)|\le\left|\mathfrak{m}_k(x)-\frac{S}{N_k}\right|+\frac{S}{N_k}=\frac{1}{N_k}\left|\sum_{x'\in\mathbb{X}_{N_k}}(\mathfrak{m}_k(x)-\mathfrak{m}_k(x'))\right|+\frac{S}{N_k}\le D+\frac{S}{N_k}.
\end{equation*}
Therefore, $\sup_{x\in\mathbb{X}_{N_k},k\ge 1}|\mathfrak{m}_k(x)|$ is finite.

\medskip

\noindent
\textbf{Step 4.} \emph{There are no distinct integers $m$ and $m'$ such that
\begin{equation*}
\limsup_{k\to\infty}\min\{\#\{x\in\mathbb{X}_{N_k}:\mathfrak{m}_k(x)=m\},\#\{x\in\mathbb{X}_{N_k}:\mathfrak{m}_k(x)=m'\}\}=+\infty.
\end{equation*}
}

\medskip

Indeed, assume that such $m$ and $m'$ exist. Then, there exists an increasing sequence $\{k_j\}_{j=0}^\infty$ of positive integers such that
\begin{equation*}
L_j=\min\{\#\{x\in\mathbb{X}_{N_{k_j}}:\mathfrak{m}_{k_j}(x)=m\},\#\{x\in\mathbb{X}_{N_{k_j}}:\mathfrak{m}_{k_j}(x)=m'\}\}\to\infty ~\text{when}~j\to\infty.
\end{equation*}
Let $k_0=1$ and choose any positive integer $K_1$ with $0<K_1<N_1$. Then, the condition \eqref{m_occurences} in Lemma \ref{lim_ineq_lemma} holds for any positive integer $L$. Therefore, for any $L\ge 1$ we have
\begin{equation}\label{L_infinity}
|\varphi(\mathfrak{t}^{\theta}_{N_1,K_1})|\le\left(\alpha^2+2\alpha(1-\alpha)\left|\cos\frac{(m-m')\theta}{2}\right|+(1-\alpha)^2\right)^{L},
\end{equation}
where $\alpha=K_1/N_1\in(0,1)$. Now note that if $\cos\frac{1}{2}(m-m')\theta\neq \pm 1$, then
\begin{equation}
\alpha^2+2\alpha(1-\alpha)\left|\cos\frac{(m-m')\theta}{2}\right|+(1-\alpha)^2<\alpha^2+2\alpha(1-\alpha)+(1-\alpha)^2=1.
\end{equation}
Thus, for any $\theta$ with $\cos\frac{1}{2}(m-m')\theta\neq \pm 1$ we obtain from \eqref{L_infinity} that $\varphi(\mathfrak{t}^{\theta}_{N_1,K_1})=0$. But this contradicts the continuity of $\varphi(\mathfrak{t}^{\theta}_{N_1,K_1})$ and the fact that $\varphi(\mathfrak{t}^{\theta}_{N_1,K_1})|_{\theta=0}=1$.

\medskip

\noindent
\textbf{Step 5.} \label{M_1, M_2_ldots_constants}\emph{There exist a finite collection of non-zero integers $M_1$, \ldots, $M_l$ and a subsequence $\{k_i\}_{i=1}^{\infty}$ such that for all $i$ the sequence
\begin{equation*}
(\mathfrak{m}_{k_i}(1),\ldots,\mathfrak{m}_{k_i}(N_{k_i}))
\end{equation*}
is a permutation of the sequence $(M_1,\ldots,M_l,0\ldots,0)$.
}

\medskip

By Step 3, there exists a finite subset $V\subset\mathbb{Z}$ such that $\mathfrak{m}_k(\mathbb{X}_{N_k})\subset V$ for all $k$. Set
$\mathcal{C}_k(m)=\{x\in\mathbb{X}_{N_k}:\mathfrak{m}_k(x)=m\}$.

Observe that for every $k$ we have $\sum_{m\in V}\#\mathcal{C}_k(m)=\#\mathbb{X}_{N_k}=N_k$. Since $V$ is finite and $\lim\limits_{k\to\infty}N_k=\infty$, there exists $m_0\in V$ and a subsequence $\{k_i\}_{i=1}^{\infty}$ such that for all $k$
\begin{equation*}
\lim_{i\to\infty}\#\mathcal{C}_{k_i}(m_0)=\infty.
\end{equation*}
Without loss of generality we may assume that $k_i=i$ for all $i$. Let us prove that $m_0=0$. Indeed, by Step 4, there exists a constant $C>0$ such that
\begin{equation*}
\sum_{m\in V, m\neq m_0}\#\mathcal{C}_{k}(m)<C.
\end{equation*}
In particular, $\mathcal{C}_k(m_0)\ge N_k-C$ for all $k$. Since $V$ is finite we obtain
\begin{equation*}
\lim_{k\to\infty}\frac{1}{N_k}\sum_{x\in\mathbb{X}_{N_k}}\mathfrak{m}_k(x)=m_0.
\end{equation*}
On the other hand, by Step 1, we have
\begin{equation*}
\lim_{k\to\infty}\frac{1}{N_k}\sum_{x\in\mathbb{X}_{N_k}}\mathfrak{m}_k(x)=\lim_{k\to\infty}\frac{S}{N_k}=0,
\end{equation*}
hence, $m_0=0$.

Thus, for all $k$ the sequence $\{\mathfrak{m}_k(x)\}_{x\in\mathbb{X}_{N_k}}$ contains at most $C$ non-zero elements. The statement now follows from the fact that each $\mathfrak{m}_k(x)$ belongs to the finite set $V$.

\medskip

\noindent
\textbf{Step 6.} \emph{The statement of Step 5 holds for all large $k$, i.e., there exist a finite collection of non-zero integers $M_1,\ldots,M_l$ such that for all large $k$ the sequence
\begin{equation*}
(\mathfrak{m}_k(1),\ldots,\mathfrak{m}_{k}(N_k))
\end{equation*}
is a permutation of the sequence $(M_1,\ldots,M_l,0,\ldots,0)$.}

\medskip

Observe the argument of Step 5 applies to any subsequence of the original sequence $\{N_{k}\}_{k=1}^{
\infty}$. The statement now follows from Lemma \ref{lem:lim_sph_func_formula} and the fact that different multisets $(M_1,\ldots,M_l)$ produce different limiting functions in Lemma \ref{lem:lim_sph_func_formula}.
\end{proof}

\subsection{Classification of spherical functions.} Now we can complete the classification of spherical functions of the Gelfand pair $(G_{\widehat{\mathbf{n}}},\mathfrak{S}_{\widehat{\mathbf{n}}})$.

\begin{proof}[Proof of Theorem \ref{thm:spherical_func}]
Using Proposition \ref{prop:sph_func_conv_seq} and Proposition \ref{prop:sph_func_approx}, we obtain that any spherical function of the Gelfand pair $(G_{\widehat{\mathbf{n}}},\mathfrak{S}_{\widehat{\mathbf{n}}})$ must be of the form indicated in Theorem \ref{thm:spherical_func}. Thus, it remains to show that each of these functions is indeed an indecomposable spherical function. The indecomposability is a direct consequence of Lemma \ref{lem:lin_independence_phi_func} or Steps 2 and 3 of the proof of Lemma \ref{lem:indecomp_main}.
\end{proof}

\section{Irreducible representations and characters of $\mathbb{T}^N\rtimes\mathfrak{S}_{N}$}\label{sect:irred_char_wreath_prod}

In this section we give explicit formulas for irreducible characters of the compact group $G_{N}\simeq\mathbb{T}^{N}\rtimes\mathfrak{S}_{N}$.

\subsection{General facts.}
Let us recall a few facts from the representation theory of semidirect products $G=A\rtimes H$, where $A$ is a compact abelian group and $H$ is a finite group.
Denote by $\psi(h)$ the corresponding action of $h\in H$ on $A$. In other words, $\psi\colon H\to\Aut(A)$ is a group homomorphism and the group structure on $A\rtimes H$ (which is just $A\times H$ as a set) is defined as follows:
\begin{equation*}
(a_1,h_1)\cdot(a_2,h_2)=(a_1\psi(h_1)(a_2),h_1h_2).
\end{equation*}

All irreducible representations of $A$ are one-dimensional and parameterized by multiplicative characters (i.e., group homomorphisms) $\alpha\colon A\to\mathbb{T}$. Multiplicative characters on $A$ form an abelian group called the \emph{Pontryagin dual} group, denoted by $\widehat{A}$. The action $\psi$ of $H$ on $A$ induces the  action $\widehat{\psi}$ of $H$ on the dual group $\widehat{A}$:
\begin{equation*}
\widehat{\psi}(h)(\alpha)(a)=\alpha(\psi(h^{-1})(a)).
\end{equation*}
Let $1_H$ be the identity of $H$, and let  $1_A$ be the identity of $A$.
We identify subgroups $\{(a,1_H):a\in A\}$ and $\{(1_A,h):h\in H\}$ of $A\rtimes H$ with $A$ and $H$, respectively. Note that $A$ is a normal subgroup of $A\rtimes H$ and the action of $\psi(h)$ is given by the conjugation by $h$:
\begin{equation*}
\psi(h)(a)=hah^{-1},~a\in A.
\end{equation*}

Now let us recall the description of irreducible unitary representations of the group $G=A\rtimes H$. Fix any $\alpha\in\widehat{A}$ and let $\Stab(\alpha)=\left\{h\in H: \widehat{\psi}(h)(\alpha)=\alpha\right\}$ be its stabilizer in $H$. For a unitary irreducible representation $\rho$ of $\Stab(\alpha)$ in Hilbert space $\mathcal{H}_\rho$ define representation $\alpha\otimes\rho$ of the subgroup $A\times \Stab(\alpha)\subset G$ by the formula
\begin{equation*}
(\alpha\otimes\rho)(ah)\,\eta=\alpha(a)\rho(h)\,\eta,~\text{where}~\eta\in\mathcal{H}_\rho, a\in A,h\in\Stab(\alpha).
\end{equation*}
Denote by $\mathfrak{C}_\alpha=H/\Stab(\alpha)$ the space of left $\Stab(\alpha)$-cosets in $H$. Then, it is known (see \cite[Ch. 17, Theorem 4]{Barut_Raczka} or \cite[\S 13.3, Theorem 1]{Kirillov}) that any irreducible unitary representation $\pi$ of $A\rtimes H$ is equivalent to a representation of the form
\begin{equation*}
\pi=\Ind_{A\times\Stab(\alpha)}^{A\rtimes H}\alpha\otimes\rho.
\end{equation*}
Note that the induced representation $\pi$ has dimension $\#\mathfrak{C}_\alpha\cdot\dim\rho$. Fix a left transversal $\mathfrak{i}\colon\mathfrak{C}_\alpha\to H$, i.e., an injective map such that $\mathfrak{i}(c)\in c$ for each left coset $c\in\mathfrak{C}_{\alpha}$.
The induced character formula \cite[\S 13.6, Theorem 1]{Kirillov} implies that the normalized character $\chi_\pi$ of the  representation $\pi$ is given by
\begin{equation}\label{char_semidirect_pr}
\chi_{\pi}(g)=\frac{1}{\#\mathfrak{C}_\alpha}\sum_{\substack{c\in\mathfrak{C}_\alpha\\\mathfrak{i}(c)^{-1}\,g\, \mathfrak{i}(c)\in A\times\Stab(\alpha)}}(\alpha\otimes\chi_{\rho})(\mathfrak{i}(c)^{-1}\,g\,\mathfrak{i}(c)),~g\in A\rtimes H,
\end{equation}
where $\chi_{\rho}$ is the normalized character of $\rho$.

\begin{Rem}
The character $\chi_{\rho}$ is normalized in a sense that $\chi_{\rho}(1_{G})=1$, i.e., we define $\chi_{\rho}(g)=\frac{\Tr\rho(g)}{\Tr\rho(e)}$. For $g=a\in A$ the formula above simplifies to
\begin{equation*}
\chi_{\pi}(a)=\frac{1}{\#\mathfrak{C}_\alpha}\sum_{\substack{c\in\mathfrak{C}_\alpha}}\alpha(\mathfrak{i}(c)^{-1}\,a\,\mathfrak{i}(c)).
\end{equation*}
\end{Rem}

\subsection{Case of $G_{N}\simeq\mathbb{T}^{N}\rtimes\mathfrak{S}_{N}$.} Let us apply the results of the previous subsection for the group $G_{N}=\mathcal{F}_{N}(\mathbb{T})\rtimes\mathfrak{S}_{N}$. In this case the group $H=\mathfrak{S}_{N}$ acts on an abelian group $A=\mathcal{F}_{N}(\mathbb{T})\simeq\mathbb{T}^{N}$ by permutations.

\subsubsection{Description of irreducible representations.}
Note that for any $\mathfrak{m}\in\mathcal{F}_{N}(\mathbb{Z})\simeq\widehat{\mathbb{T}^{N}}$ the stabilizer subgroup $\Stab(\mathfrak{m})$ is a \emph{Young subgroup}, i.e., it is isomorphic to a direct product of symmetric groups. More precisely, if we define
\begin{equation*}
Y_j(\mathfrak{m})=\mathfrak{m}^{-1}(\{j\})=\left\{x\in \mathbb{X}_{N}:\mathfrak{m}(x)=j\right\},
\end{equation*}
then $\mathbb{X}_{N}=\bigsqcup_{j\in\mathbb{Z}}Y_j(\mathfrak{m})$ and
\begin{equation*}
\Stab(\mathfrak{m})=\prod_{j\in\mathbb{Z}}\mathfrak{S}_{Y_{j}(\mathfrak{m})},~\text{where}~\mathfrak{S}_{Y_{j}(\mathfrak{m})}~\text{is the group of all permutations of}~Y_{j}(\mathfrak{m}).
\end{equation*}
Since for all but finitely many $j\in\mathbb{Z}$ the set $Y_{j}(\mathfrak{m})$ is empty, the product on the right is actually finite. It is clear that each element $s\in\Stab(\mathfrak{m})$ is a product $s=\prod\limits_js_j$, where $s_j\in\mathfrak{S}_{Y_{j}(\mathfrak{m})}$. Define  homomorphism $\mathfrak{p}_j$ from $\Stab(\mathfrak{m})$ onto $\mathfrak{S}_{Y_{j}(\mathfrak{m})}$ by $\mathfrak{p}_j(s)=s_j$.

The irreducible representations of the symmetric group $\mathfrak{S}_{d_j}$, where $d_j=\#Y_j(\mathfrak{m})$, are parameterized by partitions $\lambda(j)\vdash d_j$. Denote by $\rho_{\lambda(j)}$ the irreducible representation of $\mathfrak{S}_{d_j}$ which corresponds to $\lambda(j)$.

Let $\pi$ be any irreducible representation of the group $G_{N}$. Then, there exist $\mathfrak{m}\in\mathcal{F}_{N}(\mathbb{Z})$ and an irreducible representation $\rho$ of the group $\Stab(\mathfrak{m})$ in Hilbert space $\mathcal{H}_\rho$ such that
\begin{equation*}
\pi\simeq\Ind_{\mathcal{F}_{N}(\mathbb{T})\times\Stab(\mathfrak{m})}^{G_{N}}(R_{\mathfrak{m}}\otimes\rho),
\end{equation*}
where $R_{\mathfrak{m}}$ is one-dimensional representation of the abelian group $\mathcal{F}_{N}(\mathbb{T})$ of the form $R_{\mathfrak{m}}(f)= \prod\limits_{x\in \mathbb{X}_N} f(x)^{\mathfrak{m}(x)}$, $f\in \mathcal{F}_{N}(\mathbb{T})$.

Since each irreducible representations $\rho$ of $\Stab(\mathfrak{m})$ is the tensor product of irreducible representations of $\mathfrak{S}_{Y_{j}(\mathfrak{m})}$, i.e., $\rho=\bigotimes_{j\in\mathbb{Z}}\rho_{\lambda(j)}$ $j\in\mathbb{Z}$, where  $\rho_{\lambda(j)}$ is an irreducible representation of $\mathfrak{S}_{Y_{j}(\mathfrak{m})}=\mathfrak{S}_{d_j}$, we obtain that for any $s\in \Stab(\mathfrak{m})$ we have
\begin{equation*}
\chi_{\rho}(s)=\prod_{j\in\mathbb{Z}}\chi_{\rho_{\lambda(j)}}(s_j),~s_j=\mathfrak{p}_j(s).
\end{equation*}
Denote by $\lambda^\rho$ the sequence $\{\lambda(j)\}_{j\in\mathbb{Z}}$  of partitions $\lambda(j)\vdash d_j$.

\subsubsection{Character formula.}

Put $\mathfrak{C}_\mathfrak{m}=\mathfrak{S}_{N}/\Stab(\mathfrak{m})$ and fix a corresponding left transversal $\mathfrak{i}\colon\mathfrak{C}_\mathfrak{m}\to \mathfrak{S}_{N}$.
Let $g=(\mathfrak{t},\sigma)$ be an element of $G_{N}=\mathcal{F}_{N}(\mathbb{T})\rtimes\mathfrak{S}_{N}$. Applying the character formula \eqref{char_semidirect_pr}, we obtain that
\begin{equation}\label{Frob_formula}
\chi_{\pi}(\mathfrak{t}\sigma)=\frac{1}{\#\mathfrak{C}_\mathfrak{m}}\sum_{{\substack{c\in\mathfrak{C}_\mathfrak{m}\\ \mathfrak{i}(c)^{-1}\sigma\,\mathfrak{i}(c)\in \Stab(\mathfrak{m})}}}R_{\mathfrak{m}}(\mathfrak{t}\circ\, {\mathfrak{i}(c)^{-1}} )\prod_{j\in\mathbb{Z}}\chi_{\lambda(j)}\left(\mathfrak{p}_j\left(\mathfrak{i}(c)^{-1}\sigma\,\mathfrak{i}(c)\right)\right),
\end{equation}
where $\chi_{\lambda(j)}$ is the normalized character of the representation $\rho_{\lambda(j)}$. (See also \cite[Theorem 4.5(ii)]{Hirai_Hirai_Hora_1}.)

Denote by $\mathcal{A}(\sigma,\mathfrak{m})$ the set of partitions $\{A_j\}_{j\in\mathbb{Z}}$ of $\mathbb{X}_N$ satisfying the following properties:
\begin{equation*}
\#A_j=d_j,~\sigma(A_j)=A_j ~\text{for all}~j\in\mathbb{Z}.
\end{equation*}
Introduce the subset $\mathfrak{C}_{\mathfrak{m},\sigma}=\left\{c\in\mathfrak{C}_\mathfrak{m}:\mathfrak{i}(c)^{-1}\sigma\,\mathfrak{i}(c)\in\Stab(\mathfrak{m}) \right\}$.
Since the map $\mathfrak{C}_{\mathfrak{m},\sigma}\ni c\mapsto \left\{ \mathfrak{i}(c) (Y_j(\mathfrak{m}))\right\}_{j\in\mathbb{Z}}\in \mathcal{A}(\sigma, \mathfrak{m})$ is a bijection, we obtain from \eqref{Frob_formula} that
\begin{equation*}
\chi_{\pi}(g)=\frac{1}{\#\mathfrak{C}_\mathfrak{m}}\sum_{\{A_j\}_{j\in\mathbb{Z}}\in\mathcal{A}(\sigma, \mathfrak{m})}\prod_{j\in\mathbb{Z}}\prod_{x\in A_j}\mathfrak{t}(x)^{j}\times\prod_{j\in\mathbb{Z}}\chi_{\lambda(j)}(\mathfrak{p}_j(\sigma)),
\end{equation*}

Let $c(\sigma)=(c_1(\sigma),c_2(\sigma),\ldots)$ be the cycle type of $\sigma\in\mathfrak{S}_{N}$. Note that the orbits of $\sigma$ on $\mathbb{X}_{N}$ are in one-to-one correspondence with cycles of $\sigma$. Clearly, any $\sigma$-invariant subset is a disjoint union of some of orbits of $\sigma$.
Consider any partition $\{A_j\}_{j\in\mathbb{Z}}\in \mathcal{A}(\sigma, \mathfrak{m})$ and let $\beta_{jl}$  be the number of $\sigma$-orbits of length $l$ in $A_j$. Then, we have
\begin{equation}\label{beta_matr_cond}
\sum_{l=1}^{\infty}\beta_{jl}l=d_j,~\sum_{j\in\mathbb{Z}}\beta_{jl}=c_l(\sigma).
\end{equation}
Observe that the action of $\sigma$ on $A_{j}$ is a permutation of the cycle type $\beta_j=(\beta_{j1},\beta_{j2},\ldots)$.
Thus any partition $\left\{ A_j\right\}_{j\in\mathbb{Z}}\in \mathcal{A}(\sigma,\mathfrak{m})$ defines the matrix 
\begin{equation*}
\beta\left( \left\{ A_j\right\}_{j\in\mathbb{Z}}\right)=
\left[\begin{matrix}
\vdots&\vdots&\vdots&\vdots&\iddots
\\
\beta_{-1,1}&\beta_{-1,2}&\beta_{-1,3}&\beta_{-1,4}&\ldots
\\
\beta_{0,1}&\beta_{0,2}&\beta_{0,3}&\beta_{0,4}&\ldots
\\
\beta_{1,1}&\beta_{1,2}&\beta_{1,3}&\beta_{1,4}&\ldots
\\
\vdots&\vdots&\vdots&\vdots&\ddots
\end{matrix}\right]
\end{equation*}
with non-negative integer entries that satisfy \eqref{beta_matr_cond}. Let us denote  by $\beta\in\mathcal{B}(\sigma,\mathfrak{m})$ the set of all such matrices.
For any $\beta\in\mathcal{B}(\sigma,\mathfrak{m})$ set
\begin{equation*}
\mathcal{A}(\beta,\sigma,\mathfrak{m})=\{\{A_j\}_{j\in\mathbb{Z}}\in\mathcal{A}(\sigma, \mathfrak{m}):\beta(\{A_j\})=\beta\}.
\end{equation*}
Note that
\begin{equation}\label{A_beta_sigma_m_card}
\#\mathcal{A}(\beta,\sigma,\mathfrak{m})=\prod_{l=1}^{\infty}
\binom{c_l(\sigma)}{\{\beta_{jl}\}_{j\in\mathbb{Z}}}=\prod_{l=1}^{\infty}\Big(c_{l}(\sigma)!\cdot\prod_{j\in\mathbb{Z}}(\beta_{jl}!)^{-1}\Big).
\end{equation}
\begin{Exa}
Consider a special case when $\sigma$ is the identity permutation of $\mathfrak{S}_{N}$. Then, the cycle type of $\sigma$ is $c(\sigma)=(N,0,0,\ldots)$. Therefore, the set $\mathcal{B}(\sigma,\mathfrak{m})$ consists of only one element, namely, the matrix $\beta=\{\beta_{jl}\}_{j\in\mathbb{Z},l\in\mathbb{N}}$ with entries
\begin{equation*}
\beta_{jl}=
\begin{cases}
d_j,~&l=1,
\\
0,~&l>1.
\end{cases}
\end{equation*}
The corresponding set $\mathcal{A}(\beta,\sigma,\mathfrak{m})$ is just the set of all partitions $\{A_j\}_{j\in\mathbb{Z}}$ of the set $\mathbb{X}_{N}$ into disjoint subsets $A_j$, $j\in\mathbb{Z}$, such that $\#A_j=d_j$.
\end{Exa}

We summarize the discussion above in the following proposition.
\begin{Prop}\label{Proposition 4.1 _character_formula}
The irreducible representations of the group $G_{N}$ can be parameterized by pairs $(\mathfrak{m},\lambda^{\rho})$, where $\mathfrak{m}\in\mathcal{F}_{N}(\mathbb{Z})$ and $\lambda^{\rho}$ is a sequence of partitions $\{\lambda(j)\vdash d_j=\#Y_j(\mathfrak{m})\}_{j\in\mathbb{Z}}$. Using the notation introduced above, the corresponding normalized character is given by the formula
\begin{equation}\label{G_N_char_formula}
\chi_{\pi}(g)=\frac{1}{\#\mathfrak{C}_\mathfrak{m}}\sum_{\beta\in\mathcal{B}(\sigma,\mathfrak{m})}\sum_{\{A_j\}\in\mathcal{A}(\beta,\sigma,\mathfrak{m})}\prod_{j\in\mathbb{Z}}\prod_{x\in A_j}\mathfrak{t}(x)^{j}\times\prod_{j\in\mathbb{Z}}\chi_{\lambda(j)}^{\beta_j},~g=(\mathfrak{t},\sigma)\in G_N.
\end{equation}
Here,  $\beta_j=(\beta_{j1},\beta_{j2},\ldots)$ is the cycle type of the permutation $\mathfrak{p}_j(\sigma)$ and $\chi_{\lambda}^{\mu}$ is the value of the normalized character of the irreducible representation  $\rho_{\lambda}$ on a permutation of the cycle type $\mu$.
\end{Prop}

\section{Limits of characters of $G_{N_k}$}\label{sect:char_limits}

\subsection{Special characters of the group $\mathfrak{S}_{\widehat{\mathbf{n}}}$}
Recall that $\mathfrak{S}_{\widehat{\mathbf{n}}}$ is the inductive limit of symmetric groups $\{\mathfrak{S}_{N_k}\}_{k}$ with the block-diagonal embedding. The group $\mathfrak{S}_{\widehat{\mathbf{n}}}$ has at most two multiplicative characters, namely, the \emph{trivial} character and the \emph{sign} character $\sgn_{\infty}\colon\mathfrak{S}_{\widehat{\mathbf{n}}}\to\mathbb{C}$ given by
\begin{equation*}
\sgn_{\infty}(\sigma):=\lim_{k\to\infty}\sgn_{N_k}(\sigma),
\end{equation*}
where $\sgn_{N_k}(\sigma)$ is the sign of $\sigma$ viewed as an element of $\mathfrak{S}_{N_k}$. Note that $\sgn_{\infty}$ is non-trivial if and only if the sequence $\widehat{\mathbf{n}}=\{n_k\}_{k=1}^{\infty}$ contains only finitely many even numbers.

Denote also $\chi_{\reg}\colon\mathfrak{S}_{\widehat{\mathbf{n}}}\to\mathbb{C}$ the \emph{regular character} of $\mathfrak{S}_{\widehat{\mathbf{n}}}$ defined by the formula
\begin{equation*}
\chi_{\reg}(\sigma)=
\begin{cases}
1,~&\sigma=1_{\mathfrak{S}_{\widehat{\mathbf{n}}}},
\\
0,~&\sigma\neq1_{\mathfrak{S}_{\widehat{\mathbf{n}}}}.
\end{cases}
\end{equation*}
Finally, the \emph{natural} character $\chi_{\nat}$ of $\mathfrak{S}_{\widehat{\mathbf{n}}}$ is given by the formula
\begin{equation*}
\chi_{\nat}(\sigma)=\nu_{\widehat{\mathbf{n}}}(\Fix(\sigma;\mathbb{X}_{\widehat{\mathbf{n}}})),
\end{equation*}
where $\Fix(\sigma;\mathbb{X}_{\widehat{\mathbf{n}}})$ stands for the subset of $\sigma$-fixed points in $\mathbb{X}_{\widehat{\mathbf{n}}}$ (see Section \ref{sect:measure_space_interpret}). For any permutation $\sigma\in\mathfrak{S}_{N}$ we denote by $\Fix_{N}(\sigma)$ the set of its fixed points in $\mathbb{X}_{N}$. In particular, for $\sigma\in\mathfrak{S}_{N_k}\subset\mathfrak{S}_{\widehat{\mathbf{n}}}$ we have
\begin{equation*}
\chi_{\nat}(\sigma)=\frac{\#\Fix_{N_k}(\sigma)}{N_k}.
\end{equation*}
For convenience we introduce the following notation for special characters of $\mathfrak{S}_{\widehat{\mathbf{n}}}$: for $\sigma\in\mathfrak{S}_{\widehat{\mathbf{n}}}$ we let
\begin{equation}\label{phi_mathrm sym}
\phi_{\sym}(\sigma;p)=
\begin{cases}
1,~&p=\triv~\text{(trivial character)},
\\
\sgn_{\infty}(\sigma),~&p=\sgn~\text{(sign character)},
\\
\chi_{\reg}(\sigma),~&p=\reg~\text{(regular character)}.
% \\
% \chi_{\mathrm{nat}}(\sigma),~&p=\nat~\text{(natural character)}.
\end{cases}
\end{equation}

\subsection{The approximation theorem.} As was the case for the spherical functions, to obtain the classification of indecomposable characters on $G_{\widehat{\mathbf{n}}}$ we use the approximation technique.

\begin{Prop}\label{prop:approx_char}
Any normalized indecomposable character of the group $G_{\widehat{\mathbf{n}}}$ is a weak limit of a sequence of normalized irreducible characters of groups $\{G_{N_{k_l}}\}_{l=1}^{\infty}$ for some subsequence $\{N_{k_l}\}_{l=1}^{\infty}$.
\end{Prop}

\begin{proof}
The proof of this fact in a more general situation of a chain of compact groups is given in Appendix, see Proposition \ref{prop:general_approx_thm}.    
\end{proof}

The proposition above reduces the problem of classification of indecomposable characters of $G_{\widehat{\mathbf{n}}}$ to computation of weak limits of irreducible characters of finite-dimensional subgroups $G_{N_k}$. This is the main content of this section, see Proposition \ref{prop:character_limits}. Note, however, that Proposition \ref{prop:approx_char} does not imply that every such a limit is indeed an indecomposable character, and proving this fact in our case requires an additional effort. We complete the classification in Section \ref{sect:char_classification}, see Theorem \ref{thm:classification_char}.

\subsection{Calculation of limits.}

For any integer $M$ define
\begin{equation}\label{eq:phi_M_formula}
\phi_{M}(g)=\int_{\Fix(\sigma;\mathbb{X}_{\widehat{\mathbf{n}}})}\mathfrak{t}(x)^{M}d\nu_{\widehat{\mathbf{n}}}(x),~g=(\mathfrak{t},\sigma)\in G_{\widehat{\mathbf{n}}}.
\end{equation}
Here by $\Fix(\sigma;\mathbb{X}_{\widehat{\mathbf{n}}})$ we denote the set of $\sigma$-fixed points in $\mathbb{X}_{\widehat{\mathbf{n}}}$.
Note that the value of $\phi$ at the identity of $G_{\widehat{\mathbf{n}}}$ equals to 1.

\begin{Prop}\label{prop:character_limits}
Let $\{\pi_k\}_{k=1}^{\infty}$ be a sequence of irreducible representations of groups $G_{N_k}$. Assume that the corresponding sequence of normalized characters $\{\chi_{\pi_k}\}_{k=1}^{\infty}$ converges pointwise on $G_{\widehat{\mathbf{n}}}$ to a continuous function $\chi\colon G_{\widehat{\mathbf{n}}}\to\mathbb{C}$. Then, there exist a (possibly empty) collection of integers $M_1,\ldots,M_l$ and $p\in\{\triv,\sgn,\reg\}$ such that for any $g=(\mathfrak{t},\sigma)\in G_{\widehat{\mathbf{n}}}$
\begin{equation*}
\chi(g)=\phi_{\sym}(\sigma;p)\phi_{M_1}(g)\cdots\phi_{M_l}(g).
\end{equation*}
\end{Prop}
\begin{proof}
The representation $\pi_k$ is given by a character $\mathfrak{m}_k\in\mathcal{F}_{N_k}(\mathbb{Z})$ and an irreducible representation of $\Stab(\mathfrak{m}_{k})$, namely
\begin{equation*}
\pi_k\simeq\Ind_{\mathcal{F}_{N_k}(\mathbb{T})\times\Stab(\mathfrak{m}_{k})}^{G_{N_k}}(R_{\mathfrak{m}_k}\otimes\rho^{(k)}),
\end{equation*}
where $\rho^{(k)}=\bigotimes_{j\in\mathbb{Z}}\rho_{\lambda(k;j)}$ is a representation of the group $\Stab(\mathfrak{m}_k)$. Here $\{\lambda(k;j)\}_{j\in\mathbb{Z}}$ is a sequence of partitions $\lambda(k;j)\vdash\#Y_j(\mathfrak{m}_k)$, where $Y_j(\mathfrak{m}_k)=\left\{ x\in \mathbb{X}_{N_k}: \mathfrak{m}(x)=j\right\}$.

Fix an arbitrary element $g=(\mathfrak{t},\sigma)\in G_{N_k}\subset G_{\widehat{\mathbf{n}}}$. Let $c(\sigma)=\left( c_1(\sigma),c_2(\sigma),\ldots\right)$ be its cycle type. Recall that if we regard $(\mathfrak{t},\sigma)\in G_{N_k}$ as an element of group $G_{N_l}$, where $l>k$, then the cycle type of $\sigma$ is $\left(c_1(\sigma)\frac{N_l}{N_k},c_2(\sigma)\frac{N_l}{N_k},\ldots\right)$. For this reason there exist non-negative rational numbers $\{\alpha_i:i=1,2,\ldots\}$ such that all but finitely many of them are zero, and $\alpha_1+2\alpha_2+\ldots=1$ and $c_l(\sigma)=\alpha_{l}N_k$ for all large $k$.

By formula \eqref{G_N_char_formula} we have
\begin{equation}\label{G_N_k_char_formula}
\chi_{\pi_k}(g)=\chi_{\pi_k}(\mathfrak{t}\sigma)=\frac{1}{\#\mathfrak{C}_{\mathfrak{m}}}\sum_{\beta\in\mathcal{B}(\sigma,\mathfrak{m})}\sum_{\{A_j\}
\in\mathcal{A}(\beta,\sigma,\mathfrak{m})}\prod_{j\in\mathbb{Z}}\prod_{x\in A_j}\mathfrak{t}(x)^{j}\times\prod_{j\in\mathbb{Z}}\chi_{\lambda(k;j)}^{\beta_j},\;x \in \mathbb{X}_{N_k}.
\end{equation}
Our aim is to compute the limit of $\{\chi_{\pi_k}(g)\}_{k=1}^{\infty}$ that exists according to our assumptions. We split this calculation into several steps. Since we are interested only in the limiting function we can always pass to a subsequence if necessary which we do freely in the arguments below.

\medskip

\noindent
\textbf{Step 1.} {\it The sequence $\{\mathfrak{m}_k\}_{k=1}^{\infty}$ satisfies the condition of Proposition \ref{prop:sph_func_conv_seq}. In particular, there exist non-negative integers $\{d_j\}_{j\in\mathbb{Z}\setminus\{0\}}$ which satisfy the following conditions:
\begin{itemize}
\item
$d_j=0$ for all but finitely many $j\in\mathbb{Z}\setminus\{0\}$;

\item
for all large $k$ we have $\#Y_j(\mathfrak{m}_k)=d_j$ for $j\neq 0$ and $\#Y_0(\mathfrak{m}_k)=N_k-d$, where $d=\sum_{j\in\mathbb{Z}\setminus\{0\}}d_j$.
\end{itemize}
}

\smallskip

Indeed, observe that for $\sigma=1$ we have $\chi_{\pi_k}(\mathfrak{t})=\varphi_{\mathfrak{m}_k}(\mathfrak{t})$. Since $\{\chi_{\pi_k}\}_{k=1}^{\infty}$ weakly converges to a continuous function on $\mathcal{F}_{\widehat{\mathbf{n}}}(\mathbb{T})$, the same holds for $\{\varphi_{\mathfrak{m}_k}\}_{k=1}^{\infty}$. Therefore, the sequence $\mathfrak{m}_k$ satisfies the condition of Proposition \ref{prop:sph_func_conv_seq}. Now the existence of the required sequence $\{d_{j}\}_{j\in\mathbb{Z}\setminus\{0\}}$ follows from the conclusion of Proposition \ref{prop:sph_func_conv_seq}.

\medskip

\noindent
\textbf{Step 2.} {\it
We may assume that there exist partitions $\{\mu(j)\vdash d_j\}_{j\in\mathbb{Z}\setminus\{0\}}$ such that $\lambda(k;j)=\mu(j)$ for all large $k$ and any $j\in\mathbb{Z}\setminus\{0\}$.
}

\medskip

Observe that for any $j\neq 0$ the sequence $\{\lambda(k;j)\}_{k=1}^{\infty}$ is a sequence of partitions of $d_j$. Since the set of partitions of given weight is finite, we can  assume without loss of generality that $\lambda(k;j)\equiv\mu(j)$ for all $k$ and $j\neq 0$, where $\mu(j)\vdash d_j$ (pass to a subsequence if needed).

\medskip

\noindent
\textbf{Step 3.} {\it
Under the assumptions of Steps 1 and 2 we can rewrite the formula for $\chi_{\pi_k}(g)$ as follows:
\begin{equation}\label{character_k}
\chi_{\pi_k}(\mathfrak{t}\sigma)=\frac{1}{\#\mathfrak{C}_{\mathfrak{m}_k}}\sum_{\beta\in\mathcal{B}(\sigma,\mathfrak{m}_k)}
\sum_{\{A_{j}\}\in\mathcal{A}(\beta,\sigma,\mathfrak{m}_k)}\prod_{j\in\mathbb{Z}\setminus\{0\}}\prod_{x\in A_j}\mathfrak{t}(x)^{j}\times\chi_{\lambda(k;0)}^{\beta_0}\prod_{j\in\mathbb{Z}\setminus\{0\}}\chi_{\mu(j)}^{\beta_j},
\end{equation}
where for a matrix $\beta=\{\beta_{jl}\}_{j\in\mathbb{Z},l\in\mathbb{N}}$ in $\mathcal{B}(\sigma,\mathfrak{m}_k)$ we denote $\beta_j=(\beta_{j1},\beta_{j2},\ldots)$.
}

\medskip

This formula is an immediate consequence of \eqref{G_N_k_char_formula}.

\medskip

\noindent
\textbf{Step 4.} {\it
The number of elements in $\mathcal{B}(\sigma,\mathfrak{m}_k)$ is bounded by a constant independent on $k$.
}

\medskip

Observe that any matrix $\beta=\left[\beta_{jl}\right]_{j\in\mathbb{Z},l\in\mathbb{N}}\in\mathcal{B}(\sigma,\mathfrak{m}_k)$ is in fact uniquely determined by the submatrix $\left[\beta_{jl}\right]_{j\in\mathbb{Z}\setminus\{0\},l\in\mathbb{N}}$. Indeed, we can recover $\{\beta_{0l}\}_{l\in\mathbb{N}}$ using \eqref{beta_matr_cond}:
\begin{equation}\label{zero_beta}
\beta_{0l}=c_l(\sigma)-\sum_{j\in\mathbb{Z}\setminus\{0\}}\beta_{jl}=\alpha_l N_k-\sum_{j\in\mathbb{Z}\setminus\{0\}}\beta_{jl}.
\end{equation}
Besides that, for any non-zero $j$ we have $\sum_{l\ge1}\beta_{jl}l=\#Y_j(\mathfrak{m}_k)=d_j$ and hence, $\beta_{jl}\le d_{j}/l$. Thus, we have the following estimate:
\begin{equation*}
\#\mathcal{B}(\sigma,\mathfrak{m}_k)\leq\prod_{j\in\mathbb{Z}\setminus\{0\}}\prod_{l=1}^{\infty}(d_j/l+1).
\end{equation*}
Note that only finitely many $d_j$'s are non-zero, the right-hand side of the inequality above is finite. Since this bound does not depend on $k$, the claim follows.

\medskip

\noindent
\textbf{Step 5.} {\it
Introduce the matrix $\widehat{\beta}(k)=\left[\widehat{\beta}_{jl}(k) \right]\in\mathcal{B}(\sigma,\mathfrak{m}_k)$, where
$
\widehat{\beta}_{jl}(k)=
\begin{cases}
0,~& j\neq 0~\text{and}~l>1,
\\
d_j,~& j\neq 0~\text{and}~l=1,
\\
\alpha_lN_k,& j=0~\text{and}~l>1,
\\
\alpha_1N_k-d,& j=0~\text{and}~l=1,
\end{cases}
$
and $d=\sum_{j\in\mathbb{Z}\setminus\{0\}}d_j$. Then, for any $\beta\in\mathcal{B}(\sigma,\mathfrak{m}_k)\setminus\{\widehat{\beta}(k)\}$ we have
\begin{equation}\label{frac_estimate}
\frac{\#\mathcal{A}(\beta,\sigma,\mathfrak{m}_k)}{\#\mathfrak{C}_{\mathfrak{m}_k}}\leq C_0/N_k,
\end{equation}
where $C_0>0$ is an absolute constant independent on $k$.
}
\begin{Rem}
As we will see in Step 6, the matrix $\widehat{\beta}(k)$ plays a special role because it gives the main contribution to the limit.
\end{Rem}

Set $p_l=\sum\limits_{j\in \mathbb{Z}\setminus{0}}\beta_{jl}$. If $\beta\neq\widehat{\beta}$ then there exist $l'>1$ and $j\neq 0$ such that $\beta_{jl'}>0$. Since $\sum\limits_{l=1}^\infty\sum\limits_{{j\in\mathbb{Z}\setminus\{0\}}}l\beta_{jl}=d$ we have
\begin{equation}\label{difference_betwen_d_sump}
d-\sum\limits_{l=1}^\infty p_l=\sum\limits_{l=1}^\infty\sum\limits_{j\in \mathbb{Z}\setminus{0}}(l-1)\beta_{jl}\geq(l'-1)\beta_{jl'}\geq\beta_{jl'}\geq 1.
\end{equation}
It follows from \eqref{A_beta_sigma_m_card} that
\begin{equation*}
\#\mathcal{A}(\beta,\sigma,\mathfrak{m}_k)=\prod_{l=1}^{\infty}\binom{\alpha_l N_k}{\{\beta_{jl}\}_{j\in\mathbb{Z}}}=\prod_{l=1}^{\infty}\frac{(\alpha_l N_k)!}{\beta_{0l}!\prod_{j\in\mathbb{Z}\setminus\{0\}}\beta_{jl}!}.
\end{equation*}
Since $\beta_{0l}=\alpha_{l}N_k-p_l$ by \eqref{zero_beta}, we obtain
\begin{equation*}
\#\mathcal{A}(\beta,\sigma,\mathfrak{m}_k)=\prod_{l=1}^{\infty}\frac{(\alpha_l N_k)\ldots(\alpha_lN_k-p_l+1)}{\prod_{j\in\mathbb{Z}\setminus\{0\}}\beta_{jl}!}\leq N_k^{p_1+p_2+\ldots}.
\end{equation*}
Now we conclude from equality $\#\mathfrak{C}_{\mathfrak{m}_k}=\dfrac{N_k!}{(N_k-d)!\prod_{j\in\mathbb{Z}\setminus\{0\}}d_j!}$ that
\begin{equation*}
\frac{\#\mathcal{A}(\beta,\sigma,\mathfrak{m}_k)}{\#\mathfrak{C}_{\mathfrak{m}_k}}\leq \frac{\prod_{j\in\mathbb{Z}\setminus\{0\}}d_j!\cdot N_k^{\sum_{l=1}^\infty p_l}}{N_k(N_k-1)\cdots(N_k-d+1)}.
\end{equation*}
Applying \eqref{difference_betwen_d_sump}, we obtain the bound \eqref{frac_estimate}.

\medskip

\noindent
\textbf{Step 6.} {\it
We have the equality
\begin{equation}\label{hat_k}
\chi_{\pi_k}(\mathfrak{t}\sigma)=\chi_{\lambda(k;0)}^{\widehat{\beta}_0(k)}\cdot\frac{1}{\#\mathfrak{C}_{\mathfrak{m}_k}}
\sum_{\{A_{j}\}\in\mathcal{A}(\widehat{\beta}(k),\sigma,\mathfrak{m}_k)}\prod_{j\in\mathbb{Z}\setminus\{0\}}\prod_{x\in A_j}\mathfrak{t}(x)^{j}+O(1/N_k),~k\to\infty.
\end{equation}
In other words, only the terms with $\beta=\widehat{\beta}(k)$ in formula \eqref{G_N_k_char_formula} can contribute to the limit as $k\to\infty$.
}

\medskip

Take any $\beta\in \mathcal{B}(\sigma,\mathfrak{m}_k)$ and suppose that $\beta\neq\widehat{\beta}(k)$.
Since $\chi_{\mu(j)}$ and $\chi_{\lambda(k;0)}$ are normalized characters (hence bounded by $1$ in the absolute value), we obtain from Step 5 that
\begin{equation}\label{G_N_k_char_asymp_expr}
\frac{1}{\#\mathfrak{C}_{\mathfrak{m}_k}}\left|\sum_{\{A_j\}\in\mathcal{A}(\beta,\sigma,\mathfrak{m}_k)}\prod_{j\in\mathbb{Z}\setminus\{0\}}\prod_{x\in A_j}\mathfrak{t}(x)^{j}\times\chi_{\lambda(k;0)}^{\beta_0}\prod_{j\in\mathbb{Z}\setminus\{0\}}\chi_{\mu(j)}^{\beta_j}\right|
\leq\frac{\#\mathcal{A}(\beta,\sigma,\mathfrak{m}_k)}{\#\mathfrak{C}_{\mathfrak{m}_k}}
\leq C_0/N_k.
\end{equation}
Hence, applying \eqref{character_k}, we get
\begin{equation*}
\begin{split}
&\left|\chi_{\pi_k}(\mathfrak{t}\sigma)-\frac{1}{\#\mathfrak{C}_{\mathfrak{m}_k}}\sum\limits_{\{A_j\}\in\mathcal{A}(\widehat{\beta}(k),\sigma,\mathfrak{m}_k)}
\prod_{j\in\mathbb{Z}\setminus\{0\}}\prod_{x\in A_j}\mathfrak{t}(x)^{j}\times\chi_{\lambda(k;0)}^{\widehat{\beta}_0(k)}
\prod_{j\in\mathbb{Z}\setminus\{0\}}\chi_{\mu(j)}^{\widehat{\beta}_j(k)}\right|\leq
\\
&\leq
\sum_{\substack{\beta\in\mathcal{B}(\sigma,\mathfrak{m}_k)\\ \beta\neq\widehat{\beta}(k)}}\frac{1}{\#\mathfrak{C}_{\mathfrak{m}_k}}\left|\sum\limits_{\{A_j\}\in\mathcal{A}(\beta,\sigma,\mathfrak{m}_k)}
\prod_{j\in\mathbb{Z}\setminus\{0\}}\prod_{x\in A_j}\mathfrak{t}(x)^{j}\times\chi_{\lambda(k;0)}^{\beta_0}
\prod\limits_{j\in\mathbb{Z}\setminus\{0\}}\chi_{\mu(j)}^{\beta_j}\right|\leq
\\
&\leq\#\mathcal{B}(\sigma,\mathfrak{m}_k)\cdot C_0/{N_k}=\big[~\text{{\bf Step 4}}~\big]=O(1/N_k),~k\to\infty.
\end{split}
\end{equation*}
Finally, by definition of $\widehat{\beta}(k)$ we have\footnote{See Step 3 for definition of the partition $\beta_j$.} $\widehat{\beta}_0(k)=(\alpha_1N_k-d,\alpha_2N_k,\alpha_3N_k,\ldots)$ and $\widehat{\beta}_j(k)=(d_j,0,0,\ldots)$ for $j\neq 0$. Thus, $\chi_{\mu(j)}^{\widehat{\beta}_j(k)}=1$ for all non-zero $j$ and
\begin{equation*}
\prod_{j\in\mathbb{Z}\setminus\{0\}}\prod_{x\in A_j}\mathfrak{t}(x)^{j}\times\chi_{\lambda(k;0)}^{\widehat{\beta}_0(k)}\prod_{j\in\mathbb{Z}\setminus\{0\}}
\chi_{\mu(j)}^{\widehat{\beta}_j(k)}=\chi_{\lambda(k;0)}^{\widehat{\beta}_0(k)}\cdot\prod_{j\in\mathbb{Z}\setminus\{0\}}\prod_{x\in A_j}\mathfrak{t}(x)^{j}.
\end{equation*}
This finishes the proof of \eqref{hat_k}.

\medskip

\noindent
\textbf{Step 7.} {\it
We have
\begin{equation*}
\lim_{k\to\infty}\frac{1}{\#\mathfrak{C}_{\mathfrak{m}_k}}\sum_{\{A_j\}\in\mathcal{A}(\widehat{\beta}(k),\sigma,\mathfrak{m}_k)}\prod_{j\in\mathbb{Z}\setminus\{0\}}\prod_{x\in A_j}\mathfrak{t}(x)^{j}=\prod_{j\in\mathbb{Z}\setminus\{0\}}\left(\int_{\Fix(\sigma;\mathbb{X}_{\widehat{\mathbf{n}}})}\mathfrak{t}(x)^{j}d\nu_{\widehat{\mathbf{n}}}(x)\right)^{d_j}.
\end{equation*}
}

\medskip

It follows from the definition of $\widehat{\beta}(k)$ that $\mathcal{A}(\widehat{\beta},\sigma_{k},\mathfrak{m}_{k})$ is the set of all partitions $\{A_{j}\}_{j\in\mathbb{Z}}$ of the set $\mathbb{X}_{N_k}$ into disjoint subsets such that $\#A_j=d_j$ and $\sigma$ acts on $A_j$ trivially for all non-zero $j$. Thus,
\begin{equation*}
\sum_{\{A_j\}\in\mathcal{A}(\widehat{\beta}(k),\sigma,\mathfrak{m}_k)}\prod_{j\in\mathbb{Z}\setminus\{0\}}\prod_{x\in A_j}\mathfrak{t}(x)^{j}=\sum_{\substack{\bigsqcup_{j\neq 0}A_{j}\subset\Fix_{N_k}(\sigma)\\\#A_j=d_j}}\prod_{j\in\mathbb{Z}\setminus\{0\}}\prod_{x\in A_j}\mathfrak{t}(x)^{j}.
\end{equation*}
Lemma \ref{gen_sph_func_lim} implies that
\begin{equation*}
\sum_{\substack{\bigsqcup_{j\neq 0}A_{j}\subset\Fix_{N_k}(\sigma)\\\#A_j=d_j}}\prod_{j\in\mathbb{Z}\setminus\{0\}}\prod_{x\in A_j}\mathfrak{t}(x)^{j}\sim N_k^d\cdot\prod_{j\in\mathbb{Z}\setminus\{0\}}\frac{1}{d_j!}\Bigg(\sum_{x\in\Fix_{N_k}(\sigma)}\mathfrak{t}(x)^{j}\Bigg)^{d_j}\,\text{as}~k\to\infty.
\end{equation*}
To conclude, it remains to note that $\#\mathfrak{C}_{\mathfrak{m}_k}=[\mathfrak{S}_{N_k}:\Stab(\mathfrak{m}_k)]\sim N_k^d/\prod_{j\in\mathbb{Z}\setminus\{0\}}d_j!$ as $k\to\infty$.

\medskip

\noindent
\textbf{Step 8.} {\it
Denote for $(\mathfrak{t},\sigma)\in G_{\widehat{\mathbf{n}}}$
\begin{equation*}
\phi(\mathfrak{t},\sigma)=\prod_{j\in\mathbb{Z}\setminus\{0\}}\left(\int_{\Fix(\sigma;\mathbb{X}_{\widehat{\mathbf{n}}})}\mathfrak{t}(x)^{j}d\nu_{\widehat{\mathbf{n}}}(x)\right)^{d_j}.
\end{equation*}
Possible options for the limit of the sequence $\{\chi_{\pi_k}(g)\}_{k=1}^{\infty}$ are as follows:
\begin{itemize}
\item
if $\limsup\limits_{k\to\infty}\min\{|\lambda(k;0)/(\lambda_1(k;0))|,|\lambda'(k;0)/(\lambda'_1(k;0))|\}=\infty$, then
\begin{equation*}
\lim\limits_{k\to\infty}\chi_{\pi_k}(g)=\chi_{\reg}(\sigma)\phi(\mathfrak{t},\sigma),
\end{equation*}

\item
if $\limsup\limits_{k\to\infty}|\lambda(k;0)/(\lambda_1(k;0))|<\infty$, then there exists a non-negative integer $q$ which does not depend on $g=(\mathfrak{t},\sigma)$ such that
\begin{equation*}
\lim_{k\to\infty}\chi_{\pi_k}(g)=\nu_{\widehat{\mathbf{n}}}(\Fix(\sigma;\mathbb{X}_{\widehat{\mathbf{n}}}))^{q}\cdot\phi(\mathfrak{t,\sigma})=(\chi_{\nat}(\sigma))^{q}\phi(\mathfrak{t,\sigma}),
\end{equation*}

\item
if $\limsup\limits_{k\to\infty}|\lambda'(k;0)/(\lambda'_1(k;0))|<\infty$, then there exists a non-negative integer $q$, which does not depend on $g=(\mathfrak{t},\sigma)$ such that
\begin{equation*}
\lim_{k\to\infty}\chi_{\pi_k}(g)=\sgn_{\infty}(\sigma)\nu_{\widehat{\mathbf{n}}}(\Fix(\sigma;\mathbb{X}_{\widehat{\mathbf{n}}}))^{q}\cdot\phi(\mathfrak{t},\sigma)=\sgn_{\infty}(\sigma)(\chi_{\nat}(\sigma))^{q}\phi(\mathfrak{t,\sigma}),
\end{equation*}
where $\sgn_{\infty}(\sigma)=\lim\limits_{k\to\infty}\sgn_{N_k}(\sigma)$.
\end{itemize}
}

\medskip

Indeed, from Steps 6 and 7 we obtain
\begin{equation*}
\chi_{\pi_k}(\mathfrak{t}\sigma)=\chi_{\lambda(k;0)}^{\widehat{\beta}_0(k)}\cdot\prod_{j\in\mathbb{Z}\setminus\{0\}}\left(\int_{\Fix(\sigma;\mathbb{X}_{\widehat{\mathbf{n}}})}\mathfrak{t}(x)^{j}d\nu_{\widehat{\mathbf{n}}}(x)\right)^{d_j}+o(1)=\chi_{\lambda(k;0)}^{\widehat{\beta}_{0}(k)}\cdot\phi(\mathfrak{t},\sigma)+o(1),~k\to\infty.
\end{equation*}
Now we consider three cases depending on the behaviour of the sequence $\{\lambda(k;0)\}_{k=1}^{\infty}$.
\begin{itemize}
\item
If $\limsup\limits_{k\to\infty}\min\{|\lambda(k;0)/(\lambda_1(k;0))|,|\lambda'(k;0)/(\lambda'_1(k;0))|\}=+\infty$, then Lemma \ref{lem:sym_char_lim} implies that as $k\to\infty$ we have $\chi_{\lambda(k;0)}^{\widehat{\beta}_0(k)}\to 0$ if $\alpha_1<1$, and $\chi_{\lambda(k;0)}^{\widehat{\beta}_0(k)}\to 1$ if $\alpha_1=1$ (note that $\alpha_1=1$ is equivalent to $\sigma=1$). Hence, $\{\chi_{\pi_k}(g)\}_{k=1}^{\infty}$ converges to $\chi_{\reg}(\sigma)\phi(\mathfrak{t},\sigma)$.

\item
If $\liminf\limits_{k\to\infty}|\lambda(k;0)/(\lambda_1(k;0))|<+\infty$, then one can pass to a subsequence and assume that $\lambda(k;0)/(\lambda_1(k))=\mu$ for some fixed partition $\mu$. Denote $q=|\mu|$ and apply Lemma \ref{lem:sym_char_lim}. Since $\widehat{\beta}_0(k)=(\alpha_1N_k-d,\alpha_2N_k,\alpha_3 N_k,\ldots)$ and $\nu_{\widehat{\mathbf{n}}}(\Fix(\sigma;\mathbb{X}_{\widehat{\mathbf{n}}}))=\alpha_1$ this gives the required formula.

\item
If $\liminf\limits_{k\to\infty}|\lambda'(k;0)/(\lambda'_1(k;0))|<+\infty$, then we apply Lemma \ref{lem:sym_char_lim} as in the second case. Note that the factor $\sgn_{\infty}(\sigma)=\lim\limits_{k\to\infty}\sgn_{N_k}(\sigma)$ equals to the factor $\lim\limits_{k\to\infty}(-1)^{N_k(\alpha_2+\alpha_4+\ldots)}$ from Lemma \ref{lem:sym_char_lim}.
\end{itemize}

\medskip

\noindent
\textbf{The final formula.}
Now we are ready to conclude the proof of Proposition \ref{prop:character_limits}. Firstly, observe that
\begin{equation*}
\prod_{j\in\mathbb{Z}\setminus\{0\}}\left(\int_{\Fix(\sigma;\mathbb{X}_{\widehat{\mathbf{n}}})}\mathfrak{t}(x)^{j}d\nu_{\widehat{\mathbf{n}}}(x)\right)^{d_j}=\prod_{j\in\mathbb{Z}\setminus\{0\}}\phi((\mathfrak{t},\sigma);j)^{d_j}.
\end{equation*}
Denote the right-hand side of the equality above by $\phi(\mathfrak{t},\sigma)$.

Now, if the first case of Step 8 applies, then we obtain $\lim\limits_{k\to\infty}\chi_{\pi_k}(g)=\chi_{\reg}(\sigma)\cdot\phi(\mathfrak{t},\sigma)=\phi_{\sym}(\sigma;\reg)\phi(\mathfrak{t},\sigma)$, as claimed. If the second or the third case of Step 8 applies, then we let $(M_1,\ldots,M_l)$ to be the sequence of integers consisting of $q$ zeroes and $d_j$ elements equal to $j$ for each $j\neq 0$. Then, by Step 8
\begin{equation*}
\lim_{k\to\infty}\chi_{\pi_k}(g)=\phi_{\sym}(\sigma;p)(\chi_{\mathrm{nat}}(\sigma))^{q}\phi(\mathfrak{t},\sigma)=\phi_{\sym}(\sigma;p)\prod_{r=1}^{l}\phi_{M_r}(g)
\end{equation*}
for some $p\in\{\triv,\sgn\}$, where $\phi_{\sym}(\sigma;p)$ is defined in \eqref{phi_mathrm sym}. This finishes the proof of Proposition \ref{prop:character_limits}.
\end{proof}

\section{Classification of indecomposable characters on $G_{\hat{\mathbf{n}}}$}\label{sect:char_classification}

\subsection{Characters.} For any non-negative integer $l$ denote by $\mathbb{Z}^{(l)}$ the set of all $l$-element multisets of $\mathbb{Z}$, i.e., $\mathbb{Z}^{(l)}$ is the quotient of $\mathbb{Z}^l$ by the natural action of $\mathfrak{S}_l$ given by $s\colon(M_1, \ldots,M_l)\mapsto \left(M_{s^{-1}(1)},\ldots,M_{s^{-1}(l) } \right)$. 
Sometimes we identify $\mathbb{Z}^{(l)}$ with the set of all non-increasing $l$-tuples of integers, that is, with elements $\mathbf{m}=(M_1,\ldots,M_l)$ of $\mathbb{Z}^l$ satisfying $M_1\geq\ldots\geq M_l$. 
By convention, for $l=0$ the set $\mathbb{Z}^{(0)}$ consists only of the empty tuple.
Finally, let $\mathbb{Z}^{(\infty)}$ be the set of countable multisets $\mathbf{m}=\{M_1,M_2,\ldots,M_l,\ldots\}$ of integers such that all but finitely many elements $M_i$ are zero. For $\mathbf{m}\in \mathbb{Z}^{(l)}$, $l=0,1,\ldots,\infty$ we also denote by $\widehat{\mathbf{m}}=\{M_{i_1},\ldots,  M_{i_p}\}\in\mathbb{Z}^{(p)}$ the (finite) multi-subset of all nonzero elements of $\mathbf{m}$. The cardinality or length of a multiset $\mathbf{m}=\{M_1,\ldots,M_l\}$ is denoted by $\ell(\mathbf{m})=l$, and we put $\ell(\mathbf{m})=\infty$ if $\mathbf{m}\in\mathbb{Z}^{(\infty)}$.

For any $\mathbf{m}=(M_1,\ldots,M_l)$ in $\mathbb{Z}^{(l)}$, where $l<\infty$, and $g=(\mathfrak{t},s)\in G_{\widehat{\mathbf{n}}}$ we define
\begin{equation}\label{eq:phi_m_def}
\phi_{\mathbf{m}}(g)=\phi_{M_1}(g)\cdots\phi_{M_l}(g)=\prod_{i=1}^{l}\phi_{M_i}(g),~\text{where}~\phi_{M_j}~\text{is defined in \eqref{eq:phi_M_formula}},
\end{equation}
and set $\phi_{\mathbf{m}}(g)\equiv 1$ if  $\mathbf{m}$ is the empty multiset. Then, for $g=\mathfrak{t}s$, where $\mathfrak{t}\in\mathcal{F}_{\widehat{\mathbf{n}}}(\mathbb{T})$, $s\in \mathfrak{S}_{\widehat{\mathbf{n}}}$, we have
\begin{equation*}
\phi_{\mathbf{m}}(g)=\chi_{\rm nat}(s)^{l-p}\prod_{j=1}^p\phi_{M_{i_j}}(g),~ l=\ell(\mathbf{m}),~\widehat{\mathbf{m}}=\{M_{i_1},\ldots,M_{i_p}\}\in\mathbb{Z}^{(p)}.
\end{equation*}
In the case $l=\infty$ the analogous product is infinite, and for $\mathbf{m}=(M_1,M_2,\ldots)\in\mathbb{Z}^{(\infty)}$ we set
\begin{equation}\label{eq:phi_infty}
\phi_{\mathbf{m}}(g)=\chi_{\reg}(s) \prod_{j=1}^p\phi_{M_{i_j}}(\mathfrak{t}),~\widehat{\mathbf{m}}=\{M_{i_1},\ldots,M_{i_p}\}\in\mathbb{Z}^{(p)}.
\end{equation}
Note that this is compatible with finite $l$ case since $\chi_{\reg}(s)=\lim\limits_{l\to\infty}(\chi_{\nat}(s))^{l}$.

Now we can define the functions on $G_{\widehat{\mathbf{n}}}$ which appear in the classification of indecomposable characters.
For $p\in\{\triv,\sgn\}$ and $\mathbf{m}\in\mathbb{Z}^{(l)}$, or $p=\reg$ and $\mathbf{m}\in\mathbb{Z}^{(\infty)}$ we set
\begin{equation}\label{eq:phi_p_m_def}
\phi_{p,\mathbf{m}}(g)=\phi_{\sym}(s;p)\phi_{\mathbf{m}}(g)=
\begin{cases}
\phi_{M_1}(g)\cdots\phi_{M_l}(g),~& p=\triv,~\mathbf{m}=(M_1,\ldots,M_l),
\\
\sgn_{\infty}(s)\phi_{M_1}(g)\cdots\phi_{M_l}(g),~& p=\sgn,~\mathbf{m}=(M_1,\ldots,M_l),
\\
\chi_{\reg}(s)\phi_{M_1}(g)\cdots\phi_{M_l}(g),~& p=\reg,~\widehat{\mathbf{m}}=(M_1,\ldots,M_l).
\end{cases}
\end{equation}
In particular, we have $\phi_{\triv,\mathbf{m}}=\phi_{\mathbf{m}}$. Observe that the restriction of $\phi_{p,\mathbf{m}}$ to the subgroup $\mathfrak{S}_{\widehat{\mathbf{n}}}$ of $G_{\widehat{\mathbf{n}}}$ is given by the formula
\begin{equation}\label{eq:phi_restriction_sym}
\phi_{p,\mathbf{m}}(s)=\chi_{\nat}(s)^{\ell(\mathbf{m})}\phi_{\sym}(s;p),~\text{where}~\phi_{\sym} ~\text{is defined in \eqref{phi_mathrm sym}}.
\end{equation}

\begin{Rem}\label{rem:equal_characters}
Note that some of the characters $\phi_{p,\mathbf{m}}$ may coincide. Namely, the characters $\phi_{p,\mathbf{m}}$ and $\phi_{q,\mathbf{k}}$ are equal only in the following cases:
\begin{itemize}
\item 
$p=q$ and $\mathbf{m}=\mathbf{k}$,

\item
$p=\triv$, $q=\sgn$ (or vice versa), $\mathbf{m}=\mathbf{k}$, and the sequence $\widehat{\mathbf{n}}=\{n_k\}_{k=1}^{\infty}$ contains infinitely many \emph{even} integers (in this case $\phi_{\sym}(-;\triv)\equiv\phi_{\sym}(-;\sgn)$).
\end{itemize}
\end{Rem}

\subsection{Indecomposability of $\phi_{p,\mathbf{m}}$.}
The primary goal of this section is to prove the following theorem.

\begin{Th}\label{thm:classification_char}
The set $\exchar(G_{\widehat{\mathbf{n}}})$ of indecomposable characters on $G_{\widehat{\mathbf{n}}}$ is
\begin{equation}\label{eq:list_characters}
\{\phi_{p,\mathbf{m}}:p\in\{\triv,\sgn\},\mathbf{m}\in\mathbb{Z}^{(l)},l\geq 0\}\sqcup\{\phi_{\reg,\mathbf{m}}:\mathbf{m}\in\mathbb{Z}^{(\infty)}\}.
\end{equation}
\end{Th}
\begin{Rem}\label{rem:G_infty_char_factorization}
Alternatively, the list can be rewritten as follows:
\begin{equation*}
\exchar(G_{\widehat{\mathbf{n}}})=\{\chi(g)=\psi(s)\cdot\phi_{\mathbf{m}}(g),~g=\mathfrak{t}s\in G_{\widehat{\mathbf{n}}}:\psi\in \exchar(\mathfrak{S}_{\widehat{\mathbf{n}}}),\mathbf{m}=(M_1,\ldots,M_l)\in\mathbb{Z}^{(l)},M_i\neq 0, l\ge 0\},
\end{equation*}
where $\exchar(\mathfrak{S}_{\widehat{\mathbf{n}}})$ denotes the set of indecomposable characters of $\mathfrak{S}_{\widehat{\mathbf{n}}}$ (see Proposition \ref{prop:inf_sym_indec_char}). 
This follows from the theorem and the fact that $\chi_{\nat}(s)=\phi_0(\mathfrak{t}s)$.
\end{Rem}

The proof of Theorem \ref{thm:classification_char} is divided into several lemmas.

\begin{Lm}\label{lem:phi_positive_def}
Each function from the list \eqref{eq:list_characters} is a normalized character on $G_{\widehat{\mathbf{n}}}$.
\end{Lm}

\begin{proof}
Let $\phi_{p,\mathbf{m}}$ be a character from \eqref{eq:list_characters}. Then, $\mathbf{m}=(M_1,M_2,\ldots)\in\mathbb{Z}^{(l)}$ for some $l=0,1,\ldots,\infty$ (if $p\in\{\triv,\sgn\}$, then $l$ is finite, if $p=\reg$, then $l=\infty$). Denote the multisubset of non-zero elements of $\mathbf{m}$ by $\widehat{\mathbf{m}}=(M_{{i_1}},\ldots,M_{i_q})\in\mathbb{Z}^{(q)}$.

Denote $d_j=\#\{i:M_i=j\}$ for $j\neq 0$, then  $q=\sum_{j\neq 0}d_j$. It follows from the definition \eqref{eq:phi_M_formula} that for any $M\in\mathbb{Z}$ the function $\phi_{M}\colon G_{\widehat{\mathbf{n}}}\to\mathbb{C}$ is central and $\phi_M\left(1_{G_{\widehat{\mathbf{n}}}}\right)=1$.
Consequently, it suffices to check that $\phi_{p,\mathbf{m}}$ is a positive-definite function on $G_{\widehat{\mathbf{n}}}$. Indeed, the proof of Proposition \ref{prop:character_limits} shows that $\phi_{p,\mathbf{m}}$ is a pointwise limit of a certain sequence $\{\chi_{\pi_{k}}\}_{k=1}^{\infty}$ of irreducible characters of groups $G_{N_k}$. Namely, one can take $\pi_{k}$ to be the irreducible representation of $G_{N_{k}}$ associated to $\mathfrak{m}_{k}\in\mathbb{Z}^{N_k}$ and sequence of partitions $\{\lambda(k;j)\}_{j\in\mathbb{Z}}$, where $\mathfrak{m}_{k}=(M_{i_1},\ldots,M_{i_q},\underbrace{0,\ldots,0}_{N_k-q})$ and
\begin{equation*}
\lambda(k;j)=
\begin{cases}
(d_j),~&\text{if}~j\neq 0,
\\
(N_k-l,1^{l-q}),~&\text{if}~j=0~\text{and}~p=\triv,
\\
(l-q+1,1^{N_k-l-1}),~&\text{if}~j=0~\text{and}~p=\sgn,
\\
(\lceil \frac{1}{2}(N_k-q)\rceil,1^{\lfloor \frac{1}{2}(N_k-q)\rfloor}),~&\text{if}~j=0~\text{and}~p=\reg.
\end{cases}
\end{equation*}
Since $\chi_{\pi_{k}}$ are positive-definite functions on $G_{N_k}$, the claim of the lemma follows.
\end{proof}
\begin{Rem}
Alternatively, one can prove that $\phi_{p,\mathbf{m}}$ is a character by constructing a unitary representation $\Pi\colon G_{\widehat{\mathbf{n}}}\to \mathcal{U}(\mathcal{H})$  with a unit vector $\xi$ in the ambient Hilbert space $\mathcal{H}$ such that $\varphi_{p,\mathbf{m}}(g)=\langle\Pi(g)\xi,\xi\rangle_{\mathcal{H}}$ for all $g\in G_{\widehat{\mathbf{n}}}$. See Section \ref{sect:realizations} and Remark \ref{phi_pos_def_rem} therein.
\end{Rem}

\begin{Prop}\label{prop:phi_indecomp}
Each character from the list \eqref{eq:list_characters} is an indecomposable character on $G_{\widehat{\mathbf{n}}}$.
\end{Prop}

The proof of this fact relies on the following lemma.

\begin{Lm}\label{lem:indecomp_main}
Assume that there exists a collection
\begin{equation*}
\{\alpha_{\mathbf{k}}:\mathbf{k}=(K_1,\ldots,K_j)\in\mathbb{Z}^{(j)},~j\in\mathbb{Z}_{\ge 0}\}\sqcup\{\alpha_{\mathbf{k}}:\mathbf{k}=(K_1,K_2,\ldots)\in\mathbb{Z}^{(\infty)}\}
\end{equation*}
of non-negative real numbers satisfying
$\sum\limits_{j\in\mathbb{Z}_{\ge 0}}\sum\limits_{\mathbf{k}\in\mathbb{Z}^{(j)}}\alpha_{\mathbf{k}}+\sum\limits_{\mathbf{k}\in\mathbb{Z}^{(\infty)}}\alpha_{\mathbf{k}}=1$, and an element $\mathbf{m}\in\bigcup\limits_{j\in\mathbb{Z}_{\ge 0}\cup\{\infty\}}\mathbb{Z}^{(j)}$ such that for all $g=\mathfrak{t}s\in  G_{\widehat{\mathbf{n}}}$, where $\mathfrak{t}\in\mathcal{F}_{\widehat{\mathbf{n}}}(\mathbb{T}), s\in\mathfrak{S}_{\widehat{\mathbf{n}}}$, the equality
\begin{equation}\label{eq:phi_conv_comb}
\phi_{\mathbf{m}}(g)=\sum_{j\in\mathbb{Z}_{\ge 0}}\sum_{\mathbf{k}\in\mathbb{Z}^{(j)}}\alpha_{\mathbf{k}}\cdot\phi_{\mathbf{k}}(g)+\sum_{\mathbf{k}  \in\mathbb{Z}^{(\infty)}} \alpha_{\mathbf{k}}\cdot\phi_{\mathbf{k}}(g)~\text{holds}.
\end{equation}
Then, $\alpha_{\mathbf{m}}=1$ and $\alpha_{\mathbf{k}}=0$ for all other $\mathbf{k}\neq \mathbf{m}$.
\end{Lm}

\begin{proof}
\textbf{Step 1: fixing the length of tuples.}
Recall that for $\mathbf{m}\in\mathbb{Z}^{(j)}$ we have $\phi_{\mathbf{m}}(s)=\chi_{\nat}(s)^{j}$, and for $\mathbf{m}\in\mathbb{Z}^{(
\infty)}$ we have $\phi_{\mathbf{m}}(g)=\chi_{\reg}(s)\phi_{\mathbf{k}}(\mathfrak{t})$, $g=\mathfrak{t}s$. Then, restricting the equality \eqref{eq:phi_conv_comb} to the subgroup $\mathfrak{S}_{\widehat{\mathbf{n}}}\subset G_{\widehat{\mathbf{n}}}$ yields
\begin{equation*}
\phi_{\mathbf{m}}(s)=\sum_{j\in\mathbb{Z}_{\ge 0}}\gamma_j\cdot\chi_{\nat}(s)^{j}+\gamma_{\infty}\cdot\chi_{\reg}(s),~\text{where}~\gamma_j=\sum_{\mathbf{k}\in\mathbb{Z}^{(j)}}\alpha_{\mathbf{k}},~j\in\mathbb{Z}_{\ge 0}\cup\{\infty\}.
\end{equation*}
Note that $\gamma_j\ge 0$ and $\sum\limits_{j\in\mathbb{Z}_{\ge 0}\cup\{\infty\}}\gamma_j=1$. Now we have two cases:
\begin{itemize}
\item 
$\mathbf{m}\in\mathbb{Z}^{(l)}$ for some $l\in\mathbb{Z}_{\ge 0}$. Then, $\phi_{\mathbf{m}}(s)=\chi_{\nat}(s)^{l}$, and from the classification theorem for indecomposable characters of $\mathfrak{S}_{\widehat{\mathbf{n}}}$ (see Proposition \ref{prop:inf_sym_indec_char}), we obtain that $\gamma_l=1$ and $\gamma_j=0$ for $l=j$. Consequently, $\alpha_{\mathbf{k}}=0$ whenever $\mathbf{k}\notin\mathbb{Z}^{(l)}$. Next we restrict this identity to the subgroup $\mathcal{F}_{\widehat{\mathbf{n}}}(\mathbb{T})\subset G_{\widehat{\mathbf{n}}}$.

\item 
$\mathbf{m}\in\mathbb{Z}^{(\infty)}$. Then, $\phi_{\mathbf{m}}(s)=\chi_{\reg}(s)$ and we similarly obtain that $\gamma_{\infty}=1$ and $\gamma_l=0$ for $l\in\mathbb{Z}_{\ge 0}$, and hence, $\alpha_{\mathbf{k}}=0$ if $\mathbf{k}\notin\mathbb{Z}^{(\infty)}$.
\end{itemize}
Thus, if $\mathbf{m}\in\mathbb{Z}^{(l)}$, where $l\in\mathbb{Z}_{\ge 0}\cup\{\infty\}$, then
\begin{equation}\label{eq:phi_conv_comb_red}
\phi_{\mathbf{m}}(g)=\sum_{\mathbf{k}\in\mathbb{Z}^{(l)}}\alpha_{\mathbf{k}}\cdot\phi_{\mathbf{k}}(g),~g=\mathfrak{t}s\in G_{\widehat{\mathbf{n}}}.
\end{equation}
\textbf{Step 2: extending the characters by continuity.}
Denote by $\mathcal{F}^{L^1}(\mathbb{T})$ the norm closure of $\mathcal{F}_{\widehat{\mathbf{n}}}(\mathbb{T})$ inside $L^1(\mathbb{X}_{\widehat{\mathbf{n}}}, \nu_{\widehat{\mathbf{n}}})$. It is clear that 
\begin{equation*}
\mathcal{F}^{L^1}(\mathbb{T})=\{f\in L^1(\mathbb{X}_{\widehat{\mathbf{n}}}, \nu_{\widehat{\mathbf{n}}}):|f(x)|=1~\text{for almost all}~x\in\mathbb{X}_{\widehat{\mathbf{n}}}\}.
\end{equation*}
Recall that
\begin{equation*}
\phi_{\mathbf{k}}(g)=\prod_{1\le j\le l}\left(\int_{\Fix(s;\mathbb{X}_{\widehat{\mathbf{n}}})}\mathfrak{t}(x)^{K_j}\,d\nu_{\widehat{\mathbf{n}}}(x)\right),~\mathbf{k}=(K_1,K_2,\ldots)\in\mathbb{Z}^{(l)},~l\in\mathbb{Z}_{\ge 0}\cup\{\infty\}.
\end{equation*}
(If $l=\infty$ the product exists but is non-zero only when $s=1_{\mathfrak{S}_{\widehat{\mathbf{n}}}}$.)
We claim that all these characters in fact extend by continuity to $\mathcal{F}^{L^1}(\mathbb{T})$. Indeed, assume that the sequence $\{f_m\}^\infty_{m=1}\subset \mathcal{F}_{\widehat{\mathbf{n}}}(\mathbb{T})$ is a Cauchy sequence in $L^1\left(\mathbb{X}_{\widehat{\mathbf{n}}},\nu_{\widehat{\mathbf{n}}}\right)$, i.e., we have
\begin{equation*}
\lim_{m,n\to \infty}\int_{\mathbb{X}_{\widehat{\mathbf{n}}}}|f_m(x)-f_n(x)|d\nu_{\widehat{\mathbf{n}}}(x)=0.
\end{equation*}
Since for any integer $L$ and $z,w\in\mathbb{T}$ the inequality $|z^L-w^L|\le|L|\cdot|z-w|$ holds, we obtain
\begin{equation*}
\lim_{m,n\to \infty}\left| \int_{\mathbb{X}_{\widehat{\mathbf{n}}}}f_m(x)^L  d\nu_{\widehat{\mathbf{n}}}(x)-\int\limits_{\mathbb{X}_{\widehat{\mathbf{n}}}}f_n(x)^L  d\nu_{\widehat{\mathbf{n}}}(x)\right|=0 ~\text{for any integer}~L.
\end{equation*}
Therefore, every character from \eqref{eq:phi_p_m_def} can be extended by continuity to a character on the group $G^{L^1}$, generated by $\mathfrak{S}_{\widehat{\mathbf{n}}}$ and $\mathcal{F}^{L^1}(\mathbb{T})$. We slightly abuse the notation and denote the extension by the sane symbol $\phi_{\mathbf{k}}$. Then, by continuity, the equality \eqref{eq:phi_conv_comb_red} holds for any $g\in\mathcal{F}^{L^1}(\mathbb{T})$ as well (the series on the right converges absolutely).

\textbf{Step 3: using the multiplicativity.}
For any positive integer $N$ there is a partition (up to null sets) of the Lebesgue space $(\mathbb{X}_{\widehat{\mathbf{n}}},\nu_{\widehat{\mathbf{n}}})$ into $N$ subsets $\{\mathbb{S}_{m,N}\}_{m=1}^{N}$ each of measure $1/N$:
\begin{equation*}
\mathbb{S}_{m,N}=\left\{(x_1,x_2,\ldots)\in\mathbb{X}_{\widehat{\mathbf{n}}}=\prod_{j=1}^{\infty}\mathbb{X}_{n_j}:\sum_{j=1}^{\infty}\frac{x_j-1}{N_j}\in\left[\frac{m-1}{N},\frac{m}{N}\right)\right\},~1\le m\le N.
\end{equation*}
It is not difficult to verify that for $N=N_i$ we have $\mathbb{S}_{m,N_i}=\{m\}\times\prod\limits_{j= i+1}^{\infty}\mathbb{X}_{n_j}$, where $m$ is regarded as an element of $\mathbb{X}_{N_i}\simeq\mathbb{X}_{n_1}\times\ldots\times\mathbb{X}_{n_i}$.
Next, for each positive integer $N$ consider the subgroup of $\mathcal{F}^{L^1}(\mathbb{T})$ consisting of functions that are constant on subsets $\mathbb{S}_{1,N}$, \ldots, $\mathbb{S}_{N,N}$; we denote this subgroup by $\widetilde{\mathcal{F}}_{N}(\mathbb{T})$. Clearly $\widetilde{\mathcal{F}}_{N}(\mathbb{T})$ is isomorphic to $\mathcal{F}_N(\mathbb{T})$ (in fact, these two are same for $N=N_i$). Besides that, for any positive integers $N$ and $M$ the group $\widetilde{\mathcal{F}}_{N}(\mathbb{T})$ embeds into $\widetilde{\mathcal{F}}_{MN}(\mathbb{T})$ in the block-diagonal way.

To complete the proof of the lemma, we apply Lemma \ref{lem:indecomp_mult} (see below) for the set $S=\bigcup_{N\ge 1}\widetilde{\mathcal{F}}_{N}(\mathbb{T})\subset\mathcal{F}^{L^1}(\mathbb{T})$ and functions $\{\phi_{\mathbf{k}}:\mathbf{k}\in\mathbb{Z}^{(l)}\}$ (these are distinct on $S$).

It remains to check the assumptions of Lemma \ref{lem:indecomp_mult}. Pick any positive integer $N$ and element $\mathfrak{t}\in\widetilde{\mathcal{F}}_N(\mathbb{T})$. Then, there exist $t_1,\ldots,t_N\in\mathbb{T}$ such that
\begin{equation*}
\mathfrak{t}(x)=
\begin{cases}
t_i,~\text{if}~x\in\mathbb{S}_{i,N},~i=1,\ldots,N.
\end{cases}
\end{equation*}
We consider the following two functions in $\widetilde{\mathcal{F}}_{N^2}(\mathbb{T})$:
\begin{equation*}
\begin{split}
&\mathfrak{t}_1(x)=
\begin{cases}
t_i,&\text{if}~x\in\bigsqcup\limits_{j=1}^{N}\mathbb{S}_{j+N(i-1),N^2},~i=1,\ldots,N,
\end{cases}
\\
&\mathfrak{t}_2(x)=
\begin{cases}
t_j^{-1},&\text{if}~x\in\bigsqcup\limits_{i=1}^{N}\mathbb{S}_{j+N(i-1),N^2},~j=1,\ldots,N.
\end{cases}
\end{split}
\end{equation*} 
Note that $\mathfrak{t}\in\widetilde{\mathcal{F}}_{N}(\mathbb{T})$ and $\mathfrak{t}_1\in\widetilde{\mathcal{F}}_{N^2}(\mathbb{T})$ are the same as elements in $\mathcal{F}^{L^1}(\mathbb{T})$ because $\mathbb{S}_{i,N}=\bigsqcup_{j=1}^{N}\mathbb{S}_{j+N(i-1),N^2}$.
We claim now that for any (extended) character $\phi_{\mathbf{m}}$, where $\mathbf{m}\in\bigcup_{j\in\mathbb{Z}_{\ge 0}\cup\{\infty\}}\mathbb{Z}^{(j)}$, the following \emph{multiplicativity property} holds:
\begin{equation*}
\phi_{\mathbf{m}}(\mathfrak{t}_1\cdot\mathfrak{t}_2)=\phi_{\mathbf{m}}(\mathfrak{t}_1)\cdot\phi_{\mathbf{m}}(\mathfrak{t}_2).
\end{equation*}
In fact, we also have $\phi_{\mathbf{m}}(\mathfrak{t}_1\cdot\mathfrak{t}_2)=|\phi_{\mathbf{m}}(\mathfrak{t}_1)|^2=|\phi_{\mathbf{m}}(\mathfrak{t}_2)|^2$. Indeed, it suffices to verify that for $\mathbf{m}$ consisting of a single entry $M\in\mathbb{Z}$. In the case we have
\begin{equation*}
\phi_{M}(\mathfrak{t}_1\mathfrak{t}_2)=\int_{\mathbb{X}_{\widehat{\mathbf{n}}}}\mathfrak{t}_1(x)^{M}\mathfrak{t}_2(x)^{M}d\nu_{\widehat{\mathbf{n}}}(x)=\sum_{i,j=1}^{N}\frac{1}{N^2}\cdot t_{i}^Mt_{j}^{-M}=\left(\sum_{i=1}^{N}\frac{1}{N}\cdot t_i^M\right)\left(\sum_{j=1}^{N}\frac{1}{N}\cdot t_j^{-M}\right)=\phi_{M}(\mathfrak{t}_1)\phi_{M}(\mathfrak{t}_2).
\end{equation*}
Since $t_i\in\mathbb{T}$, we also have $\phi_{M}(\mathfrak{t}_1)=\overline{\phi_{M}(\mathfrak{t}_2)}$.
\begin{Rem}
Let us explain the construction of sets $\mathbb{S}_{m,N^2}$ and functions $\mathfrak{t}_1$, $\mathfrak{t}_2$ in a more intuitive way. The Lebesgue space $(\mathbb{X}_{\widehat{\mathbf{n}}},\nu_{\mathbf{\widehat{\mathbf{n}}}})$ can be decomposed as a product space $\mathbb{X}_{N}\times\mathbb{X}_{N}\times \mathbb{X}$, where $\mathbb{X}_{N}=\{1,\ldots,N\}$ is the discrete $N$-point space with uniform measure, and $\mathbb{X}$ is a Lebesgue space. Under this identification the set $\mathbb{S}_{j+N(i-1),N^2}$ is simply $\{j\}\times\{i\}\times\mathbb{X}$. Then, the functions $\mathfrak{t}_1$ and $\mathfrak{t}_2$ can be interpreted as functions $\mathfrak{t}\otimes\mathds{1}_{\mathbb{X}_{N}}\otimes\mathds{1}_{\mathbb{X}}$ and $\mathds{1}_{\mathbb{X}_{N}}\otimes\overline{\mathfrak{t}}\otimes\mathds{1}_{\mathbb{X}}$, respectively, where $\mathfrak{t}\in\mathcal{F}_N(\mathbb{T})$ is given by $\mathfrak{t}(x)=t_x$, $x\in\mathbb{X}_{N}$. In this interpretation the factorization property of the integral above is immediate.
\end{Rem}
Finally, from the above we see that the assumptions of Lemma \ref{lem:indecomp_mult} for functions $\{\phi_{\mathbf{k}}:\mathbf{k}\in\mathbb{Z}^{(l)}\}$ on the set $S=\bigcup_{N\ge 1}\widetilde{\mathcal{F}}_{N}(\mathbb{T})$ are satisfied: for $u=\mathfrak{t}\in\widetilde{\mathcal{F}}_N(\mathbb{T})$ we can put $v_u=\mathfrak{t}_1\mathfrak{t}_2\in\widetilde{\mathcal{F}}_{N^2}(\mathbb{T})$. Therefore, we must have $\alpha_{\mathbf{m}}=1$ and $\alpha_{\mathbf{k}}=0$ for $\mathbf{k}\neq\mathbf{m}$.
\end{proof}

Now let us prove the fact which played the crucial role in the proof of Lemma \ref{lem:indecomp_main}. The statement\footnote{We thank Artem Dudko for bringing this lemma and its proof to our attention.} is quite general and, roughly speaking, it says that ``multiplicativity forces indecomposability''. 

\begin{Lm}\label{lem:indecomp_mult}
Let $\{\psi_i\}_{i=1}^{\infty}$ be a countable set of distinct functions on a set $S$ with values in the closed unit disc $\{z\in\mathbb{C}:|z|\le 1\}$. Assume that for any element $u\in S$ there exists an element $v_u\in S$ such that $\psi_i(v_u)=|\psi_i(u)|^2$ for all $i$. Then, none of the functions $\psi_j$ can be a non-trivial convex combination of the functions $\{\psi_i\}_{i=1}^{\infty}$. In other words, if there exist non-negative real numbers $\lambda_i\in[0,1]$ with $\sum\limits_{i\ge 1}\lambda_i=1$ such that
\begin{equation*}
\psi_j=\sum_{i\ge 1}\lambda_i\cdot \psi_i~\text{as functions on}~S,
\end{equation*}
then $\lambda_j=1$ and $\lambda_i=0$ for $i\neq j$.
\end{Lm}

\begin{proof}
Apply this equality to $v_u$ to obtain
\begin{equation*}
|\psi_j(u)|^2=\psi_j(v_u)=\sum_{i\ge 0}\lambda_i\psi_i(v_u)=\sum_{i\ge 0}\lambda_i|\psi_i(u)|^2.
\end{equation*}
Together with equalities $\psi_j(u)=\sum_{i\ge 0}\lambda_i\psi_i(u)$, $\sum_{i\ge 0}\lambda_i=1$, and the Cauchy-Schwarz inequality this shows that the vectors $(\sqrt{\lambda_i}\psi_i(u))_i$ and $(\sqrt{\lambda_i})_i$ are proportional for any $u$. Since the functions $\psi_i$ ar all distinct, there exists a unique $i_0$ with $\lambda_{i_0}\neq 0$. Consequently, $\lambda_{i_0}=1$ and $i_0=j$.
\end{proof}

\begin{Rem}
In practice, we apply Lemma \ref{lem:indecomp_mult} in the following setup: the set $S$ has a semigroup structure given by multiplication $\otimes$ and an involution $u\mapsto u^*$ such that $\psi_i(u^*)=\overline{\psi_i(u)}$ for all $i$ (often $S$ is the set of conjugacy classes of some big group $G$ and the $(-)^*$ corresponds to the taking the inverse). Then, often one can show that the functions of interest (in particular, $\psi_i$'s) satisfy the \emph{multiplicativity property}: $\psi(u_1\otimes u_2)=\psi(u_1)\psi(u_2)$ for $u_1,u_2\in S$ (for instance, this is usually the case for indecomposable characters of \emph{big} groups). In this case, one sets $v_u=u\otimes u^*$. Then, $\psi(u^*)=\overline{\psi(u)}$ and $\psi(u\otimes u^*)=\psi(u)\psi(u^*)=|\psi(u)|^2$, hence the conditions of the lemma hold.
\end{Rem}

For the proof of Proposition \ref{prop:phi_indecomp} instead of Lemma \ref{lem:indecomp_main} we may also use a different statement, namely Lemma \ref{lem:lin_independence_phi_func} below. It was also used for the proof of indecomposability of spherical functions in Theorem \ref{thm:spherical_func}.

\begin{Lm}\label{lem:lin_independence_phi_func}
Assume that there exist a countable set $\{c_{\mathbf{k}}:\mathbf{k}\in\mathbb{Z}^{(\infty)}\}$ of complex numbers such that
\begin{itemize}
\item
$\sum\limits_{\mathbf{k}\in\mathbb{Z}^{(\infty)}}|c_{\mathbf{k}}|<+\infty$, and

\item
the equality $\sum\limits_{\mathbf{k}\in\mathbb{Z}^{(\infty)}}c_{\mathbf{k}}\cdot\phi_{\mathbf{k}}(\mathfrak{t})=0$ holds for all $\mathfrak{t}\in\mathcal{F}_{\widehat{\mathbf{n}}}(\mathbb{T})$.
\end{itemize}
Then, $c_{\mathbf{k}}=0$ for all $\mathbf{k}$.
\end{Lm}

\begin{proof}
\textbf{Linear independence for finite fixed length.} We start by proving that for any non-negative integer $l$ any finite subset of $\{\phi_{\mathbf{k}}(\mathfrak{t}):\mathbf{k}\in\mathbb{Z}^{(l)}\}$ is linearly independent.

We proceed by induction on $l$. The case $l=0$ is trivial since $\mathbb{Z}^{(0)}$ consists only of the empty tuple $\varnothing$ and the function $\phi_{\varnothing}(\mathfrak{t})\equiv 1$ is non-zero.

Now consider the case $l>0$. Assume that there exist constants $\{a_{\mathbf{k}}:\mathbf{k}\in\mathbb{Z}^{(l)}\}$ not all zero and such that $a_{\mathbf{k}}\neq 0$ only for finitely many $\mathbf{k}\in\mathbb{Z}^{(l)}$ and
\begin{equation*}
\sum_{\mathbf{k}\in\mathbb{Z}^{(l)}}a_{\mathbf{k}}\cdot\phi_{\mathbf{k}}(\mathfrak{t})=0~\text{for all}~\mathfrak{t}\in\mathcal{F}_{\widehat{\mathbf{n}}}(\mathbb{T}).
\end{equation*}
Let $S=\{\mathbf{k}\in\mathbb{Z}^{(l)}:a_{\mathbf{k}}\neq 0\}$, the set $S$ is a finite non-empty subset of $\mathbb{Z}^{(l)}$.
Restricting the equality above to a subgroup $\mathcal{F}_{N}(\mathbb{T})$ for some large $N=N_k>l$, we obtain the following identity: for any $(t_1,\ldots,t_N)\in\mathbb{T}^N\simeq\mathcal{F}_{N}(\mathbb{T})$ we have
\begin{equation}\label{eq:phi_lin_rel_restr}
\sum_{\mathbf{k}=(K_1,\ldots,K_l)\in S}a_{\mathbf{k}}\cdot N^{-l}\prod_{j=1}^{l}\left(t_1^{K_j}+\ldots+t_{N}^{K_j}\right)=0.
\end{equation}
Since the set $S$ is non-empty, there exists an element $\mathbf{m}\in S$ which has the \emph{minimal number of zero entries}. Let $p\in\{1,\ldots,l\}$ be the number of non-zero entries of $\mathbf{k}$. Then, each element $\mathbf{m}\in S$ is a union (in the sense of multisets) of an element in $\mathbb{Z}^{(p)}$ and $l-p$ zeroes. Without loss of generality, we may assume that each $\mathbf{k}=(K_1,\ldots,K_l)\in S$ satisfies $K_{p+1}=\ldots=K_{l}=0$. Therefore, we can rewrite the equality \eqref{eq:phi_lin_rel_restr} as follows:
\begin{equation}
\sum_{\mathbf{k}=(K_1,\ldots,K_l)\in S}a_{\mathbf{k}}\cdot N^{-p}\prod_{j=1}^{p}\left(t_1^{K_j}+\ldots+t_{N}^{K_j}\right)=0,~\text{or}
\end{equation}
\begin{equation*}
\sum_{\mathbf{k}=(K_1,\ldots,K_l)\in S}\sum_{i_1,\ldots,i_p=1}^{N}a_{\mathbf{k}}\cdot t_{i_1}^{K_1}t_{i_2}^{K_2}\ldots t_{i_p}^{K_p}=0
\end{equation*}
Let $\mathbf{m}=(M_1,\ldots,M_p,0,\ldots,0)$ and note that each $M_j$ is a non-zero integer. Now multiply both sides of the last equality by $\overline{t_{1}^{M_1}\ldots t_{p}^{M_p}}$ and integrate over $(t_1,\ldots,t_N)\in\mathbb{T}^{N}$ with respect to the standard Haar measure. Note that distinct monomials in $t_1,\ldots,t_N$ are orthogonal in $L^2(\mathbb{T}^N)$, and a monomial $t_{i_1}^{K_1}\ldots t_{i_{p}}^{K_p}$ can be equal to $t_{1}^{M_1}\ldots t_{p}^{M_p}$ only\footnote{Here we substantially use the fact that $M_1$, \ldots, $M_p$ are all non-zero.} if $(t_{i_1}^{K_1},\ldots,t_{i_p}^{K_p})$ is a permutation of $(t_1^{M_1},\ldots,t_p^{M_p})$, that is, when there exists a permutation $\tau\in\mathfrak{S}_p$ such that
\begin{equation*}
K_1=M_{\tau(1)},\ldots,K_p=M_{\tau(p)}~\text{and}~i_1=\sigma(1),\ldots,i_p=\sigma(p).
\end{equation*}
Note that the number of such permutations $\tau$ is equal to the order of the stabilizer of $(M_1,\ldots,M_p)$ under the $\mathfrak{S}_p$-action.
It follows that only the term corresponding to $\mathbf{k}=\mathbf{m}$ can contribute to the left-hand side after integration. Therefore, we obtain the equality
\begin{equation*}
\sum_{\mathbf{k}=(K_1,\ldots,K_l)\in S}\delta_{\mathbf{k},\mathbf{m}}\cdot a_{\mathbf{k}}\cdot\#\Stab(M_1,\ldots,M_p)=0,
\end{equation*}
which is equivalent to $a_{\mathbf{m}}=0$. However, this contradicts the fact that $\mathbf{m}\in S$. Thus, the set of functions $\{\phi_{\mathbf{k}}(\mathfrak{t}):\mathbf{k}\in\mathbb{Z}^{(l)}\}$ is linearly independent for any $l\ge 0$.

\textbf{General case.} For general case of the lemma we rely on the multiplicativity property of functions $\phi_{\mathbf{k}}$.
Assume that for $\{c_{\mathbf{k}}:\mathbf{k}\in\mathbb{Z}^{(\infty)}\}$ with $\sum_{\mathbf{k}}|c_{\mathbf{k}}|<+\infty$ we have
\begin{equation*}
\sum_{\mathbf{k}\in\mathbb{Z}^{(\infty)}}c_{\mathbf{k}}\cdot\phi_{\mathbf{k}}\equiv 0~\text{on}~\mathcal{F}_{\widehat{\mathbf{n}}}(\mathbb{T}).
\end{equation*}
Note that the series on the left converges absolutely. Fix any element $\mathfrak{t}\in\mathcal{F}_{N_k}(\mathbb{T})$. For any positive integer $m>k$ and element $\mathfrak{r}\in\mathcal{F}_{N_m/N_k}(\mathbb{T})$ consider the element denoted by $\mathfrak{t}\otimes\mathfrak{r}\in\mathcal{F}_{N_m}(\mathbb{T})$ which is defined as follows:
\begin{equation*}
(\mathfrak{t}\otimes\mathfrak{r})(x,y)=\mathfrak{t}(x)\cdot\mathfrak{r}(y),~(x,y)\in\mathbb{X}_{N_k}\times\mathbb{X}_{N_m/N_k}\simeq\mathbb{X}_{N_k}.
\end{equation*}
It is not difficult to check that for all $\mathbf{k}$ we have
\begin{equation*}
\phi_{\mathbf{k}}(\mathfrak{t}\otimes\mathfrak{r})=\phi_{\mathbf{k}}(\mathfrak{t}\otimes\mathds{1}_{N_m/N_k})\cdot \phi_{\mathbf{k}}(\mathds{1}_{N_k}\otimes\mathfrak{r})=\phi_{\mathbf{k}}(\mathfrak{t})\cdot \phi_{\mathbf{k}}(\mathds{1}_{N_k}\otimes\mathfrak{r}),
\end{equation*}
where $\mathds{1}_{N}$ stands for constant one function in $\mathcal{F}_{N}(\mathbb{T})$.
Plugging this into the initial relation yields
\begin{equation*}
\sum_{\mathbf{k}\in\mathbb{Z}^{(\infty)}}c_{\mathbf{k}}\phi_{\mathbf{k}}(\mathds{1}_{N_k}\otimes\mathfrak{r})\cdot\phi_{\mathbf{k}}(\mathfrak{t})\equiv 0~\text{for}~\mathfrak{t}\in\mathcal{F}_{N_k}(\mathbb{T}),~\mathfrak{r}\in\mathcal{F}_{N_m/N_k}(\mathbb{T})
\end{equation*}
Now by varying $m>k$ and $\mathfrak{r}\in\mathcal{F}_{N_m/N_k}(\mathbb{T})$ we will obtain simpler linear relations. Note that for $\mathbf{k}\in\mathbb{Z}^{(\infty)}$ we have
\begin{equation*}
\phi_{\mathbf{k}}(\mathds{1}_{N_k}\otimes\mathfrak{r})=\prod_{j=1}^{l}\left(\sum_{x\in\mathbb{X}_{N_m/N_k}}\frac{1}{|\mathbb{X}_{N_m/N_k}|}\cdot\mathfrak{r}(x)^{K_j}\right),
\end{equation*}
where $K_1,\ldots,K_l$ are the non-zero elements of the multiset $\mathbf{k}$. We may treat $\mathfrak{r}$ as a discrete probability distribution on $\mathbb{T}$, namely, as $\sum_{x\in\mathbb{X}_{N_k}}\frac{1}{N_k}\delta_{\mathfrak{r}(x)}$~-- the average of Dirac delta distributions supported at points $\mathfrak{r}(x)\in\mathbb{T}$, where $x\in\mathbb{X}_{N_k}$. Since $m$ can be taken arbitrarily large and $\sum_{\mathbf{k}}|c_{\mathbf{k}}|<+\infty$, we can apply an approximation argument to obtain that the equality
\begin{equation}\label{eq:lin_comb_measure_par}
\sum_{l\ge 0}\sum_{\substack{\mathbf{k}=(K_1,\ldots,K_l,0,0,\ldots)\\K_1,\ldots,K_l\neq 0}}c_{\mathbf{k}}\left(\prod_{j=1}^{l}\int_{\mathbb{T}}z^{K_j}d\mu(z)\right)\cdot\phi_{\mathbf{k}}(\mathfrak{t})\equiv 0
\end{equation}
holds for any $\mathfrak{t}\in\mathcal{F}_{N_k}(\mathbb{T})$ and any Borel probability measure $\mu$ on $\mathbb{T}$. Now we reduce the general case of the lemma to the first one by choosing a suitable measure $\mu$ that allows us to ``separate'' our infinite sum into finite ones. The specific choice of $\mu$ is given in the following lemma:

\begin{Lm}
For any real $q\in(0,1)$ there exists a Borel probability measure $\mu$ on $\mathbb{T}$ with the following moments:
\begin{equation*}
\int_{\mathbb{T}}z^kd\mu(z)=q^{k^2}~\text{for all}~k\in\mathbb{Z}.
\end{equation*}
\end{Lm}
\begin{proof}
Set $a=-\log q\in(0,+\infty)$ and note that
\begin{equation*}
q^{k^2}=e^{-ak^2}=\frac{1}{\sqrt{2\pi}}\int_{\mathbb{R}}e^{-x^2/2+ik\sqrt{2a}x} dx=\int_{\mathbb{R}}(e^{ix})^{k}\cdot\frac{1}{\sqrt{4\pi a}}e^{-x^2/4a}dx.
\end{equation*}
Therefore, taking the measure $\mu$ on $\mathbb{T}$ with the density given by the formula
\begin{equation*}
d\mu(e^{i\theta})=\frac{1}{\sqrt{4\pi a}}\sum_{n\in\mathbb{Z}}e^{-(\theta+2\pi n)^2/4a}d\theta,~\theta\in\mathbb{R}/2\pi\mathbb{Z}\simeq\mathbb{T}
\end{equation*}
produces the required moments.
\end{proof}

Applying \eqref{eq:lin_comb_measure_par} for the measure $\mu$ from the lemma above, we obtain
\begin{equation*}
\sum_{\mathbf{k}=(K_1,K_2,\ldots)\in\mathbb{Z}^{(\infty)}}c_{\mathbf{k}}q^{\|\mathbf{k}\|^2}\cdot\phi_{\mathbf{k}}(\mathfrak{t})\equiv 0,
\end{equation*}
where $\|\mathbf{k}\|^2=\sum_{j\ge 1}K_j^2$.
Since the sum $\sum_{\mathbf{k}}|c_{\mathbf{k}}|$ is finite, for a fixed $\mathfrak{t}\in\mathcal{F}_{N_k}(\mathbb{T})$ the left-hand side is a power series $F(q)$ which converges absolutely in the disc $\{q:|q|\le 1\}$. However, as explained above, $F(q)$ vanishes whenever $q\in(0,1)$. Therefore, all coefficients of this power series must vanish as well.
Hence, for any non-negative integer $R\ge 0$ we get
\begin{equation*}
\sum_{\mathbf{k}\in\mathbb{Z}^{(\infty)}:~\|\mathbf{k}\|^2=R}c_{\mathbf{k}}\phi_{\mathbf{k}}=0.
\end{equation*}
However, now all these sums have finite number of summands, hence by the first part of the proof all $c_{\mathbf{k}}$ must be zero, as claimed.
\end{proof}

\begin{Rem}
It is a well-known fact that the products of powers sums form an additive basis of the ring of symmetric functions \cite[(2.12)]{Macdonald}.
This can be shown by proving that the power sums $\{p_{\lambda}\}$ can be obtained from the monomial basis $\{m_{\lambda}\}$ via a non-degenerate linear triangular transformation (we use the notation of Macdonald \cite{Macdonald}). Essentially, the first part of the proof of Lemma \ref{lem:lin_independence_phi_func} shows an analogous statement for the Laurent symmetric polynomials in finitely many variables. To state this more precisely, we consider the ring $\mathbb{C}[x_1^{\pm 1},\ldots,x_{N}^{\pm 1}]^{\mathfrak{S}_{N}}$ and define two families of symmetric Laurent polynomials parameterized by pairs of partitions $(\lambda,\mu)$ with $\ell(\lambda)+\ell(\mu)\le N$. Namely, if $k=\ell(\lambda)$, $l=\ell(\mu)$ and $k+l\le N$, then for $\nu=(\lambda_1,\ldots,\lambda_k,0,\ldots,0,-\mu_l,\ldots,-\mu_1)\in\mathbb{Z}^{N}$ we put
\begin{equation*}
\begin{split}
\tilde{p}_{\nu}
&=\prod_{i=1}^{N}(x_1^{\nu_i}+x_2^{\nu_i}+\ldots+x_{N}^{\nu_i}),
\\
\tilde{m}_{\nu}
&=\sum_{[\sigma]\in\mathfrak{S}_{N}/\Stab(\nu)}x_{\sigma(1)}^{\nu_1}x_{\sigma(2)}^{\nu_2}\ldots x_{\sigma(N)}^{\nu_N}.
\end{split}
\end{equation*}
In the course of the proof, we basically proved that $\tilde{p}_{\nu}=c_{\nu,\nu}\tilde{m}_{\nu}+\sum\limits_{\ell_0(\nu')<\ell_0(\nu)}c_{\nu,\nu'}\tilde{m}_{\nu'}$, where $c_{\nu,\nu}\neq 0$ and $\ell_0(\nu')$ denotes the number of non-zero entries in $\nu'$. Since the monomial symmetric Laurent polynomials $\{\tilde{m}_{\nu}\}$ form a basis in $\mathbb{C}[x_1^{\pm 1},\ldots,x_{N}^{\pm 1}]^{\mathfrak{S}_{N}}$ the elements $\tilde{p}_{\nu}$ form a basis as well.
\end{Rem}

\begin{Rem}
In some sense Lemma \ref{lem:lin_independence_phi_func} gives a stronger statement than Lemma \ref{lem:indecomp_main}: it shows that not only the characters of $\phi_{p,\mathbf{m}}$ are in convex position, but they are even linearly independent.
\end{Rem}

Now we are in position to prove the indecomposability of characters $\phi_{p,\mathbf{m}}$.

\begin{proof}[Proof of Proposition \ref{prop:phi_indecomp}]
By Lemma \ref{lem:phi_positive_def} the function $\phi_{p,\mathbf{m}}$, where $\mathbf{m}=\left(M_1,\ldots,M_l\right)\in\mathbb{Z}^{(l)}$, is a normalized character on $G_{\widehat{\mathbf{n}}}$. Suppose that $\phi_{p,\mathbf{m}}$ is not indecomposable, i.e., it is equal to a non-trivial convex combination of some indecomposable characters of $G_{\widehat{\mathbf{n}}}$. Proposition \ref{prop:approx_char} implies that any indecomposable character on $G_{\widehat{\mathbf{n}}}$ must be of the form $\phi_{q,\mathbf{k}}$, where $q\in\{\triv,\sgn,\reg\}$ and $\mathbf{k}=\left( K_1,\ldots,K_j\right)\in\mathbb{Z}^{(j)}$. Therefore, we have an equality of the form
\begin{equation*}
\phi_{p,\mathbf{m}}=\sum_{j\in\mathbb{Z}_{\ge 0}}\sum_{\mathbf{k}\in\mathbb{Z}^{(j)}}\sum_{q\in\{\triv,\sgn\}}\alpha_{q,\mathbf{k}}\cdot\phi_{q,\mathbf{k}}+\sum_{\mathbf{k}\in\mathbb{Z}^{(\infty)}}\alpha_{\reg}\cdot\phi_{\reg,\mathbf{k}},
\end{equation*}
where $\alpha_{q,\mathbf{k}}\in[0,1)$ are such that $\sum_{q,\mathbf{k}}\alpha_{q,\mathbf{k}}=1$ and the summation is over $q\in\{\triv,\sgn\}$ and multisets $\mathbf{k}\in\mathbb{Z}^{(j)}$, $j\ge 0$ and also $q=\reg$, $\mathbf{k}\in\mathbb{Z}^{(\infty)}$ (if the sign character $\sgn_{\infty}$ is trivial, we set $\alpha_{\sgn,\mathbf{k}}=0$). Applying \eqref{eq:phi_m_def} and \eqref{eq:phi_p_m_def}, we rewrite the equality above as follows:
\begin{equation*}
\phi_{p,\mathbf{m}}(g)=\sum_{j\ge 0}\sum_{\mathbf{k}\in\mathbb{Z}^{(j)}}(\alpha_{\triv,\mathbf{k}}+\alpha_{\sgn,\mathbf{k}}\sgn_{\infty}(s))\cdot\phi_{\mathbf{k}}(g)+\sum_{\mathbf{k}\in\mathbb{Z}^{(\infty)}}\alpha_{\reg,\mathbf{k}}\chi_{\reg}(s)\cdot\phi_{\mathbf{k}}(g),~g=\mathfrak{t}s\in G_{\widehat{\mathbf{n}}}.
\end{equation*}
Now we restrict this identity to the subgroup $\mathfrak{S}_{\widehat{\mathbf{n}}}$ of the group $G_{\widehat{\mathbf{n}}}$. Since the restriction of $\phi_{p,\mathbf{k}}$ to $\mathfrak{S}_{\widehat{\mathbf{n}}}$ is given by \eqref{eq:phi_restriction_sym}, and in particular $\phi_{\mathbf{k}}(g)=\chi_{\nat}(s)^{\ell(\mathbf{k})}$ for any $g=s\in\mathfrak{S}_{\widehat{\mathbf{n}}}$, we obtain the following equality of functions on $\mathfrak{S}_{\widehat{\mathbf{n}}}$:
\begin{equation*}
\phi_{\sym}(s;p)\chi_{\nat}(s)^{\ell(\mathbf{m})}=\sum_{j\ge 0}\sum_{\mathbf{k}\in\mathbb{Z}^{(j)}}(\alpha_{\triv,\mathbf{k}}+\alpha_{\sgn,\mathbf{k}}\sgn_{\infty}(s))\chi_{\nat}(s)^{j}+\sum_{\mathbf{k}\in\mathbb{Z}^{(\infty)}}\alpha_{\reg,\mathbf{k}}\chi_{\reg}(s),~s\in \mathfrak{S}_{\widehat{\mathbf{n}}}.
\end{equation*}
The classification of the indecomposable characters of $\mathfrak{S}_{\widehat{\mathbf{n}}}$, see \cite[Theorem 1.2 and Lemma 6.1]{Nessonov_Ngo}, implies that there are two possibilities:
\begin{itemize}
\item
$p=\reg$.

\smallskip

\noindent
Then, $\phi_{\sym}(s;p)\chi_{\nat}(s)^{\ell(\mathbf{m})}=\chi_{\reg}(s)$ and $\alpha_{\triv,\mathbf{k}}=\alpha_{\sgn,\mathbf{k}}=0$ for all $\mathbf{k}\in\bigcup_{l\in\mathbb{Z}_{\ge 0}}\mathbb{Z}^{(l)}$. Hence,
\begin{equation*}
\phi_{p,\mathbf{m}}=\phi_{\reg,\mathbf{m}}=\sum_{\mathbf{k}\in\mathbb{Z}^{(\infty)}}\alpha_{\mathbf{\reg,\mathbf{k}}}\cdot\phi_{\reg,\mathbf{k}}.
\end{equation*}
Now Lemma \ref{lem:indecomp_main} (or Lemma \ref{lem:lin_independence_phi_func}) and the formula $\phi_{\reg,\mathbf{k}}(g)=\chi_{\reg}(s)\phi_{\mathbf{k}}(\mathfrak{t})$ for $g=\mathfrak{t}s\in G_{\widehat{\mathbf{n}}}$ yields $\alpha_{\reg,\mathbf{m}}=1$, which is a contradiction.

\medskip

\item
$p\in\{\triv,\sgn\}$.

\smallskip

\noindent
In this case $\alpha_{q,\mathbf{k}}=0$ unless $q=p$ and $\ell(\mathbf{k})=\ell(\mathbf{m})$. Therefore, we obtain
\begin{equation*}
\phi_{\mathbf{m}}=\sum_{\mathbf{k}\in\mathbb{Z}^{(l)}}\alpha_{p,\mathbf{k}}\cdot\phi_{\mathbf{k}},
\end{equation*}
where $l=\ell(\mathbf{m})$. Applying Lemma \ref{lem:indecomp_main}, we get $\alpha_{p,\mathbf{m}}=1$, which is a contradiction.
\end{itemize}

Thus, $\phi_{p,\mathbf{m}}$ is an indecomposable character of $G_{\widehat{\mathbf{n}}}$ for all choices of $p$ and $\mathbf{m}$.
\end{proof}

\subsection{Proof of the classification of characters for $G_{\widehat{\mathbf{n}}}$.}
Now we are ready to prove the main result.

\begin{proof}[Proof of Theorem \ref{thm:classification_char}]

Let $\chi$ be a normalized indecomposable character of the group $G_{\widehat{\mathbf{n}}}$. By Proposition \ref{prop:approx_char}, $\chi$ is a weak limit of a sequence of irreducible characters $\{\chi_{\pi_{k_l}}\}_{l=1}^{\infty}$, where $\pi_{k_l}$ is an irreducible representation of the group $G_{N_{k_l}}$ and $\{k_l\}_{l=1}^{\infty}$ is an increasing sequence of positive integers. Without loss of generality, we may assume that $k_l=l$ for all $l$.
Then, it follows from Proposition \ref{prop:character_limits} that $\chi=\phi_{p,\mathbf{m}}$ for some $p\in\{\triv,\sgn,\reg\}$ and tuple $\mathbf{m}\in\bigcup_{j\in\mathbb{Z}_{\ge 0}}\mathbb{Z}^{(j)}$ if $p=\triv,\sgn$ or $\mathbf{m}\in\mathbb{Z}^{(\infty)}$ for $p=\reg$. Finally, the characters $\phi_{p,\mathbf{m}}$ are indeed indecomposable by Proposition \ref{prop:phi_indecomp}.
\end{proof}

\section{Realizations of representations}\label{sect:realizations}

In this section we give constructions of unitary representations that realize the characters which we obtained in Theorem \ref{thm:classification_char} as vector states. We use the same approach as in our previous work \cite{Nessonov_Ngo} for $\mathfrak{S}_{\widehat{\mathbf{n}}}$ and refer the reader to \cite[Section 2]{Nessonov_Ngo} for details.

\subsection{An odometer construction.}
Recall that for the sequence $\widehat{\mathbf{n}}=\{n_j\}_{j=1}^{\infty}$ we constructed a measure space $(\mathbb{X}_{\widehat{\mathbf{n}}},\nu_{\widehat{\mathbf{n}}})$ which is defined as the product of finite measure spaces $(\mathbb{X}_{n_k},\nu_{n_k})$ endowed with the uniform measure. We view $\mathbb{X}_{\widehat{\mathbf{n}}}$ as the space of sequences $x=(x_1,\ldots,x_k,\ldots)$, where $x_k\in\mathbb{X}_{n_k}$. The group $\mathfrak{S}_{\widehat{\mathbf{n}}}$ naturally embeds into the group $\Aut_{0}(\mathbb{X}_{\widehat{\mathbf{n}}},\nu_{\widehat{\mathbf{n}}})$ of measure preserving transformations of $(\mathbb{X}_{\widehat{\mathbf{n}}},\nu_{\widehat{\mathbf{n}}})$. Namely, an element $\sigma\in\mathfrak{S}_{N_k}\subset\mathfrak{S}_{\widehat{\mathbf{n}}}$ acts on $(x,y)\in\mathbb{X}_{N_k}\times\prod_{j>k}\mathbb{X}_{n_j}$ as $\sigma((x,y))=(\sigma(x),y)$.

Define an automorphism $O\in \Aut_0\left( \mathbb{X}_{\widehat{\mathbf{n}}},\nu_{\widehat{\mathbf{n}}}\right)$ as follows: set $Ox=( y_1,y_2,\ldots)\in\mathbb{X}_{\widehat{\mathbf{n}}}$, where
\begin{equation*}
y_j=
\begin{cases}
0,& \text{if}~x_i=n_i-1~\text{for all}~i,
\\
x_j+1~(\bmod\,{n_p}),&\text{if}~p\leq\min\{i:x_i<n_i-1\},
\\
x_j,&\text{if}~p>\min\{i:x_i<n_i-1\}.
\end{cases}
\end{equation*}
Then, $O$ is a measure preserving ergodic transformation of $(\mathbb{X}_{\widehat{\mathbf{n}}},\nu_{\widehat{\mathbf{n}}})$ (see \cite[Lemma 2.2]{Nessonov_Ngo}). For any $\sigma\in\mathfrak{S}_{\widehat{\mathbf{n}}}$ there is a measurable function $d(\sigma,-)\colon\mathbb{X}_{\widehat{\mathbf{n}}}\to\mathbb{Z}$ such that
\begin{equation*}
\sigma(x)=O^{d(\sigma,x)}(x)~\text{for all}~x\in\mathbb{X}_{\widehat{\mathbf{n}}}.
\end{equation*}
Note that $d$ satisfies the following cocycle condition: for any $\sigma,\tau\in\mathfrak{S}_{\widehat{\mathbf{n}}}$ one has
\begin{equation}\label{eq:d_cocycle}
d(\sigma\tau,x)=d(\sigma,\tau(x))+d(\tau,x).
\end{equation}

Below we give explicit constructions of unitary representations of the group $G_{\widehat{\mathbf{n}}}$ which realize the characters from the classification in Theorem \ref{thm:classification_char} as vector states of a certain unit vector.

\subsection{Construction for characters without regular part.}

Assume that $p\in\{\triv,\sgn\}$. Consider the Hilbert space $\mathcal{H}_{1}=L^2(\mathbb{X}_{\widehat{\mathbf{n}}},\nu_{\widehat{\mathbf{n}}})\otimes\ell^2(\mathbb{Z})$ viewed as the space of square-integrable functions on $\mathbb{X}_{\widehat{\mathbf{n}}}\times\mathbb{Z}$. For any integer $M$ and any element $g=(\mathfrak{t},\sigma)\in G_{\widehat{\mathbf{n}}}$ define a unitary operator $\mathcal{F}^{(1)}_{M}(g)$ by the following formula\footnote{For $M=0$ the restriction of $\mathcal{F}^{(1)}_{0}$ to $\mathfrak{S}_{\widehat{\mathbf{n}}}$ coincides with the representation $\mathcal{F}$ from \cite[Eq. (2.7)]{Nessonov_Ngo}.}:
\begin{equation*}
\big(\mathcal{F}^{(1)}_{M}(g)\eta\big)(x,m)=\mathfrak{t}(x)^{M}\eta\big(\sigma^{-1}(x),m-d(\sigma^{-1},x)\big).
\end{equation*}
Then, it is not difficult to verify that $\mathcal{F}^{(1)}_{M}\colon G_{\widehat{\mathbf{n}}}\to \mathcal{U}(\mathcal{H}_1)$ is indeed a unitary representation of the group $G_{\widehat{\mathbf{n}}}$. Moreover, a direct calculation shows that for any $g\in G_{\widehat{\mathbf{n}}}$ we have
\begin{equation*}
\langle\mathcal{F}^{(1)}_{M}(g)\xi_{1},\xi_{1}\rangle_{\mathcal{H}}=\phi(g;M),
\end{equation*}
where the unit vector $\xi_{1}\in \mathcal{H}_1$ is defined by the formula
\begin{equation*}
\xi_{1}(x,m)=\delta_{0,m}=
\begin{cases}
0,&~m\ne 0,
\\
1,&~m=0.
\end{cases}
\end{equation*}
As $p\in\{\triv,\sgn\}$, the function $\phi_{\sym}(-;p)\colon\mathfrak{S}_{\widehat{\mathbf{n}}}\to\mathbb{C}$ is a multiplicative character of $\mathfrak{S}_{\widehat{\mathbf{n}}}$ and it lifts to a one-dimensional representation of the group $G_{\widehat{\mathbf{n}}}$. Then, for any tuple $\mathbf{m}=(M_1,\ldots,M_l)\in\mathbb{Z}^{(l)}$ we can define a unitary representation $\mathcal{F}_{p,\mathbf{m}}$ of the group $G_{\widehat{\mathbf{n}}}$ acting in the Hilbert space $\mathcal{H}_{1}^{\otimes l}$ by the following formula:
\begin{equation*}
\mathcal{F}_{p,\mathbf{m}}(g)=\phi_{\sym}(\sigma;p)\otimes\bigotimes_{j=1}^{l}\mathcal{F}^{(1)}_{M_{j}}(g),~g=(\mathfrak{t},\sigma)\in G_{\widehat{\mathbf{n}}}.
\end{equation*}
The vector state corresponding to $\xi_1^{\otimes l}\in\mathcal{H}_1^{\otimes l}$ recovers the character $\phi_{p,\mathbf{m}}$:
\begin{equation}\label{eq:phi_nonreg_matrix_el}
\langle\mathcal{F}_{p,\mathbf{m}}(g)\xi_{1}^{\otimes l},\xi_{1}^{\otimes l}\rangle_{\mathcal{H}_1^{\otimes l}}=\phi_{\sym}(\sigma;p)\phi_{M_1}(g)\ldots\phi_{M_l}(g)=\phi_{p,\mathbf{m}}(g),~g=(\mathfrak{t},\sigma)\in G_{\widehat{\mathbf{n}}}.
\end{equation}

\begin{Rem}
One can prove directly that $\mathcal{F}^{(1)}_{M}$ is a $\mathrm{II}_1$-representation of $G_{\widehat{\mathbf{n}}}$ adapting the proof of \cite[Proposition 2.6]{Nessonov_Ngo}.
\end{Rem}

\subsection{Construction for characters with regular part.} Assume that $p=\reg$.
Consider the Hilbert space $\mathcal{H}_2=L^2(\mathbb{X}_{\widehat{\mathbf{n}}},\nu_{\widehat{\mathbf{n}}})\otimes\ell^2(\mathfrak{S}_{\widehat{\mathbf{n}}})$ which we can regard as the space of square-integrable functions on $\mathbb{X}_{\widehat{\mathbf{n}}}\times\mathfrak{S}_{\widehat{\mathbf{n}}}$. For any tuple  $\mathbf{m}=(M_1,\ldots,M_l)\in\mathbb{Z}^{(l)}$ define the following unitary operator $\mathcal{F}^{(2)}_{M}$ on $\mathcal{H}_2$:
\begin{equation*}
(\mathcal{F}^{(2)}_{M}(g)\eta)(x,\tau)=\mathfrak{t}(x)^{M}\cdot\eta(\sigma^{-1}(x),\sigma^{-1}\tau).
\end{equation*}
Then, we have
\begin{equation*}
\langle\mathcal{F}^{(2)}_{M}\xi_2,\xi_2\rangle_{\mathcal{H}_2}=\phi(g;M)\cdot\chi_{\reg}(\sigma),~g=(\mathfrak{t},\sigma)\in G_{\widehat{\mathbf{n}}},
\end{equation*}
where the unit vector $\xi_2\in\mathcal{H}_2$ is given by
\begin{equation*}
\xi_2(x,\sigma)=
\begin{cases}
1,~&\sigma=1_{\mathfrak{S}_{\widehat{\mathbf{n}}}},\\
0,~&\sigma\neq1_{\mathfrak{S}_{\widehat{\mathbf{n}}}}.
\end{cases}
\end{equation*}
Similar to the first case, we define for a tuple $\mathbf{m}=(M_1,\ldots,M_l)\in\mathbb{Z}^{(l)}$ a unitary representation $\mathcal{F}_{\reg,\mathbf{m}}$ as
\begin{equation*}
\mathcal{F}_{\reg,\mathbf{m}}=\bigotimes_{j=1}^{l}\mathcal{F}^{(2)}_{M_j}
\end{equation*}
Finally, it is clear that the vector state corresponding $\xi_2^{\otimes l}\in\mathcal{H}_2^{\otimes l}$ yields the character $\phi_{\reg,\mathbf{m}}$:
\begin{equation}\label{eq:phi_reg_matrix_el}
\langle\mathcal{F}_{\reg,\mathbf{m}}(g)\xi_2^{\otimes l},\xi_2^{\otimes l}\rangle=\prod_{j=1}^{l}(\phi(g;M_j)\chi_{\reg}(\sigma))=\chi_{\reg}(\sigma)\cdot\prod_{j=1}^l\phi(g;M_j)=\phi_{\reg,\mathbf{m}}(g).
\end{equation}

\begin{Rem}\label{phi_pos_def_rem}
In particular, the identities \eqref{eq:phi_nonreg_matrix_el} and \eqref{eq:phi_reg_matrix_el} imply that $\phi_{p,\mathbf{m}}$ is a positive-definite function on $G_{\widehat{\mathbf{n}}}$. The centrality of $\phi_{p,\mathbf{m}}$ follows directly from the definition of functions $\phi_{\sym}(\sigma;p)$ and $\phi_{\mathbf{m}}(g)$, see \eqref{eq:phi_m_def} and \eqref{eq:phi_p_m_def}.
\end{Rem}

\section{An application: characters of the infinite block-diagonal unitary group}\label{sect:ind_char_unitary}

As was mentioned in the introduction, among the motivations behind the study of characters of the group $G_{\widehat{\mathbf{n}}}$ was obtaining an alternative proof of the description of indecomposable characters on the inductive limit $\operatorname{U}_{\widehat{\mathbf{n}}}=\varinjlim \operatorname{U}_{N_i}$ of the finite-dimensional unitary groups $\{\operatorname{U}_{N_k}\}$ with block-diagonal embeddings. In this section we explain how this classification can be deduced from our Theorems \ref{thm:main_th_intro} and \ref{thm:sph_th_intro}. 

The problem of classifying indecomposable characters of $\operatorname{U}_{\widehat{\mathbf{n}}}$ was addressed previously by Boyer \cite[Proposition 1]{Boyer_1993} for $\widehat{\mathbf{n}}=(2,2,\ldots)$ and Enomoto--Izumi \cite[Theorems 1.2 and 1.5]{Enomoto_Izumi}  (note that \cite
{Enomoto_Izumi} considers a more general situation of unitary groups of arbitrary infinite-dimensional unital simple AF algebras). However, their methods rely on complicated branching rules for unitary groups, whereas we obtain this result from classification of characters of $G_{\widehat{\mathbf{n}}}$ (which in turn was proved by direct analysis of limits of characters).

The connection between the characters of $\operatorname{U}_{\widehat{\mathbf{n}}}$ and $G_{\widehat{\mathbf{n}}}(\mathbb{T})$ relies on the following observation: the group $G_N(\mathbb{T})\simeq\mathbb{T}^{N}\rtimes\mathfrak{S}_{N}$ is the normalizer of the diagonal torus $\mathbb{T}^{N}\subset\operatorname{U}_{N}$. This implies that any character on $\operatorname{U}_{\widehat{\mathbf{n}}}$ is completely determined by its restriction to $G_{\widehat{\mathbf{n}}}(\mathbb{T})$. 

\subsection{Embedding into the hyperfinite $\mathrm{II}_1$ factor.}

We regard the group $\operatorname{U}_{\widehat{\mathbf{n}}}$ as the unitary subgroup of the inductive limit of matrix algebras $\varinjlim\operatorname{Mat}_{N_i}$ (see Section \ref{subsect:block_diag_unitary}). The norm closure of the latter is the corresponding \emph{uniformly hyperfinite (UHF) algebra} $A_{\widehat{\mathbf{n}}}=\overline{\varinjlim \operatorname{Mat}_{N_i}}$.
Denote by $\tr$ the unique faithful tracial state on $A_{\widehat{\mathbf{n}}}$ (for $x\in\mathrm{Mat}_{N_i}$ we have $\tr(x)=\frac{1}{N_i}\Tr(x)$).
Using this trace we define a norm on $A_{\widehat{\mathbf{n}}}$ by the formula
\begin{equation*}
\|x\|_{2}:=\tr(x^*x),~x\in A_{\widehat{\mathbf{n}}}.
\end{equation*}
By using the GNS construction for $A_{\widehat{\mathbf{n}}}$ and $\tr$, we may assume that the algebra $A_{\widehat{\mathbf{n}}}$ is embedded into a hyperfinite $\mathrm{II}_1$-factor $R$. Then, $\varinjlim \operatorname{Mat}_{N_i}$ is a dense subalgebra of $R$ with respect to $\|\cdot\|_2$.  
It is not difficult to verify the following claim:
\begin{Lm}
For every operator $u\in\operatorname{U}(R)$ and any $\epsilon>0$, there exists a unitary $v\in\operatorname{U}_{\widehat{\mathbf{n}}}$ such that
\begin{equation}\label{approx_of_unitary_factor}
\|u-v\|_2= \sqrt{\tr\left((u-v)^*(u-v) \right)}<\epsilon.
\end{equation}
In other words, $\operatorname{U}_{\widehat{\mathbf{n}}}$ is dense in $\operatorname{U}(R)$ with respect to $\|\cdot\|_2$.
\end{Lm}

\subsection{Extension of characters to the unitary group of hyperfinite $\mathrm{II}_1$ factor.}
Let $\chi$ be a character (not necessarily indecomposable) of the group $\operatorname{U}_{\widehat{\mathbf{n}}}$.  Since its restriction to the subgroup $G_{\widehat{\mathbf{n}}}\subset \operatorname{U}_{\widehat{\mathbf{n}}}$ is a normalized character, by Theorem \ref{thm:main_th_intro} and Remark \ref{rem:G_infty_char_factorization}, there exist non-negative numbers $\alpha_{\psi,\mathbf{m}}$, where $\psi\in\exchar(\mathfrak{S}_{\widehat{\mathbf{n}}})$ and $\mathbf{m}=(M_1,\ldots,M_l)\in\mathbb{Z}^{(l)}$ with $M_i\neq 0$ for $l\ge 0$ such that
\begin{equation}\label{eq:chi_formula_exchar_S}
\chi(g)=\sum_{\psi\in\exchar(\mathfrak{S}_{\widehat{\mathbf{n}}})}\sum_{l\ge 0}\,\sum_{\substack{\mathbf{m}=(M_1,\ldots,M_l)\in\mathbb{Z}^{(l)}\\M_i\neq 0}}\alpha_{\psi,\mathbf{m}}\cdot\psi(s)\phi_{\mathbf{m}}(g),~g=\mathfrak{t}s\in G_{\widehat{\mathbf{n}}}.
\end{equation}

By restricting to the toroidal part $\mathcal{F}_{\widehat{\mathbf{n}}}(\mathbb{T})\subset\operatorname{U}_{\widehat{\mathbf{n}}}$, and observing that any element of $\operatorname{U}_{\widehat{\mathbf{n}}}$ is conjugate to an element of $\mathcal{F}_{\widehat{\mathbf{n}}}(\mathbb{T})$, we obtain that for all $u\in\operatorname{U}_{\widehat{\mathbf{n}}}$ we have
\begin{equation}\label{eq:restr_tor_unitary}
\chi(u)=\sum_{l\ge 0}\sum_{\substack{\mathbf{m}=(M_1,\ldots,M_l)\in\mathbb{Z}^{(l)}\\M_i\neq 0}}\gamma_\mathbf{m}\,\prod_{i=1}^{l}\tr\left(u^{M_i}\right),~\text{where}~\gamma_{\mathbf{m}}=\sum_{\psi\in\exchar(\mathfrak{S}_{\widehat{\mathbf{n}}})}\alpha_{\psi,\mathbf{m}} ~~\text{and}~~\sum_{\mathbf{m}} \gamma_\mathbf{m}=1.
\end{equation}    
Applying Lemma \ref{approx_of_unitary_factor} and using the fact that the map $u\mapsto\tr(u^M)$ is continuous with respect to $\|\cdot\|_2$, we can extend the character $\chi$ by continuity to a character of $\operatorname{U}(R)$ using the formula \eqref{eq:restr_tor_unitary}. Since all hyperfinite $\mathrm{II}_1$-factors are isomorphic, one can construct inside $R$ a subalgebra isomorphic to $\varinjlim\operatorname{Mat}_{2^i}$ (note that the restriction of $\tr$ to $\operatorname{Mat}_{2^i}$ is the usual normalized trace). We denote by $\operatorname{U}_{2^{\infty}}$ the unitary subgroup of this subalgebra. Then, it is not difficult to check that the restriction of $\chi$ to $\operatorname{U}_{2^{\infty}}$ is a normalized character.
\begin{Lm}
(a) Any character $\chi$ on $\operatorname{U}_{\widehat{\mathbf{n}}}$ extends continuously to a character on $\operatorname{U}(R)$.
\\
(b) If $\chi$ is an indecomposable character of $\operatorname{U}_{\widehat{\mathbf{n}}}$, then by extending it to $\operatorname{U}(R)$ and restricting to $\operatorname{U}_{2^{\infty}}$, we again get an indecomposable character.
\end{Lm}
\begin{proof}
Part (a) follows from the discussion above. For part (b) assume that $\chi=\alpha\chi_1+(1-\alpha)\chi_2$ for two distinct characters $\chi_1,\chi_2$ on $\operatorname{U}_{2^{\infty}}$ and $\alpha\in(0,1)$. Then, by the same argument as above, $\chi_1$ and $\chi_2$ can be extended by continuity from $\operatorname{U}_{2^{\infty}}$ to $\operatorname{U}(R)$ and then restricted back to $\operatorname{U}_{\widehat{\mathbf{n}}}$. It follows that $\chi_1$ and $\chi_2$ must coincide on $\operatorname{U}_{\widehat{\mathbf{n}}}$, but this forces $\chi_1\equiv\chi_2$ by density and continuity.
\end{proof}

\subsection{Multiplicativity for $\operatorname{U}_{2^{\infty}}$.}
The group $\operatorname{U}_{2^{\infty}}$ is a particular example of the group $\operatorname{U}_{\widehat{\mathbf{n}}}$ corresponding to the homogeneous case when $n_i=2$ for all $i$. We show that in this case indecomposable characters satisfy a certain multiplicativity property. For $u\in\operatorname{U}_N$ and $v\in\operatorname{U}_M$ we denote by $u\otimes v\in\operatorname{U}_{NM}$ the unitary matrix corresponding to the tensor product of unitary operators $u\colon\mathbb{C}^N\to\mathbb{C}^N$ and $v\colon\mathbb{C}^M\to\mathbb{C}^M$. 

\begin{Lm}\label{multiplicativity_lemma}
Assume that $\chi$ is an indecomposable character of $\operatorname{U}_{2^{\infty}}$. Let $u\in \operatorname{U}_{2^k}$ and $v\in I_{2^k}\otimes\operatorname{U}_{2^l}\subset\operatorname{U}_{2^{k+l}}$. Then $\chi(uv)=\chi(u)\chi(v)$.
\end{Lm}

\begin{proof}
The argument is standard. Consider the sequence of elements $u_n=I_{2^n}\otimes u\in\operatorname{U}_{2^{n+k}}$. Clearly, $u_n$ is conjugate to the image of $u\in\operatorname{U}_{2^k}$ in $\operatorname{U}_{2^{n+k}}$, and similarly for $uv$ and $u_nv$ for $n>k+l$. If $(\pi,\mathcal{H},\xi)$ is the GNS representation associated to $\chi$, then the sequence of operators $\pi(u_n)$ must have a weak-$*$ accumulation point $A\in\mathcal{B}(\mathcal{H})$. Since $\pi(u_n)\in\pi(\operatorname{U}_{2^n})'\cap\pi(\operatorname{U}_{2^{n+k}})''$ for all $n$ and $\pi(\operatorname{U}_{2^{\infty}})''$ is a factor, the operator $A$ must be scalar. However, $\langle\pi(u_n)\xi,\xi\rangle=\chi(u_n)=\chi(u)$. Consequently, $A=\chi(u)\cdot I_{\mathcal{H}}$. It remains to notice that
\begin{equation*}
\chi(uv)=\chi(u_nv)=\langle\pi(u_nv)\xi,\xi\rangle\to\langle A\pi(v)\xi,\xi\rangle=\chi(u)\langle\pi(v)\xi,\xi\rangle=\chi(u)\chi(v),~n\to\infty,
\end{equation*}
hence $\chi(uv)=\chi(u)\chi(v)$, as claimed.
\end{proof}

From now on we assume that $\chi$ is an indecomposable character of $\operatorname{U}_{\widehat{\mathbf{n}}}$.

\begin{Lm}\label{lem:chi_product_tr}
There exists a unique $\mathbf{m}=\mathbf{m}_0$ in the sum \eqref{eq:restr_tor_unitary} such that $\gamma_{\mathbf{m}_0}>0$. In particular, $\gamma_{\mathbf{m}_0}=1$ and
\begin{equation}\label{eq:prod_traces}
\chi(u)=\prod_{i=1}^{l}\tr\left(u^{M_i}\right)~\text{for all}~ u\in\operatorname{U}(R),
\end{equation}
where $\mathbf{m}_{0}=(M_1,\ldots,M_l)\in\mathbb{Z}^{(l)}$ with $M_1,\ldots,M_l\neq 0$.
\end{Lm}

\begin{proof}
Restrict the equality \eqref{eq:restr_tor_unitary} to $\operatorname{U}_{2^{\infty}}\subset\operatorname{U}(R)$ (this basically reduces the lemma to the case when $n_i=2$ for all $i$). Then, for any $u\in\operatorname{U}_{2^k}\subset\operatorname{U}_{2^{\infty}}$ we have an element $u\otimes u^*\in\operatorname{U}_{2^{k+1}}\subset\operatorname{U}_{2^{\infty}}$ which satisfies equalities
\begin{equation*}
\tr((u\otimes u^*)^{M})=|\tr(u^M)|^2~\text{for all}~M\in\mathbb{Z}~\text{and}~\chi(u\otimes u^*)=|\chi(u)|^2
\end{equation*}
by definition of $\tr$ and previous lemma.
Therefore, we may apply Lemma \ref{lem:indecomp_mult} and the claim follows.
%We apply the previous lemma to elements $u,v\in G_{2^{\infty}}(\mathbb{T})$.
\end{proof}

\subsection{Proof of the classification.} Now we are ready to complete the proof.

\begin{Th}\label{thm:indec_char_inf_unitary}
The elements of the multiset $\mathbf{m}_{0}=(M_1, \ldots,M_l)$ from Lemma \ref{lem:chi_product_tr} belong to the set $\{-1,1\}$. Consequently, for any indecomposable character $\chi$ on $\operatorname{U}_{\widehat{\mathbf{n}}}$ there exist non-negative integers $p$ and $q$ such that $\chi(u)=\tr(u)^p\tr(u^*)^q$ for any $u\in\operatorname{U}_{\widehat{\mathbf{n}}}$. Moreover, for any non-negative integers $p$, $q$ the function $\tr^p\overline{\tr}^q$ is indeed an indecomposable character of $\operatorname{U}_{\widehat{\mathbf{n}}}$.
\end{Th}

\begin{proof}
Our aim is to show that $M_i=\pm 1$ for each $i$. Since $\gamma_{\mathbf{m}}=0$ for $\mathbf{m}\neq\mathbf{m}_0$ we have by \eqref{eq:chi_formula_exchar_S}
\begin{equation*}
\chi(g)=\psi_0(s)\cdot\phi_{\mathbf{m}_0}(g),~\text{where}~\psi_0(s)=\sum_{\psi\in\exchar(\mathfrak{S}_{\widehat{\mathbf{n}}})}\alpha_{\psi,\mathbf{m}_0}\cdot\psi(s).
\end{equation*}
Observe that $\psi(s)\in[-1,1]$ for all $\psi$ and $s$. Thus, $\chi(s)=\psi_0(s)\phi_{\mathbf{m}_0}(s)=\psi_0(s)\chi_{\nat}(s)^{l}\le\chi_{\nat}(s)^l$.

Restrict the formula \eqref{eq:prod_traces} above to a permutation $\sigma\in\mathfrak{S}_{N_i}$ for a sufficiently large $i$. Then, we obtain
\begin{equation*}
\chi(\sigma)=\prod_{i=1}^{l}\tr(\sigma^{M_i})=\prod_{i=1}^{l}\chi_{\nat}(\sigma^{M_i}).
\end{equation*}
On the other hand, from the discussion above we have $\phi_{\mathbf{m}_0}(\sigma)\le\chi_{\nat}(\sigma)^{l}$. Thus, we obtain the inequality
\begin{equation}\label{eq:chi_nat_sigma_ineq}
\chi_{\nat}(\sigma^{M_1})\ldots\chi_{\nat}(\sigma^{M_l})\le\chi_{\nat}(\sigma)^l~\text{for all}~\sigma\in\mathfrak{S}_{N_i}.
\end{equation}
Now suppose that $M_j\neq\pm 1$ for some $j$. Then, there exists a prime number $p$ dividing $M_j$. Assuming that $N_i>lp$, we can construct a permutation $\sigma_0\in\mathfrak{S}_{N_i}$ whose cycle decomposition consists of $l$ independent cycles of length $p$. Then, it is not difficult to check that for $M\in\mathbb{Z}$ we have
\begin{equation*}
\chi_{\nat}(\sigma_0^{M})=\frac{1}{N_i}\#\Fix(\sigma_0^{M};\mathbb{X}_{N_i})=
\begin{cases}
1-\frac{lp}{N_i},&~p\nmid M,
\\
1,&~p\mid M.
\end{cases}
\end{equation*}
It follows that for $\sigma=\sigma_0$ the inequality \eqref{eq:chi_nat_sigma_ineq} yields
\begin{equation*}
(1-lp/N_i)^{l}=\chi_{\nat}(\sigma_0)^{l}\ge\chi_{\nat}(\sigma_0^{M_1})\ldots\chi_{\nat}(\sigma_0^{M_l})\ge(1-lp/N_i)^{l-1},
\end{equation*}
which is clearly a contradiction since $1-lp/N_i\in(0,1)$. Thus, $M_i\in\{\pm1\}$ for each $i$ and hence, $\chi(u)=\tr(u)^p\tr(u^*)^q$ for some $p,q\in\mathbb{Z}_{\ge0}$.

Finally, to prove that the function $\tr^p\overline{\tr}^q$ is a character we note that $\tr$ is a normalized central positive-definite function on $\operatorname{U}_{\widehat{\mathbf{n}}}$. The indecomposability of $\tr^p\overline{\tr}^q$ can be proved using the multiplicativity property using the argument from the proof of Lemma \ref{lem:chi_product_tr}.
\end{proof}

\begin{Rem}
An alternative, but more complicated proof of the fact that $M_i\in\{\pm 1\}$ uses the restriction to finite-dimensional subgroups $\operatorname{U}_{N_k}$, the Schur positivity and certain asymptotic results on the multiplicities of irreducibles in tensor products of representations of $\operatorname{U}_{N_k}$.
\end{Rem}

\section{Discussion of the general case of $G_{\hat{\mathbf{n}}}(A)$}\label{sect:gen_ab_group}

Fix an arbitrary compact abelian group $A$. Denote by $\widehat{A}$ the \emph{Pontryagin dual group} for $A$. It is known that $\widehat{A}$ is a discrete abelian group.
Recall that the group $G_{\widehat{\mathbf{n}}}(A)$ is constructed as the inductive limit of semidirect products $\mathcal{F}_{N_k}(A)\rtimes\mathfrak{S}_{N_k}\simeq A^{N_k}\rtimes\mathfrak{S}_{N_k}$ under the block-diagonal embeddings. Note that $\mathcal{F}_{N_k}(A)$ is a compact abelian group and we identify its Pontryagin dual group with $\mathcal{F}_{N_k}(\widehat{A})$. We use the multiplicative notation for these groups.

In previous sections we proved the classification theorem for the indecomposable characters of the group $G_{\widehat{\mathbf{n}}}(\mathbb{T})$. It turns out that many proof ideas apply for more general abelian groups $A$. In particular, we believe that a modification of our methods can be used to deduce the classification of indecomposable characters of the group $G_{\widehat{\mathbf{n}}}(A)$ when $A=A_0\times\mathbb{T}^d$, where $A_0$ is a finite abelian group and $d\in\mathbb{Z}_{\ge 0}$. We plan to explain the details in future work.

\textbf{Expectations.} Based on preliminary calculations and the answer for $A=\mathbb{T}$, it seems plausible to us that there will be a family indecomposable characters of $G_{\widehat{\mathbf{n}}}(A)$ of the form
\begin{equation*}
\chi(\mathfrak{a},\sigma)=\phi_{A}(\mathfrak{a})\phi_{\sym}(\sigma;p)\prod_{i=1}^{l}\left(\int_{\Fix(\sigma;\mathbb{X}_{\widehat{\mathbf{n}}})}\alpha_i(\mathfrak{a}(x))d\nu_{\widehat{\mathbf{n}}}(x)\right),~(\mathfrak{a},\sigma)\in G_{\widehat{\mathbf{n}}}(A),
\end{equation*}
where $\phi_{A}$ and $\phi_{\sym}(\sigma;p)$ are one-dimensional characters of groups $\mathcal{F}_{\widehat{\mathbf{n}}}(A)$ and $\mathfrak{S}_{\widehat{\mathbf{n}}}$, respectively. However, we believe that for general $A$ the structure of the classification will be different from the case $A=\mathbb{T}$. In particular, there will be more indecomposable characters especially when $A$ is discrete. Moreover, the interplay between the arithmetic properties of $\widehat{\mathbf{n}}=\{n_k\}_{k=1}^{\infty}$ and the structure of the torsion part of $A$ might affect the classification. We describe below some examples of these phenomena when $A$ is a finite abelian group.

\subsection{Some examples.} We consider the case of cyclic group: let $A=\{e^{2\pi i k/r}:0\le k\le r-1\}\simeq\mathbb{Z}/r\mathbb{Z}$ for some positive integer $r>1$. In all examples below the characters will be supported on the normal abelian subgroup $\mathcal{F}_{\widehat{\mathbf{n}}}(A)$ of $G_{\widehat{\mathbf{n}}}(A)$.

\begin{Exa}[Regular character]
The group $G_{\widehat{\mathbf{n}}}(A)$ has non-trivial center isomorphic to $A$. One can use this to construct characters as follows: for $\alpha\in \widehat{A}$ define
\begin{equation*}
\chi_{\alpha}(\mathfrak{a})=
\begin{cases}
\alpha(a),~&\text{if}~\mathfrak{a}(x)\equiv a~\text{for some}~a\in A,
\\
0,~&\text{otherwise}.
\end{cases}
\end{equation*}
It is a straightforward to verify that $\chi_{\alpha}$ is indeed a character. Note the continuity holds automatically since the group $G_{\widehat{\mathbf{n}}}(A)$ is discrete. In particular, for $r=2$ note that the convex combination $\frac{1}{2}(\chi_{+1}+\chi_{-1})$ is equal to the regular character of $G_{\widehat{\mathbf{n}}}(A)$ (similarly for $r>2$). Therefore, the regular character is not indecomposable in this case (the group is also not ICC since it has a non-trivial center) and we also get $|\widehat{A}|=r$ characters supported on the center of $G_{\widehat{\mathbf{n}}}(A)$.
\end{Exa}

\begin{Exa}[Multiplicative characters]
One can also have some non-trivial multiplicative characters coming from $\mathcal{F}_{\widehat{\mathbf{n}}}(A)$. Namely, for an element $\alpha\in\widehat{A}$ we may define
\begin{equation*}
\chi(\mathfrak{a})=\lim_{k\to\infty}\alpha\Bigg(\prod_{x\in\mathbb{X}_{N_l}}\mathfrak{a}(x)\Bigg)^{N_k/N_l},~\mathfrak{a}\in\mathcal{F}_{N_l}(A).
\end{equation*}
Of course, in order for this to make sense, we have to impose certain arithmetic conditions on $\alpha$ and $\widehat{\mathbf{n}}$. In the case when $A$ is a cyclic group of order $r=2$, we may assume that the sequence of $\{n_i\}$ consists only of odd integers. For other cyclic groups the situation can be more complicated: for instance, if $r=6$, we may take the sequence $\widehat{\mathbf{n}}=(3,5,7,3,5,7,\ldots)$, but not $\widehat{\mathbf{n}}=(5,7,5,7,5,7,\ldots)$.
\end{Exa}
 
\begin{Exa}[Intermediate situations]
We also can get a ``mixture'' of the two examples above. Assume that $r=4$. Then, $A$ has a non-trivial subgroup $B$ isomorphic to $\mathbb{Z}/2\mathbb{Z}$, let $\beta\in\widehat{A}$ be such that $B=\ker\beta$. Define the character $\chi$ as follows: for $\mathfrak{a}\in\mathcal{F}_{N_k}(A)$ put
\begin{equation*}
\chi(\mathfrak{a})=
\begin{cases}
\beta(a)^{N_k},~&\text{if}~\mathfrak{a}(x)/a\in B~\text{for all}~x\in \mathbb{X}_{N_k}~\text{for some}~a\in A,
\\
0,~&\text{otherwise}.
\end{cases}
\end{equation*}
The character $\chi$ is supported on the preimage of the center of $\mathcal{F}_{\widehat{\mathbf{n}}}(A/B)$ under the natural map $\mathcal{F}_{\widehat{\mathbf{n}}}(A)\to \mathcal{F}_{\widehat{\mathbf{n}}}(A/B)$ (similar to the first example). On this support the formula for $\chi$ is basically the pullback of the character of $\mathcal{F}_{\widehat{\mathbf{n}}}(A/B)$ from the second example.
One may generalize this construction to any subgroup $B$ of a finite abelian group $A$.
\end{Exa}

Finally, let us note that when $A$ is not discrete, the continuity imposes strong restrictions on characters. In particular, there no analogues of the examples above in the case of $A=\mathbb{T}$.

\appendix

\section{Supplementary material}\label{sect:appendix}

\subsection{Characters and representations of the symmetric group.}

In this subsection we give a brief overview of the Okounkov--Vershik approach to the representation theory of the symmetric group.

\subsubsection{Basic definitions.}
Fix a positive integer $N$. The symmetric group $\mathfrak{S}_{N}$ is the group of permutations of the set $\mathbb{X}_{N}=\{1,2,\ldots,N\}$. For any $\sigma\in\mathfrak{S}_{N}$ we denote by $c(\sigma)=(c_1(\sigma),c_2(\sigma),\ldots)$ its cycle type, i.e., $c_k(\sigma)$ is the number of $k$-cycles in the cycle decomposition of $\sigma$. Note that $\sum_{k}kc_k(\sigma)=N$. It is clear that cycle types are in one-to-one correspondence with conjugacy classes of $\mathfrak{S}_{N}$.

It is well-known that $\mathfrak{S}_{N}$ has the following presentation in terms of the \emph{Coxeter generators} $s_i=(i~i+1)$, $i\in\{1,\ldots,N-1\}$:
\begin{equation*}
\mathfrak{S}_{N}=\langle s_1,\ldots,s_{N-1}\mid s_i^2=1,i\in\{1,\ldots,N-1\};s_{i}s_{i+1}s_{i}=s_{i+1}s_{i}s_{i+1},i\in\{1,\ldots,N-2\}\rangle.
\end{equation*}

A \emph{partition} of $N$ is a non-increasing sequence $\lambda=(\lambda_1,\ldots,\lambda_l)$ of positive integers such that $N=\lambda_1+\ldots+\lambda_l$. For such a partition we denote $|\lambda|=N$ and $\ell(\lambda)=l$. If $\lambda$ is a partition of $N$, we write $\lambda\vdash N$.

It is well known that the irreducible representations of $\mathfrak{S}_{N}$ are parameterized by partitions $\lambda$ of $N$. We denote the corresponding representation by $\rho_{\lambda}$ and its \emph{normalized character}\footnote{Note that this is not a conventional notation, but nevertheless we use it here for technical reasons.} by $\chi_{\lambda}$. Note that the conjugacy classes of $\mathfrak{S}_{N}$ are also parameterized by partitions of $N$ (we regard the cycle type $c$ of a permutation as a partition $1^{c_1}2^{c_2}\ldots$). We denote by $\chi_{\lambda}^{\mu}$ the value of $\chi_{\lambda}$ on conjugacy class that corresponds to $\mu\vdash N$ (recall that characters are central functions on $\mathfrak{S}_{N}$).

The \emph{Young diagram} associated to a partition $\lambda=(\lambda_1,\ldots,\lambda_l)\vdash N$ is a collection of $N$ boxes arranged in left-justified rows such that the $i$-th row consists of $\lambda_i$ boxes (see Figure \ref{fig:young_diagram}). Denote the set of Young diagrams with $N$ boxed by $\mathbb{Y}_{N}$ and let $\mathbb{Y}=\bigcup_{N\ge 0}\mathbb{Y}_{N}$. We often identify a partition with its Young diagram.

\begin{Exa}
For instance, the trivial partition $\lambda=(n)$ corresponds to the trivial one-dimensional representation of $\mathfrak{S}_n$, while the conjugate partition $\lambda'=(1^n)$ corresponds to the sign representation of $\mathfrak{S}_n$.
\end{Exa}

Flipping a Young diagram $\lambda$ over its main diagonal gives another Young diagram which we denote by $\lambda'$. The corresponding partition $\lambda'$ is called the \emph{conjugate} partition (see an example on Figure \ref{fig:young_diagram}).

\begin{figure}[h!]
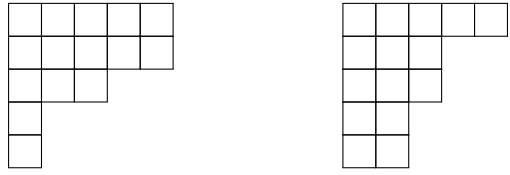

\centering
\ytableausetup{boxsize=1.2em}
\begin{ytableau}
 ~ & ~ & ~ & ~ & ~ \\
 ~ & ~ & ~ & ~ & ~ \\
 ~ & ~ & ~ \\
 ~ \\
 ~
\end{ytableau}
\hspace{20mm}
\begin{ytableau}
 ~ & ~ & ~ & ~ & ~ \\
 ~ & ~ & ~ \\
 ~ & ~ & ~ \\
 ~ & ~ \\
 ~ & ~
\end{ytableau}
\caption{Young diagrams of $\lambda=(5,5,3,1,1)$ and $\lambda'=(5,3,3,2,2)$.}
\label{fig:young_diagram}
\end{figure}
A \emph{(standard) Young tableau} of shape $\lambda\vdash N$ is a filling of boxes of the corresponding Young diagram with numbers $1,2,\ldots,N$ such that the entries increase across each row and down each column. The set of all Young tableaux of shape $\lambda$ is denoted by $\Tab(\lambda)$. It turns out that the dimension of the irreducible representation $\rho_{\lambda}$ is equal to the cardinality of $\Tab(\lambda)$ (see Subsection \ref{sect:sym_irrep}). In view of this fact, we use the notation $\dim\lambda=\#\Tab(\lambda)$.

For a given Young diagram $\lambda$ and a box $\square\in\lambda$ located in the $p$-th row and the $q$-th column of $\lambda$, the number $\mathrm{cont}(\square)=q-p$ is called the \emph{content} of the box $\square$. For a Young tableau $T\in\Tab(\lambda)$ and $i\in\{1,2,\ldots,|\lambda|=N\}$ we use the notation $a_i=a_i(T)$ to denote the content of the box of $T$ which contains $i$.

\begin{Exa}
Consider the Young tableau of shape $\lambda=(4,2,1)$ depicted in Figure \ref{fig:young_tableau} on the left. Then, the contents of its entries are as follows: $(a_1,a_2,a_3,a_4,a_5,a_6,a_7)=(0,1,-1,-2,2,0,3)$.
\end{Exa}
\begin{figure}[h!]
\centering
\ytableausetup{boxsize=1.8em}
\begin{ytableau}
 1 & 2 & 5 & 7 \\
 3 & 6  \\
 4
\end{ytableau}
\hspace{20mm}
\begin{ytableau}
 0 & 1 & 2 & 3 \\
 -1 & 0  \\
 -2
\end{ytableau}
\caption{A standard tableau of shape $\lambda=(4,2,1)$ and contents of its cells.}
\label{fig:young_tableau}
\end{figure}

\begin{Def}\label{def:minimal_el_conj_cl}
Let $c=(c_1,c_2,\ldots)$ be a cycle type in $\mathfrak{S}_{N}$. Let $m_1\le m_2\le\ldots\le m_r$, where $r=c_2+c_3+\ldots$, be the sequence
\begin{equation*}
\underbrace{2,\ldots,2}_{c_2},\underbrace{3,\ldots,3}_{c_3},\ldots,\underbrace{i,\ldots,i}_{c_i},\ldots.
\end{equation*}
Then, we define the \emph{minimal element} in the conjugacy class of $\mathfrak{S}_{N}$ corresponding to $c$ to be the permutation $\sigma_{c}$ given by the following formula:
\begin{equation*}
\sigma_c=\prod_{i=1,\ldots,r}(m_1+\ldots+m_{i-1}+1;m_1+\ldots+m_{i-1}+2;\ldots;m_1+\ldots+m_{i-1}+m_{i}).
\end{equation*}
In particular, $\sigma_c$ is a product of cycles of lengths $m_1,m_2,\ldots,m_r$, and $\supp(\sigma_c)=\{1,\ldots,N-c_1\}$.
\end{Def}

The main property of the minimal element $\sigma_{c}$ is that it can be represented as a product
\begin{equation*}
\sigma_c=s_{i_1}\ldots s_{i_r},
\end{equation*}
where $i_1<i_2<\ldots<i_r$ are elements of $\{1,\ldots,N-c_1-1\}$. This is a consequence of the definition and the identity $(i~i+1~\ldots~j)=s_{i}s_{i+1}\ldots s_{j}$.

\begin{Exa}
The minimal element of the conjugacy class corresponding to the cycle type $c=(3,1,2)$ in $\mathfrak{S}_{11}$ is $\sigma_c=(1~2)(3~4~5)(6~7~8)(9)(10)(11)$.
\end{Exa}

\subsubsection{Realizations of the irreducible representations of $\mathfrak{S}_{N}$}\label{sect:sym_irrep}
In this subsection we recall the explicit constructions of the irreducible representations of the symmetric group $\mathfrak{S}_{N}$ (we refer the reader to \cite{CST} and \cite{Okounkov-Vershik} for more details).

The unitary irreducible representations $(\rho_\lambda,V_{\lambda})$ of $\mathfrak{S}_{N}$ are labeled by partitions $\lambda\vdash N$. The dimension of $\rho_{\lambda}$ equals $\dim\lambda=\#\Tab(\lambda)$ and there exists an orthonormal basis $\left\{ v_T\right\}_{T\in\Tab(\lambda)}$ in $V_{\lambda}$ parameterized by the standard Young tableaux $T\in\Tab(\lambda)$ such that the Coxeter generators $s_i=(i\;i+1)$ acts on elements of this basis  as follows:
\begin{itemize}
\item
if $i$ and $i+1$ are in the same row of $T$, then $\rho_\lambda(s_i)v_T=v_T$;
	
\item
if $i$ and $i+1$ are in the same column of $T$, then $\rho_\lambda(s_i)v_T=-v_T$;
	
\item
if $i$ and $i+1$ are in different rows and columns in $T$, then swapping positions of elements $i$ and $i+1$ in $T$ also gives some tableau $T'$.
Then, the operator $\rho_\lambda(s_i)$ in the two-dimensional space $\operatorname{span}\{v_{T},v_{T'}\}$ acts on the basis $\{v_{T},v_{T'}\}$ via the following matrix:
\begin{equation}\label{matrix_generator}
\begin{pmatrix}
1/d & \sqrt{1-1/d^2}
\\
\sqrt{1-1/d^2} & -1/d
\end{pmatrix},
\end{equation}
where $d=a_{i+1}-a_{i}$ and $a_i$ is the content of the box of $T$ that contain $i$.
\end{itemize}
This construction implies the following statement.
\begin{Lm}\label{sign_lemma}
Put $\widehat{\rho}_\lambda(\sigma)=\sgn(\sigma)\rho_\lambda(\sigma)$ for $\sigma\in\mathfrak{S}_N$. Then, the representations $\widehat{\rho}_\lambda$ and $\rho_{\lambda'}$ of the group $\mathfrak{S}_{N}$ are unitarily equivalent.
\end{Lm}
Define the normalized character of the irreducible representation $\rho_\lambda$ by the formula:
\begin{equation}\label{character_of_finite_Gr}
\chi_{\lambda}(\sigma)=\frac{\Tr\left( \rho_\lambda(\sigma) \right)}{\Tr\left( \rho_\lambda(1_{\mathfrak{S}_N})\right)}=\frac{1}{\dim\lambda}\Tr(\rho_{\lambda}(\sigma)).
\end{equation}
where $\Tr$ is the standard trace of a finite-dimensional operator, $\sigma\in\mathfrak{S}_N$ and $1_{\mathfrak{S}_N}$ is the identity element of $\mathfrak{S}_N$. Note that Lemma \ref{sign_lemma} implies the equality
\begin{equation}\label{sign_char}
\chi_{\lambda'}(\sigma)=\sgn(\sigma)\chi_{\lambda}(\sigma)
\end{equation}
for all $\sigma\in\mathfrak{S}_{N}$ and $\lambda\vdash N$. Here $\sgn(\sigma)=\sgn_{N}(\sigma)$ denotes the sign of permutation $\sigma$.

We also use the following formula for diagonal matrix elements of $\rho_{\lambda}$.
\begin{Lm}[{\cite[Lemma 3.4]{Nessonov_Ngo}}]\label{lem:matrix_elem_formula}
Suppose that a permutation $\sigma\in\mathfrak{S}_{N}$ is expressed as a product of distinct Coxeter generators:
\begin{equation*}
\sigma=s_{i_1}\ldots s_{i_r}~\text{where}~i_1<i_2<\ldots<i_r.
\end{equation*}
Then for any tableau $T$ the following equality holds:
\begin{equation}\label{matrix_elem_equ}
|(\rho_\lambda(\sigma) v_T,v_T)|=\prod_{j}\frac{1}{|a_{i_j+1}-a_{i_j}|}.
\end{equation}
Here, the product is taken over all indices $j$ for which the numbers $i_j$ and $i_j+1$ are in different rows and different columns of tableau $T$ and $a_i$ denotes the content of the box in $T$ that contains $i$.
\end{Lm}

\subsubsection{Proof of Lemma \ref{lem:sym_char_lim}}

In the proof of Lemma \ref{lem:sym_char_lim} we use the following upper bound for the characters of the symmetric group which is due to Roichman (see \cite{Roichman} and also \cite{Feray_Sniady}).

\begin{Prop}[{\cite[Theorem 1]{Roichman}}]\label{char_bound}
There exist absolute constants $a\in(0,1)$ and $b>0$ such that for any Young diagram $\lambda$ with $N=|\lambda|> 4$ boxes and for any $\sigma\in\mathfrak{S}_N$ the following inequality holds:
\begin{equation}
\left|\chi_{\lambda}(\sigma)\right|\leq
\left(\max\left\{\frac{\lambda_1}{N},\frac{\lambda_1'}{N},a\right\}\right)^{b\cdot\#\supp_{N}(\sigma)}.
\end{equation}
Here $\lambda_1$ (respectively, $\lambda_1'$) is the number of boxes in the first row (respectively, column) of the diagram $\lambda$.
\end{Prop}

\begin{Lm}[{cf. \cite[Section 5]{Nessonov_Ngo}}]\label{lem:sym_char_lim}
Fix a non-negative integer $d$ and non-negative rational numbers $\alpha_1,\ldots,\alpha_r$ such that $\alpha_1+2\alpha_2+\ldots+r\alpha_r=1$. For each $k$ define
\begin{equation*}
\gamma(k)=(\alpha_1N_k-d,\alpha_2N_k,\ldots,\alpha_r N_k).
\end{equation*}
Assume that for all sufficiently large $k$ all entries of $\gamma(k)$ are non-negative integers, i.e., $\gamma(k)$ is a cycle type in $\mathfrak{S}_{N_k-d}$. Let $\{\lambda(k)\vdash N_k-d\}_{k=1}^{\infty}$ be a sequence of partitions. Consider the sequence of values of normalized characters $\{\chi_{\lambda(k)}^{\gamma(k)}\}_{k=1}^{\infty}$. Then,
\begin{itemize}
\item[(1)]
if $\limsup\limits_{k\to\infty}\min\{|\lambda(k)/(\lambda_1(k))|,|\lambda'(k)/(\lambda'_1(k))|\}=+\infty$, then $\liminf\limits_{k\to\infty}|\chi_{\lambda(k)}^{\gamma(k)}|=0$ if $\alpha_1<1$ (i.e., if at least one of $\alpha_2,\ldots,\alpha_r$ is positive);

\item[(2)]
if there exists a partition $\mu$ such that $\lambda(k)/(\lambda_1(k))=\mu$ for all large $k$, then $\lim\limits_{k\to\infty}\chi_{\lambda(k)}^{\gamma(k)}=\alpha_1^{q}$;

\item[(3)]
if there exists a partition $\mu$ such that $\lambda'(k)/(\lambda'_1(k))=\mu$ for all large $k$, then $\lim\limits_{k\to\infty}\chi_{\lambda(k)}^{\gamma(k)}=u\cdot \alpha_1^{q}$,
\end{itemize}
where we denote $q=|\mu|$ and $u=\lim\limits_{k\to\infty}(-1)^{N_k(\alpha_2+\alpha_4+\ldots)}\in\{\pm 1\}$.
\end{Lm}

\begin{proof}
Denote $M_k=N_k-d$ and let $g_k$ be the \emph{minimal element} of the conjugacy class of $\mathfrak{S}_{M_k}$ that corresponds to the cycle type $\gamma(k)\vdash M_k$, see Definition \eqref{def:minimal_el_conj_cl}. It follows that for all sufficiently large $k$ we have
\begin{equation*}\label{support_of_min_el}
\supp_{M_k}g_k=\left\{1,2,\ldots,(1-\alpha_1)N_k\right\}.
\end{equation*}

\textbf{(1)}
Fix a sequence $\left\{k_i\right\}_{i=1}^{\infty}$ which satisfies the condition
\begin{equation}
\lim\limits_{i\to\infty}\min\{|\lambda(k_i)/( \lambda_1(k_i))|,|\lambda'(k_i)/(\lambda_1'(k_i))|\}=+\infty.
\end{equation}
It suffices to consider the case when $k_i=i$ for all $i$, i.e., when the sequence $\{k_i\}_{i=1}^{\infty}$ coincides with the sequence $\{k\}_{k=1}^{\infty}$ (the proof of the general case is analogous). Our goal is to show that $\lim\limits_{k\to\infty}\chi_{\lambda(k)}^{\gamma(k)}=0$.

Using Proposition \ref{char_bound} and the equality $\#\supp_{M_k}(g_k)=(2\alpha_2+3\alpha_3+\ldots+r\alpha_r)N_k=(1-\alpha_1)N_k$, we obtain
\begin{equation*}
\left|\chi_{\lambda(k)}^{\gamma(k)}\right|\leq
\left(\max\left\{a,\lambda_1(k)/M_k,\lambda_1'(k)/M_k\right\}\right)^{b(1-\alpha_1)N_k},~\text{where}~a\in (0,1), b>0.
\end{equation*}
Taking the limit when $k\to\infty$ yields
\begin{equation*}
\begin{split}
\limsup_{k\to\infty}\left|\chi_{\lambda(k)}^{\gamma(k)}\right|
&\leq\limsup\limits_{k\to\infty}
\big(\max\{\lambda_1(k)/M_{k},\lambda_1'(k)/M_k\}\big)^{b(1-\alpha_1)N_k}=
\\
&=
\limsup\limits_{k\to\infty}\left(1-\frac{1
}{M_{k}}\min\{|\lambda(k)/\lambda_1(k)|,|\lambda'(k)/\lambda_1'(k)|\}\right)^{b(1-\alpha_1)N_k}\leq
\\
&\leq\limsup\limits_{k\to\infty}\exp\big(-b(1-\alpha_1)\cdot\min\{|\lambda(k)/\lambda_1(k)|,|\lambda'(k)/\lambda_1'(k)|\}\big)=0
\end{split}
\end{equation*}
as $M_k=N_k-d$, $b>0$ and $\alpha_1<1$ by assumption. Hence, $\chi_{\lambda(k)}^{\gamma(k)}\to 0$ as $k\to\infty$.

\textbf{(2)}
Similar to (1) we may assume that $\lambda(k)/(\lambda_1(k))=\mu$ for all $k$. If $\mu$ is the empty partition, then $|\mu|=0$ and $\chi_{\lambda(k)}^{\gamma(k)}=1$ for all $k$, hence the statement is trivial. From now on we suppose that $\mu=(\mu_1,\mu_2,\ldots)$ is a partition of $q=|\mu|\geq 1$. Denote by $\mu'=\left( \mu'_1,\mu'_2,\ldots,\mu'_m \right)$ the conjugate partition. Note that $\lambda(k)=(N_k-q,\mu_1,\mu_2,\ldots)$ for all $k$. In other words, for each $k$ the Young diagram $\lambda(k)$ consists of the first row of length $N_k-q$ and the Young diagram $\mu$.
	
It follows from the hook length formula \cite[Theorem 3.10.2]{Sagan} that
\begin{equation*}\label{hook}
\begin{aligned}
\#\Tab(\lambda(k))
&=\dim\lambda(k)=\frac{\dim\mu}{|\mu|!}\cdot\frac{M_k!}{(M_k-|\mu|-m)!
\prod_{i=1}^{m}(M_k-|\mu|-i+1+\mu'_i)}=
\\
&=\frac{\dim\mu}{|\mu|!}\cdot
\frac{\prod_{i=1}^{|\mu|+m}(M_k-|\mu|-m+i)}{\prod_{i=1}^{m}(M_k-|\mu|-i+1+\mu'_i)}=\Big(\frac{\dim\mu}{q!}+o(1)\Big)N_{k}^{q}~\text{as}~k\to\infty.
\end{aligned}
\end{equation*}
To calculate the limit of the sequence $\{\chi_{\lambda(k)}^{\gamma(k)}\}_{k=1}^{\infty}$, we start with the equality
\begin{equation}\label{eq:char_matrix_elem}
\chi_{\lambda(k)}^{\gamma(k)}=\chi_{\lambda(k)}(g_k)=\frac{\Tr\rho_{\lambda(k)}(g_k)}{\Tr\rho_{\lambda(k)}(1)}=\frac{1}{\#\Tab(\lambda(k))}\sum\limits_{T\in\Tab(\lambda(k))}
\left(\rho_{\lambda(k)}
(g_k)v_T,v_T\right).
\end{equation}
It turns out that in the sum above only a certain subset of $\Tab(\lambda(k))$ gives a non-negligible contribution to the limit as $k\to\infty$.

Fix an arbitrary sufficiently large positive integer parameter $Q>q$.
Define the following sets:
\begin{itemize}
\item
Let $\Tab_{1}(\lambda(k))$ be the set of all tableaux $T\in\Tab(\lambda(k))$ such that all elements of $\{1,2,\ldots,Q\}$ and $\supp_{M_k} g_k=\{1,2,\ldots,(1-\alpha_1)N_k\}$ appear in the first row of $T$;

\item
Let $\Tab_{2}(\lambda(k))$ be the set of all tableaux $T\in\Tab(\lambda(k))\setminus\Tab_1(\lambda(k))$ such that numbers $\{1,2,\ldots,Q\}$ belong to the first row of $T$ and such that the diagram $\mu$ does not contain two consecutive integers;

\item
Let $\Tab_{3}(\lambda(k))$ be the complement of $\Tab_1(\lambda(k))\sqcup\Tab_2(\lambda(k))$ in $\Tab(\lambda(k))$.
\end{itemize}
Clearly, $\Tab(\lambda(k))=\Tab_1(\lambda(k))\sqcup\Tab_2(\lambda(k))\sqcup\Tab_3(\lambda(k))$. Our aim is to show that as $k\to\infty$, only the terms corresponding to $T\in\Tab_1(\lambda(k))$ contribute to the limit of the right-hand side of \eqref{eq:char_matrix_elem}.

First, let us give estimates for the cardinalities of sets $\Tab_i(\lambda(k))$, $i=1,2,3$. Observe that the union $\Tab_1(\lambda(k))\sqcup\Tab_2(\lambda(k))$ contains all tableaux $T\in\Tab(\lambda(k))$ for which the diagram $\mu$ does not contain two consecutive integers and does not contain an element from $\{1,2,\ldots,Q\}$. Such a tableau is completely determined by the filling of the diagram $\mu$. Then, it is not difficult to see that
\begin{equation*}
\begin{split}
\#\Tab_1(\lambda(k))+\#\Tab_2(\lambda(k))
&\ge\dim\mu\cdot\frac{(M_k-Q)(M_k-Q-3)\cdots(M_k-Q-3(|\mu|-1))}{|\mu|!}=
\\
&=\Big(\frac{\dim\mu}{q!}+o(1)\Big)N_{k}^{q},~k\to\infty.
\end{split}
\end{equation*}
Therefore, we have
\begin{equation*}
\#\Tab_3(\lambda(k))=\#\Tab\lambda(k)-\#\Tab_1(\lambda(k))-\#\Tab_2(\lambda(k))=o(N_k^q)=\#\Tab(\lambda(k))\cdot o(1)~\text{as}~k\to\infty.
\end{equation*}
Similarly, the number of tableaux in $\Tab_1(\lambda(k))$ is equal to the number of standard Young tableaux of shape $\mu$ filled by elements of $\{(1-\alpha_1)N_k+1,\ldots,M_k=N_k-d\}$. Hence,
\begin{equation*}
\#\Tab_1(\lambda(k))=\dim\mu\cdot\binom{\alpha_1N_k-d}{|\mu|}=\Big(\alpha_1^q\cdot\frac{\dim\mu}{q!}+o(1)\Big)N_k^q=(\alpha_1^q+o(1))\cdot\#\Tab(\lambda(k)),~k\to\infty.
\end{equation*}

We proceed to the analysis of the matrix coefficients $(\rho_{\lambda(k)}(g_k)v_T,v_T)$ for $T\in\Tab_1(\lambda(k))\sqcup\Tab_2(\lambda(k))$.
It follows from the definition of $g_k\in\mathfrak{S}_{M_k}$ that there exists a presentation
\begin{equation*}
g_k=s_{i_1}s_{i_2}\ldots s_{i_r}
\end{equation*}
in terms of Coxeter generators, where $i_1<i_2<\ldots<i_r$ are elements of $\supp_{M_k}(g_k)$. Consider two cases:
\begin{itemize}
\item[(a)] $T\in\Tab_1(\lambda(k))$.

\smallskip

\noindent
In this case all elements $i_1,i_2,\ldots,i_r$ belong to the first row of $T$. Hence, by Lemma \ref{lem:matrix_elem_formula} we have $(\rho_{\lambda(k)}(g_k)v_T,v_T)=1$.

\medskip

\item[(b)] $T\in\Tab_2(\lambda(k))$.

\smallskip

\noindent
In this case the definition of $\Tab_2(\lambda(k))$ implies that there exists an element $i\in\supp_{M_k}(g_k)$ such that $i-1$ and $i+1$ belong to the first row of $T$, but $i$ is not. It follows that for some $j\in\{1,2,\ldots,r\}$ we have $i_j=i$ or $i_j=i-1$. Clearly, $i\ge Q+1$ and therefore\footnote{Recall that $a_i$ is the content of the box of the tableau $T$ that contains $i$.}, $a_{i}\ge Q$. On the other hand, both $i-1$ and $i+1$ belong to the diagram $\mu$ in $T$ which gives $a_{i-1},a_{i+1}\in\{-|\mu|,-|\mu|+1\ldots,|\mu|-2\}$. Thus, $|a_{i_j+1}-a_{i_j}|\ge Q-|\mu|+2=Q-q+2$. Then, Lemma \ref{lem:matrix_elem_formula} implies that $(\rho_{\lambda(k)}(g_k)v_T,v_T)\le\frac{1}{Q-q+2}$.
\end{itemize}
Now we are ready to perform the final calculation. Recall that
\begin{equation*}
\chi_{\lambda(k)}^{\gamma(k)}=\frac{1}{\#\Tab(\lambda(k))}\sum_{i=1,2,3}\sum_{T\in\Tab_i(\lambda(k))}(\rho_{\lambda(k)}v_T,v_T).
\end{equation*}
The discussion above implies that
\begin{equation*}
\begin{split}
&\sum_{T\in\Tab_1(\lambda(k))}(\rho_{\lambda(k)}v_T,v_T)=\#\Tab_1(\lambda(k))=(\alpha_1^q+o(1))\cdot\#\Tab(\lambda(k)),~k\to\infty,
\\
&\left|\sum_{T\in\Tab_2(\lambda(k))}(\rho_{\lambda(k)}v_T,v_T)\right|\leq\frac{\#\Tab_2(\lambda(k))}{Q-q+2}\le\frac{1}{Q-q+2}\cdot\#\Tab(\lambda(k)),
\\
&\left|\sum_{T\in\Tab_3(\lambda(k))}(\rho_{\lambda(k)}v_T,v_T)\right|\leq\#\Tab_3(\lambda(k))=o(1)\cdot\#\Tab(\lambda(k)).
\end{split}
\end{equation*}
Therefore, we obtain
\begin{equation}
\left|\chi_{\lambda(k)}^{\gamma(k)}-\alpha_1^{q}\right|\le\frac{1}{Q-q+2}+o(1)~\text{as}~k\to\infty.
\end{equation}
Finally, since $Q$ can be chosen arbitrarily large, we get $\lim\limits_{k\to\infty}\chi_{\lambda(k)}^{\gamma(k)}=\alpha_1^q$.

\textbf{(3)}
By the formula \eqref{sign_char}, we have $\chi_{\lambda(k)}^{\gamma(k)}=\chi_{\lambda(k)}(g_k)=\chi_{\lambda'(k)}(g_k)\cdot\sgn_{M_k}(g_k)$. Since the cycle type of $g_k$ is $\gamma(k)=(\alpha_1N_k-d,\alpha_2N_k,\ldots,\alpha_r N_k)$, we have $\sgn_{M_k}(g_k)={(-1)}^{N_k(\alpha_2+\alpha_4+\ldots)}$. The statement of the third case of the lemma now follows directly from the previous case.
\end{proof}

\subsection{Approximation of characters of the inductive limit by the characters of pre-limit groups.}\label{characters_approximation}

Let $\{G(n)\}_{k=1}^{\infty}$ be a sequence of Hausdorff, second countable compact groups such that
\begin{equation*}
G(1)\subset G(2)\subset\ldots\subset G(n)\subset\ldots
\end{equation*}
Let $G(\infty)$ be the topological group which is the inductive limit of the sequence $\{G(n)\}_{n=1}^{\infty}$. In particular, as a set, $G(\infty)=\bigcup_{k=1}^{\infty}G(n)$. Denote by $\widehat{G}(n)$ the countable set of equivalence classes of all (finite-dimensional) irreducible representations of the group $G(n)$.

\begin{Prop}\label{prop:general_approx_thm}
Let $\chi$ be an indecomposable character of the group $G(\infty)$. Then, there exists an increasing sequence of positive integers $\{n_l\}_{l=1}^{\infty}$ and a sequence of irreducible representations $\{\pi_{n_l}\in\widehat{G}(n_l)\}_{l=1}^{\infty}$ such that the corresponding sequence $\{\chi_{\pi_{n_l},n_l}\}_{l=1}^{\infty}$ of indecomposable characters of groups $G(n_l)$ converges to $\chi$ uniformly on all compact subsets of $G(\infty)$.
\end{Prop}

\begin{Rem}
G. Olshanski in \cite[Theorem 22.10]{Olshanski} proves a version of this statement for a more general situation of \emph{locally compact} groups and \emph{arbitrary} (not necessarily central) positive-definite functions. The proof therein may be adapted for the case of characters (central positive-definite functions) as well. We use a different approach which relies on the properties of the Gelfand--Naimark--Segal construction.
\end{Rem}

\begin{proof}[Proof of Proposition \ref{prop:general_approx_thm}.]
Let $\chi\colon G(\infty)\to\mathbb{C}$ be an indecomposable normalized character. Then, by the GNS-construction, there exists a unitary $\mathrm{II}_1$-factor representation $\Pi$ of the group $G(\infty)$ in a Hilbert space $\mathcal{H}$ and a unit vector $\xi\in\mathcal{H}$ such that
\begin{itemize}
\item
$\xi$ is a cyclic and separating vector with respect to $\Pi(G(\infty))''$;

\item
for any $g\in G(\infty)$ we have $\chi(g)=\langle\Pi(g)\xi,\xi\rangle_{\mathcal{H}}$.
\end{itemize}

Let $\mathcal{B}(\mathcal{H})$ and $\mathcal{U}(\mathcal{H})$ be the sets of bounded and unitary operators on Hilbert space $\mathcal{H}$, respectively.
Denote by $M\subset \mathcal{B}(\mathcal{H})$ the $W^*$-subalgebra (the von Neumann subalgebra) of $\mathcal{B}(\mathcal{H})$ generated by the operators $\{\Pi(g):g\in G(\infty)\}$. Then, the von Neumann bicommutant theorem implies that $M=\Pi(G(\infty))''$. Since $\Pi$ is type $\mathrm{II}_1$-representation, the $W^*$-algebra $M$ is a $\mathrm{II}_1$-factor, i.e., $M\cap M'=\mathbb{C}\cdot I_{\mathcal{H}}$. The properties of the GNS-construction and $\xi$ also imply that the vector state $\tr\colon M\to\mathbb{C}$ given by the formula $\tr(a)=\langle a\xi,\xi\rangle_{\mathcal{H}}$ is a faithful normal tracial state on $M$ and, moreover, $\tr(\Pi(g))=\chi(g)$ for any $g\in G(\infty)$.

Firstly, observe that the $W^*$-algebra $M$ can be regarded as a pre-Hilbert space via the following inner product: for any $a,b\in M$ define\footnote{This is also known as the \emph{standard representation} of the tracial $W^*$-algebra $(M,\tr)$. See \cite[Ch. I.9, Eq. (5)]{Takesaki} for more details about this construction.}
\begin{equation*}
\langle a,b\rangle_{\tr}=\tr(b^*a)=\langle b^*a\xi,\xi\rangle_{\mathcal{H}}=\langle a\xi,b\xi\rangle_{\mathcal{H}}.
\end{equation*}
Denote by $L^2(M,\tr)$ the completion of $M$ with respect to the inner product $\langle\cdot,\cdot\rangle_{\tr}$. (Note that this inner product is non-degenerate since $\xi$ is a separating vector for $M$.)

For each positive integer $n$ consider the mapping $E_n\colon M\to M$ defined by the formula
\begin{equation}\label{eq:E_n_def}
E_n(m)=\int_{G(n)}\Pi(g)m\Pi(g)^{-1}dg,~m\in M.
\end{equation}
Here, we endow the group $G(n)$ with the probability Haar measure (as the group $G(n)$ is compact, its Haar measure is biinvariant).
It follows from equalities
\begin{equation*}
E_n(E_n(m))=E_n(m),~\langle E_n(m_1),m_2\rangle_{\tr}=\langle m_1,E_n(m_2)\rangle_{\tr}
\end{equation*}
which hold for all $m,m_1,m_2\in M$, that $E_n$ can be extended to an orthogonal projection in $L^2(M,\tr)$. Observe that $E_n$ also projects $M$ onto the subalgebra
\begin{equation*}
M\cap\Pi(G(n))'=M^{G(n)}=\{m\in M:\Pi(g)m\Pi(g)^{-1}=m~\text{for all}~g\in G(n)\}
\end{equation*}
of $G(n)$-invariants of $M$. In other words, $E_n$ is the \emph{conditional expectation} (see \cite[Ch.~V.2, Proposition 2.36]{Takesaki}) of $M$  onto $M^{G(n)}$.

Since $G(n)\subset G(n+1)$, we also have $E_n\ge E_{n+1}$. Thus, there exists the limit $E_{\infty}=\lim\limits_{k\to\infty}E_n$ in the strong operator topology of $\mathcal{B}(L^2(M,\tr))$. It follows that for any $m\in M$ one has
\begin{equation}\label{eq:E_n_limit}
\lim_{n\to\infty}\|E_n(m)-E_{\infty}(m)\|_{\tr}=0.
\end{equation}

\begin{Lm}\label{lem:E_infty}
For any $m\in M$ we have $E_{\infty}(m)=\tr(m)\cdot I_{\mathcal{H}}$.
\end{Lm}

\begin{proof}[Proof of Lemma \ref{lem:E_infty}]
To start with, fix any $m\in M$ and let us show that $E_{\infty}(m)$ belongs to $M$. As $M$ is a $W^*$-algebra in $\mathcal{B}(\mathcal{H})$, it suffices to check that $E_{\infty}(m)$ is the limit of $\{E_n(m)\}_{k=1}^{\infty}\subset\mathcal{B}(\mathcal{H})$ in the strong operator topology of $\mathcal{H}$.

Indeed, from the construction above it follows that the sequence $\{E_k(m)\xi\}_{n=1}^{\infty}$ is a Cauchy sequence in $\mathcal{H}$ and therefore, there exists a limit $\lim\limits_{n\to\infty}E_n(m)\xi=\eta\in\mathcal{H}$. We claim that for any operator $A$ in $M$ the sequence $\{E_n(m)A\xi\}_{n=1}^{\infty}$ converges to $A\eta$. Since $\bigcup_{n=1}^{\infty}\Pi(G(n))''$ is dense in $M=\Pi(G(\infty))''$ in the strong operator topology, and operators $\{E_n(m)\}_{k=1}^{\infty}\subset\mathcal{B}(\mathcal{H})$ are uniformly bounded in norm (namely, by $\|E_n(m)\|_{\mathcal{H}}\leq\|m\|_{\mathcal{H}}$ for all $n$), it suffices to verify the claim for $A\in\Pi(G(l))$ for each $l$.
To prove the latter, observe that for all $n$ and $l$ such that $n\ge l$ and for any $m\in M$ we have $E_n(m)\in\Pi(G(l))'\cap M$. Thus, for any fixed $l$ and $A\in\Pi(G(l))$ we have
\begin{equation*}
\lim_{n\to\infty}E_n(m)A\xi=\lim_{n\to\infty}AE_n(m)\xi=A\eta.
\end{equation*}
This shows that the sequence $\{E_n(m)\}_{n=1}^{\infty}$ converges in the strong operator topology of $\mathcal{B}(\mathcal{H})$. Denote by $\widetilde{E}(m)$ the corresponding limit. Then, we have
\begin{equation*}
\lim_{n\to\infty}\|E_n(m)-\widetilde{E}(m)\|_{\tr}=\lim_{n\to\infty}\|E_n(m)\xi-\widetilde{E}(m)\xi\|_{\mathcal{H}}=0,
\end{equation*}
and thus, $E_{\infty}(m)=\widetilde{E}(m)$ is an element of $M$.

As we noted before, $E_n(m)\in\Pi(G(l))'\cap M$ whenever $n\ge l$. Hence, the limit $E_{\infty}(m)=\lim\limits_{n\to\infty}E_k(m)$ is an element of $M\cap\Pi(G(\infty))'=M\cap M'$. However, the $W^*$-algebra $M$ is a factor, that is, $M\cap M'=\mathbb{C}\cdot I_{\mathcal{H}}$. Therefore, $E_{\infty}(m)$ is a scalar operator in $\mathcal{B}(\mathcal{H})$. The cyclicity property of the state $\tr$ implies that
\begin{equation*}
\tr(E_n(m))=\tr\left(\int_{G(n)}\Pi(g)m\Pi(g)^{-1}dg\right)=\int_{G(n)}\tr(\Pi(g)m\Pi(g)^{-1})dg=\int_{G(n)}\tr(m)dg=\tr(m),
\end{equation*}
and hence $\tr(E_{\infty}(m))=\lim\limits_{n\to\infty}\tr(E_n(m))=\tr(m)$. Thus, we obtain $E_{\infty}(m)=\tr(m)\cdot I_{\mathcal{H}}$, as claimed.
\end{proof}

\begin{Lm}\label{lem:lim_scalar_operator}
For any $m\in M$ we have
\begin{equation*}
\lim_{n\to\infty}\left\|\int_{G(n)}(\Pi(g)m\Pi(g)^{-1}\xi-\tr(m)\xi) dg\right\|^2_{\mathcal{H}}=0.
\end{equation*}
\end{Lm}

\begin{proof}[Proof of Lemma \ref{lem:lim_scalar_operator}]
The statement of the lemma now follows from the previous lemma together with formulas \eqref{eq:E_n_limit} and \eqref{eq:E_n_def}.
\end{proof}

\begin{Lm}\label{lem:Pi_m_Pi_decomposition}
For $f\in L^1(G(i))$ consider the element $m=\int_{G(i)}f(h)\Pi(h)dh$ of the $W^*$-algebra $M$. Then, for any $n\ge i$ the following identity holds:
\begin{equation*}
\int_{G(n)}\Pi(g)m\Pi(g)^{-1}dg=\sum_{\pi\in\widehat{G}(n)}\left(\int_{G(i)}f(h)\chi_{\pi,n}(h)dh\right)\cdot P_{\pi},
\end{equation*}
where $\pi\in\widehat{G}(n)$, the operator $P_{\pi,n}$ is given by the formula
\begin{equation}\label{eq:projector_isotypic_comp}
P_{\pi,n}=(\dim\pi)^2\int_{G(n)}\chi_{\pi}(g^{-1})\Pi(g)dg,
\end{equation}
and $\chi_{\pi,n}(g)=\frac{1}{\dim\pi}\Tr(\pi(g))$ is the normalized character of the finite-dimensional irreducible representation $\pi$ of $G(n)$. In other words, $P_{\pi}$ is the orthogonal projection onto the $\pi$-isotypic component of the representation $\Pi\colon G(\infty)\to\mathcal{U}(\mathcal{H})$ viewed as a $G(l)$-representation.
\end{Lm}
\begin{Rem}
Note that the orthogonal projections $\{P_{\pi,n}:\pi\in\widehat{G}(n)\}$ are pairwise orthogonal and their sum $\sum_{\pi\in\widehat{G}(l)}P_{\pi,n}$ is the identity operator in $\mathcal{H}$. This is one of the consequences of the Peter--Weyl theorem for $G(n)$, see \cite[Theorem 4.20]{Knapp}.
\end{Rem}
\begin{proof}[Proof of Lemma \ref{lem:Pi_m_Pi_decomposition}]
This is a consequence of some standard facts from the harmonic analysis on a compact group $G(n)$. Indeed, it follows from the Schur orthogonality relations, see e.g., \cite[Corollary 4.10]{Knapp}, that the equality
\begin{equation*}
\int_{G(n)}\Pi(ghg^{-1})dg=\sum_{\pi\in\widehat{G}(l)}\chi_{\pi,n}(h)\cdot P_{\pi,n},
\end{equation*}
holds for any $h\in G(n)$. Multiplying both sides by $f(h)$ and integrating over $h\in G(i)\subset G(n)$ yields the formula from the lemma.
\end{proof}

Now let conclude the proof of the proposition. Since for each $i$ the group $G(i)$ is separable, we can choose a countable dense subset $\{f_{i,1},f_{i,2},\ldots\}$ in $L^1(G(i))$. Then, by Lemma \ref{lem:lim_scalar_operator} and Lemma \ref{lem:Pi_m_Pi_decomposition} applied for $f=f_{i,j}$, for all $j$ we have
\begin{equation}\label{eq:decomposition_tends_to_zero}
\begin{split}
&\lim_{n\to\infty}
\sum_{\pi\in\widehat{G}(n)}\left|\int_{G(i)}f_{i,j}(h)(\chi_{\pi,n}(h)-\chi(h))dh\right|^2\cdot \|P_{\pi}\xi\|_{\mathcal{H}}^2=0,~\text{and}
\\
&\tr\left(\int_{G(i)}f_{i,j}(h)\Pi(h)dh\right)=\int_{G(i)}f_{i,j}(h)\tr(\Pi(h))dh=\int_{G(i)}f_{i,j}(h)\chi(h)dh.
\end{split}
\end{equation}
Denote for any $i$, $j$, for each $n\ge i$ and $\pi\in\widehat{G}(n)$ by $I_{i,j}^{(n)}(\pi)$ the following quantity:
\begin{equation*}
I_{i,j}^{(n)}(\pi)=\int_{G(i)}f_{i,j}(h)(\chi_{\pi,n}(h)-\chi(h))dh.
\end{equation*}
It follows from \eqref{eq:decomposition_tends_to_zero} that
\begin{equation*}
\lim_{n\to\infty}
\sum_{\pi\in\widehat{G}(n)}\Bigg(\sum_{i,j=1}^{l}|I_{i,j}^{(n)}(\pi)|^2\Bigg)\cdot\|P_{\pi,n}\xi\|_{\mathcal{H}}^2=0 ~\text{for any fixed positive integer}~l.
\end{equation*}
Fix a sequence $\{\varepsilon_{l}\}_{l=1}^{\infty}$ of positive reals such that $\lim\limits_{l\to\infty}\varepsilon_{l}=0$.
Since for each $n$ we have $\sum_{\pi\in\widehat{G}(n)}\|P_{\pi,n}\xi\|_{\mathcal{H}}^2=\|\xi\|^2_{\mathcal{H}}=1$, for every $l$ there exists a positive integer $n_l$ and $\pi_{n_l}\in\widehat{G}(n_l)$ such that
\begin{equation}\label{eq:ineq_for_convergence}
\sum_{i,j=1}^{l}|I_{i,j}^{(n_l)}(\pi_{n_l})|^2<\varepsilon_{l}.
\end{equation}
Without loss of generality we may assume that the sequence $\{n_l\}_{l=1}^{\infty}$ is strictly increasing. Then, for any fixed $i$ and $j$, we have
\begin{equation*}
\lim_{l\to\infty}\int_{G(i)}f_{i,j}(h)(\chi_{\pi_{n_l},n_l}(h)-\chi(h))dh=\lim_{l\to\infty}I_{i,j}^{(n_l)}(\pi_{n_l})=0.
\end{equation*}
The density of $\{f_{i,1},f_{i,2},\ldots\}$ in $L^1(G(i))$ and the fact that the functions $\{\chi_{\pi_{n_l},n_l}-\chi\}_{k=1}^{\infty}$ are uniformly bounded in $L^{\infty}$-norm, imply that $\chi_{\pi_{n_l},n_l}|_{G(i)}$ converges to $\chi|_{G(i)}$ in the weak-$*$ topology on $L^{\infty}(G(i))$. Finally, Lemma \ref{lem:weak_comp_convergence} (see below) implies that $\chi_{\pi_{n_l},n_l}$ converges to $\chi$ uniformly on each $G(i)$. This concludes the proof of the proposition.
\end{proof}

\begin{Lm}\label{lem:weak_comp_convergence}
Let $G$ be a Hausdorff locally compact group. Let $\operatorname{Pos}_1(G)$ be the set of all continuous positive-definite functions on $G$ whose value at the identity element is $1$. Then, the weak-$*$ topology on $\operatorname{Pos}_1(G)\subset L^{\infty}(G)$ induced by the natural pairing between $L^{\infty}(G)$ and $L^1(G)$ is equivalent to the topology of uniform convergence on compact subsets of $G$.
\end{Lm}

\begin{proof}
This is a well-known fact about positive-definite functions on a locally compact group. For the proof see e.g., \cite[Theorem 13.5.2]{Dixmier}.
\end{proof}

\begin{Rem}
Note that the idea of using the weak convergence of positive-definite functions also appears in the proof of Olshanski of the approximation theorem for positive-definite functions \cite[Theorem 22.10]{Olshanski}.
\end{Rem}

\subsection{Approximation of spherical functions of the inductive limits of Gelfand pairs.}\label{sect:appendix_spherical_func_approx} 
The goal of this section is to recall how the concept of a Gelfand pair generalizes for inductive limits of compact groups and to give a proof of the approximation theorem for spherical functions.

\subsubsection{Classical Gelfand pairs and spherical functions.}
Recall the definition of a Gelfand pair for compact topological groups:

\begin{Def}
Let $G$ be a compact topological group, and let $K$ be its compact subgroup. Then, $(G,K)$ is called a \emph{Gelfand pair}, if any of the following equivalent conditions holds:
\begin{itemize}
\item 
the algebra $C(K\backslash G/K)=\{f\in C(G):f(k_1gk_2)=f(g)~\text{for all}~k_1,k_2\in K\}$ of continuous $K$-biinvariant functions on $G$ is commutative with respect to the convolution;

\item 
for any irreducible (continuous) unitary representation $\pi$ of the group $G$ acting in a Hilbert space $\mathcal{H}$ the subspace $\mathcal{H}^{K}=\{v\in\mathcal{H}:\pi(k)v=v~\text{for all}~k\in K\}$ of $K$-invariant vectors has dimension at most 1.
\end{itemize}
\end{Def}

\begin{Def}\label{def:gelfand_pair}
Let $(G,K)$ be a Gelfand pair. A positive-definite function $\varphi\colon G\to\mathbb{C}$ is called a \emph{spherical (or $K$-spherical) function} if the following conditions hold:
\begin{itemize}
\item 
for all $g\in G$ and $k_1,k_2\in K$ we have $\varphi(k_1gk_2)=\varphi(g)$;

\item 
the GNS representation $\pi_{\varphi}$ of $G$ associated to $\varphi$ is a factor-representation\footnote{It is easily seen that in this case $\pi_{\varphi}$ is irreducible.}.
\end{itemize}  
\end{Def}

\begin{Rem}
For any unit $K$-spherical vector, i.e., a vector $\xi\in\mathcal{H}^{K}$, the formula $\varphi(g)=\langle\pi(g)\xi,\xi\rangle_{\mathcal{H}}$ defines a normalized $K$-spherical function on $G$. If in addition $\mathcal{H}$ is an irreducible representation of $G$, then $\varphi$ does not depend on a choice of $\xi$. Moreover, the GNS representation $\pi_{\varphi}$ is unitarily equivalent to $\pi$ in this case.
\end{Rem}

\begin{Rem}
The algebra $C(K\backslash G/K)$ naturally embeds into the von Neumann algebra $VN(G)\subset \mathcal{B}(L^2(G))$, generated by the left regular representation of $G$ in $L^2(G)$, as follows:
\begin{equation*}
C(K\backslash G/K)\ni f\mapsto \int\limits_{G}f(g)L_g\,{\rm d}\,g,~\text{where}~(L_g\eta)(x)=\eta(g^{-1}x),~ \eta\in L^2(G),
\end{equation*}
and ${\rm d}\,g$ is the normalized Haar measure on $G$.
The operator $P_{K}=\int_{K}L_k\,{\rm d}\,k\in VN(G)$ is an orthogonal projection. The pair $(G,K)$ is a Gelfand pair if and only if the subalgebra $P_{K}\cdot VN(G)\cdot P_{K}$ is abelian.
\end{Rem}

\subsubsection{Gelfand pairs for inductive limits.}
Let $\{G(n)\}_{n=1}^{\infty}$ be as in section \ref{characters_approximation}. Assume in addition that there is a sequence of compact subgroups $K(n)\subset G(n)$ such that $K(n)\subset K({n+1})$ and denote $K(\infty)=\bigcup_{n\ge 1} K(n)$. 
In general, the groups $K(\infty)=\bigcup_{n\ge 1}K(n)$ and $G(\infty)=\bigcup_{n\ge 1}G(n)$ fail to be locally compact. Nevertheless, the concept of Gelfand pair can be defined in terms of representations of $G(\infty)$.
\begin{Def}[{\cite[Definition 8.13, Proposition 8.15]{Borodin_Olshanski}}]\label{def:gelfand_pair_infty}
A pair $(G(\infty), K(\infty))$ is called a Gelfand pair if for any irreducible  representation $\Pi$ of the group $G(\infty)$ in a Hilbert space $\mathcal{H}$, we have $\dim \mathcal{H}^{K(\infty)}\leq 1$, where
\begin{equation*}
\mathcal{H}^{K(\infty)}=\{\eta\in H: \Pi(k)\eta=\eta~\text{for all}~k\in K(\infty)\}.
\end{equation*}
In case when the subspace $\mathcal{H}^{K(\infty)}$ is non-zero, we define the $K(\infty)$-spherical function $\varphi_\Pi$ of $\Pi$ by $\varphi_\Pi(g)=\langle \Pi(g)\xi,\xi\rangle_{\mathcal{H}}$ for $g\in G(\infty)$, where $\xi$ is any unit vector in $\mathcal{H}^{K(\infty)}$.
\end{Def}

\begin{Rem}
The condition in the definition above is equivalent to the property that for any $g_1,g_2\in G(\infty)$ the operators $P_{K(\infty)}\Pi(g_1)P_{K(\infty)}$ and $P_{K(\infty)}\Pi(g_2)P_{K(\infty)}$ commute (here $P_{K(\infty)}$ denotes the orthogonal projection onto $\mathcal{H}^{K(\infty)}$). As for the usual Gelfand pairs, the function $\varphi_{\Pi}$ is $K(\infty)$-biinvariant and positive-definite, does not depend on a choice of $\xi$, and the GNS representation of $G(\infty)$ associated to $\varphi_{\Pi}$ is unitarily equivalent to $\Pi$. 
\end{Rem}

% \begin{Rem}
% Let us take a unit vector  $\xi$ from the subspace $H^{K(\infty)}$. Define positive-definite function $\varphi_\Pi$ by $\varphi_\Pi(g)=\left( \Pi(g)\xi,\xi\right)$, $g\in G(\infty)$. Then GNS-representation of $G(\infty)$, corresponding to $\varphi_\Pi$, is unitary equivalent to $\Pi$. We usually call the vector $\xi$ a $K(\infty)$-spherical.
% \end{Rem}
% {\it For this reason, it is natural to call a positive-definite function $\varphi$ on $G(\infty)$ a $K(\infty)$-spherical provided that the equality $\varphi(k_1gk_2)=\varphi(g)$ holds for all $g\in G(\infty)$, $k_1, k_2\in K(\infty)$ and the associated GNS-representation is irreducible}.

\begin{Prop}
Assume that $(G(n),K(n))$ is a Gelfand pair for each positive integer $n$ in the usual sense (see Definition \ref{def:gelfand_pair}). Then, $(G(\infty),K(\infty))$ is a Gelfand pair in the sense of Definition \ref{def:gelfand_pair_infty}. 
\end{Prop}

From now on we always assume that $(G(n),K(n))$ is a Gelfand pair.
The next proposition shows that spherical functions of $(G(\infty),K(\infty))$ can be approximated by a sequence of spherical functions of $(G(n),K(n))$. Denote by $\widehat{G}_K(n)\subset\widehat{G}(n)$ the set of equivalence classes of irreducible $K$-spherical representations of $G(n)$. Then, for any $\pi\in\widehat{G}_K(n)$ we denote by $\varphi_{\pi,n}$ the corresponding $K(n)$-spherical function.

\begin{Prop}\label{prop:general_approx_sph_func}
Let $\varphi$ be a spherical function of the Gelfand pair $(G(\infty),K(\infty))$. Then, there exists an increasing sequence of positive integers $\{n_l\}_{l=1}^{\infty}$ and a sequence of irreducible representations $\{\pi_{n_l}\in\widehat{G}(n_l)\}_{l=1}^{\infty}$ possessing a $K(n_l)$-spherical vector such that the corresponding sequence $\{\varphi_{\pi_{n_l},n_l}\}_{l=1}^{\infty}$ of spherical functions of $(G(n_l),K(n_l))$ converges to $\varphi$ uniformly on all compact subsets of $G(\infty)$.
\end{Prop}

\begin{proof}
Let $(\Pi,\mathcal{H},\xi)$ be the GNS representation of $G(\infty)$ associated to the spherical function $\varphi$. Then, $\varphi(g)=\langle\Pi(g)\xi,\xi\rangle_{\mathcal{H}}$ for all $g\in G(\infty)$. For a subset $S$ of $\mathcal{H}$ we denote by $[S]$ the closure of the linear span of $S$ in $\mathcal{H}$.

Recall from the previous section that $\widehat{G}(n)$ is the set of isomorphism classes of all irreducible representations of a compact group $G(n)$, and for $\pi\in\widehat{G}(n)$ we denote by $\chi_{\pi,n}$ the normalized character of $\pi$ and by $P_{\pi,n}$ the orthogonal projection onto the $\pi$-isotypic component of $\mathcal{H}$ regarded as a representation of $G(k)$ (cf. Lemma \ref{lem:Pi_m_Pi_decomposition} and \eqref{eq:projector_isotypic_comp}). In particular, the projections $\{P_{\pi,n}:\pi\in\widehat{G}(n)\}$ are mutually orthogonal and $\sum\limits_{\pi\in\widehat{G}(n)}P_{\pi,n}=I_{\mathcal{H}}$. In particular, we have the following orthogonal decomposition:
\begin{equation}\label{eq:xi_decomp_G}
\xi=\bigoplus_{\pi\in\widehat{G}(n)}P_{\pi,n}\xi.
\end{equation}
The projections $P_{\pi,n}$ commute with the action of $G(n)$. Since $K(n)\subset G(n)$, it follows that each summand $P_{\pi,n}\xi$ is a $K(n)$-invariant vector. For $\pi\in\widehat{G}(n)$ such that $P_{\pi,n}\xi\neq 0$, we put $\xi_{\pi,n}=\frac{P_{\pi,n}\xi}{\|P_{\pi,n}\xi\|}$ (otherwise put $\xi_{\pi,n}=0$). Note that if $P_{\pi,n}\xi\neq 0$, then the subspace $[\Pi(G(n))\xi_{\pi,n}]$ is an irreducible representation of $G(n)$ isomorphic to $\pi$, and $\xi_{\pi,n}$ is a $K(n)$-spherical vector. In particular, for any $\pi\in\widehat{G}_K(n)$ the matrix element $\langle\Pi(g)\xi_{\pi,n},\xi_{\pi,n}\rangle$ is either equal to the $K(n)$-spherical function $\varphi_{\pi,n}(g)$ or zero (if $P_{\pi,n}\xi=0$). It follows that non-zero summands in \eqref{eq:xi_decomp_G} are given only by $\pi\in\widehat{G}_K(n)$ and we have the following orthogonal sum:
\begin{equation}\label{xi:decomp_G_K}
\xi=\bigoplus_{\pi\in\widehat{G}_K(n)}\|P_{\pi,n}\xi\|\cdot\xi_{\pi,n}.
\end{equation}

Let $Q_k$ be the orthogonal projection onto the subspace $\mathcal{H}^{K(n)}$. Clearly, we have
\begin{equation*}
Q_n=\int_{K(n)}\Pi(k)dk.
\end{equation*}
Besides that, $\{Q_n\}_{n=1}^{\infty}$ is a decreasing sequence of orthogonal projections which strongly converges to the orthogonal projection $Q_{\infty}$ onto the one-dimensional subspace $\mathcal{H}^{K(\infty)}$ spanned by $\xi$. Since $Q_{\infty}\eta=\langle\eta,\xi\rangle\cdot\xi$, we have
\begin{equation*}\label{eq:limit_K_inv_projections}
\lim_{n\to\infty}\left\|Q_n\eta-\langle\eta,\xi\rangle\xi\right\|=0 ~\text{for all}~\eta\in\mathcal{H}.
\end{equation*}
Since $(G(n),K(n))$ is a Gelfand pair, for any $\pi\in\widehat{G}_K(n)$ we have the following equality:
\begin{equation}\label{eq:gelfand_pair_identity}
Q_n\Pi(g)\xi_{\pi,n}=\langle\Pi(g)\xi_{\pi,n},\xi_{\pi,n}\rangle\cdot\xi_{\pi,n}=\varphi_{\pi,n}(g)\xi_{\pi,n},~g\in G(n).
\end{equation}
Indeed, if $\xi_{\pi,n}=0$, this is trivial, otherwise it is a consequence of the fact that $\xi_{\pi,n}$ is a $K(n)$-spherical vector of $G(n)$-representation $[\Pi(G(n))\xi_{\pi,n}]$, which is isomorphic to $\pi$.

Now take any function $f\in L^1(G(i))$ and apply the \eqref{eq:limit_K_inv_projections} for $\eta=\int_{G(i)}f(g)\Pi(g)\xi dg$. For $n>i$ we have
\begin{equation*}
\begin{split}
Q_n\eta
&=\int_{G(i)}f(g)Q_n\Pi(g)\xi dg=\sum_{\pi\in\widehat{G}_K(n)}\|P_{\pi,n}\xi\|\cdot\int_{G(i)}f(g)Q_n\Pi(g)\xi_{\pi,n}dg=
\\
&=\sum_{\pi\in\widehat{G}_K(n)}\|P_{\pi,n}\xi\|\cdot\int_{G(i)}f(g)\varphi_{\pi,n}(g)\xi_{\pi,n}dg.
\end{split}
\end{equation*}
On the other hand,
\begin{equation*}
\langle\eta,\xi\rangle\cdot\xi=\int_{G(i)}f(g)\langle\Pi(g)\xi,\xi\rangle\xi dg=\int_{G(i)}f(g)\varphi(g)\xi dg=\sum_{\pi\in\widehat{G}_K(n)}\|P_{\pi,n}\xi\|\cdot\int_{G(i)}f(g)\varphi(g)\xi_{\pi,n} dg.
\end{equation*}
Subtracting these identities and computing the square of the norm yields
\begin{equation*}
\|Q_n\eta-\langle\eta,\xi\rangle\xi\|^2=\sum_{\pi\in\widehat{G}_K(n)}\|P_{\pi,n}\xi\|^2\cdot\left|\int_{G(i)}f(g)(\varphi_{\pi,n}(g)-\varphi(g)) dg\right|^2.
\end{equation*}
Hence, by \eqref{eq:limit_K_inv_projections}, for any $f\in L^1(G(i))$, we have
\begin{equation}\label{eq:limit_sum_weak_conv}
\lim_{n\to\infty}\sum_{\pi\in\widehat{G}_K(n)}\|P_{\pi,n}\xi\|^2\cdot\left|\int_{G(i)}f(g)(\varphi_{\pi,n}(g)-\varphi(g)) dg\right|^2=0.
\end{equation}
Now we can conclude the argument as in the proof of Proposition \ref{prop:general_approx_thm}, see \eqref{eq:decomposition_tends_to_zero}.
\end{proof}

\begin{Rem}
In fact, Proposition \ref{prop:general_approx_thm} can seen as a special case of Proposition \ref{prop:general_approx_sph_func} for the inductive limit of Gelfand pairs $(G(n)\times G(n),\diag(G(n)))$. See Appendix \ref{sect:characters_spherical_func} for more details about the connection between characters and spherical functions of doubles.
\end{Rem}

\subsubsection{Connection to the functional equation for spherical functions.} 
Recall the following observation originally due to Olshanski:
\begin{Lm}[{\cite[Theorem 23.6]{Olshanski}}]
Any spherical function $\varphi$ of the Gelfand pair $(G(\infty),K(\infty))$ satisfies the following equality for any $g_1,g_2\in G(\infty)$:
\begin{equation*}
\lim_{n\to\infty}\int_{K(n)}\varphi(g_1kg_2)dk=\varphi(g_1)\varphi(g_2).
\end{equation*}
\end{Lm}
\begin{proof}
We use the same notation as in the proof of Proposition \ref{prop:general_approx_sph_func}. In particular, $Q_{n}$ denotes the orthogonal projection onto the $K(n)$-invariant subspace of $\mathcal{H}$. 
Since $\{Q_{n}\}_{n=1}^{\infty}$ strongly converges to $Q_{\infty}$~-- the orthogonal projection to $\mathcal{H}^{K(\infty)}=\mathbb{C}\cdot\xi$, we have
\begin{equation*}
\int_{K(n)}\varphi(g_1kg_2)dk=\langle\Pi(g_1)Q_n\Pi(g_2)\xi,\xi\rangle\to\langle\Pi(g_1)Q_{\infty}\Pi(g_2)\xi,\xi\rangle,~n\to\infty.
\end{equation*}
It remains to notice that $Q_{\infty}\eta=\langle\eta,\xi\rangle\xi$. Hence, $\langle \Pi(g_1)Q_{\infty}\Pi(g_2)\xi,\xi\rangle=\varphi(g_2)\langle\Pi(g_1)\xi,\xi\rangle=\varphi(g_1)\varphi(g_2)$, which completes the proof of the lemma.
\end{proof}

\begin{Rem}
The analogous statement for Gelfand pair $(G(n),K(n))$ is as follows: if $\pi\in\widehat{G}_K(n)$, then the corresponding spherical function $\varphi_{\pi,n}$ satisfies the functional equation:
\begin{equation*}
\int_{K(n)}\varphi_{\pi,n}(g_1kg_2)dk=\varphi_{\pi,n}(g_1)\varphi_{\pi,n}(g_2).
\end{equation*}
\end{Rem}

Now let us explain how the final expression in the proof of Proposition \ref{prop:general_approx_sph_func} can be obtained via functional equations for spherical functions of Gelfand pairs $(G(n),K(n))$ and $(G(\infty),K(\infty))$.  By slightly modifying the argument above, one can show that for any functions $f_1,f_2\in L^1(G(l))$ the following equality holds:
\begin{equation}\label{eq:weighted_func_eq}
\lim_{n\to\infty}\int_{K(n)}\int_{G(l)\times G(l)}f_1(g_1)f_2(g_2)\varphi(g_1kg_2)dkdg_1dg_2=\left(\int_{G(l)}f_1(g_1)\varphi(g_1)dg_1\right)\cdot\left(\int_{G(l)}f_2(g_2)\varphi(g_2)dg_2\right).
\end{equation}
Now we decompose the spherical function $\varphi$ using the fact that $\sum_{\pi\in\widehat{G}(n)}P_{\pi,n}=I_{\mathcal{H}}$ and that $P_{\pi,n}\xi=0$ for $\pi\notin\widehat{G}_K(n)$:
\begin{equation*}
\varphi(g)=\langle\Pi(g)\xi,\xi\rangle=\sum_{\pi\in\widehat{G}(n)}\langle\Pi(g)P_{\pi,n}\xi,P_{\pi,n}\xi\rangle=\sum_{\pi\in\widehat{G}_K(n)}\|P_{\pi,n}\xi\|^2\cdot\varphi_{\pi,n}(g),~g\in G(n).
\end{equation*}
Here we use the fact that if $P_{\pi,n}\xi\neq 0$, then $\langle\Pi(g)P_{\pi,n}\xi,P_{\pi,n}\xi\rangle=\|P_{\pi,n}\xi\|^2\cdot\varphi_{\pi,n}(g)$ (see also \eqref{eq:gelfand_pair_identity}).
Then, we can write
\begin{equation*}
\begin{split}
\int_{K(n)}\int_{G(l)\times G(l)}&f_1(g_1)f_2(g_2)\varphi(g_1kg_2)dkdg_1dg_2=
\\
&=\sum_{\pi\in\widehat{G}_K(n)}\|P_{\pi}\xi\|^2\int_{G(l)\times G(l)}f_1(g_1)f_2(g_2)dg_1dg_2\int_{K(n)}\varphi_{\pi,n}(g_1kg_2)dk=
\\
&=\sum_{\pi\in\widehat{G}_K(n)}\|P_{\pi}\xi\|^2\int_{G(l)\times G(l)}f_1(g_1)f_2(g_2)\varphi_{\pi,n}(g_1)\varphi_{\pi,n}(g_2)dg_1dg_2=
\\
&=\sum_{\pi\in\widehat{G}_K(n)}\|P_{\pi}\xi\|^2\left(\int_{G(l)}f_1(g_1)\varphi_{\pi,n}(g_1)dg_1\right)\left(\int_{G(l)}f_2(g_2)\varphi_{\pi,n}(g_2)dg_2\right).
\end{split}
\end{equation*}
Now choose $f_1=f$ and $f_2=f^*$, where $f\in L^1(G(l))$ and\footnote{Recall that $f\mapsto f^*$ is the standard involution on the Banach algebra $L^1(G(l))$. The function $\Delta$ stands for the \emph{modular function} of $G(l)$.} $f^*(h)=\Delta(h^{-1})\overline{f(h^{-1})}$. From the discussion above we obtain 
\begin{equation*}
\lim_{n\to\infty}\left[\sum_{\pi\in\widehat{G}_K(n)}\|P_{\pi}\xi\|^2\cdot\left|\int_{G(l)}f(g)\varphi_{\pi,n}(g)dg\right|^2-\left|\int_{G(l)}f(g)\varphi(g)dg\right|^2\right]=0.
\end{equation*}
Finally, observe that 
\begin{equation*}
\sum_{\pi\in\widehat{G}_K(n)}\|P_{\pi}\xi\|^2\cdot\left|\int_{G(l)}f(g)\varphi_{\pi,n}(g)dg\right|^2-\left|\int_{G(l)}f(g)\varphi(g)dg\right|^2=\sum_{\pi\in\widehat{G}_K(n)}\|P_{\pi}\xi\|^2\cdot\left|\int_{G(l)}f(g)(\varphi_{\pi,n}(g)-\varphi(g))dg\right|^2.
\end{equation*}
The last identity holds because $\varphi(g)=\sum_{\pi\in\widehat{G}_K(n)}\|P_{\pi}\xi\|^2\cdot\varphi_{\pi,n}(g)$ and $\sum_{\pi\in\widehat{G}_K(n)}\|P_{\pi}\xi\|^2=\|\xi\|^2=1$.
Together with \eqref{eq:weighted_func_eq}, this yields an alternative approach to the main formula \eqref{eq:limit_sum_weak_conv} in the proof of Proposition \ref{prop:general_approx_sph_func}.

\subsection{Other technical lemmas.}

\subsubsection{Generalization of Lemma \ref{lem:lim_sph_func_formula}.} In the proof of the main theorem we used the following fact.

\begin{Lm}\label{gen_sph_func_lim}
Let $(\mathfrak{t},\sigma)\in G_{N_k}$ be an element of $G_{\widehat{\mathbf{n}}}$. Fix non-negative integers $\{d_j\}_{j\in\mathbb{Z}\setminus\{0\}}$ such that only finitely many among them are non-zero. Denote $d=\sum_{j\in\mathbb{Z}\setminus\{0\}}d_j$ and
\begin{equation*}
\psi_k(\mathfrak{t},\sigma)=\sum_{\substack{\bigsqcup_{j}A_{j}\subset\Fix_{N_k}(\sigma) \\ \#A_j=d_j}}\prod_{j\in\mathbb{Z}\setminus\{0\}}\prod_{x\in A_j}\mathfrak{t}(x)^{j},
\end{equation*}
where the first summation runs over all choices of disjoint subsets $\{A_{j}\}_{j\in\mathbb{Z}\setminus\{0\}}$ such that $\bigsqcup_{j\in\mathbb{Z}\setminus\{0\}}A_j\subset \Fix_{N_k}(\sigma_k)$ and $\#A_j=d_j$. Then,
\begin{equation*}
\lim_{k\to\infty}N_k^{-d}\psi_k(\mathfrak{t},\sigma)=\prod_{j\in\mathbb{Z}\setminus\{0\}}\frac{1}{d_j!}\Bigg(\int_{\Fix(\sigma;\mathbb{X}_{\widehat{\mathbf{n}}})}\mathfrak{t}(x)^{j}d\nu_{\widehat{\mathbf{n}}}(x)\Bigg)^{d_j}.
\end{equation*}
\end{Lm}
\begin{proof}
Let $(M_1,\ldots,M_d)$ be a sequence of integers which contains $d_j$ entries equal to $j$ (this sequence is unique up to a permutation of $M_i$'s)\footnote{Compare with {\bf Step 5} on page \pageref{M_1, M_2_ldots_constants}.}. It follows from the definition of $\psi_k(\mathfrak{t},\sigma)$ that
\begin{equation*}
\psi_k(\mathfrak{t},\sigma)=\frac{1}{\prod_{j\in\mathbb{Z}\setminus\{0\}}d_j!}\sum_{\substack{\textrm{distinct}\\x_1,\ldots,x_d\in\Fix_{N_k}(\sigma)}}\mathfrak{t}(x_1)^{M_1}\ldots\mathfrak{t}_k(x_d)^{M_d},
\end{equation*}
where the summation is over all $d$-tuples $(x_1,\ldots,x_d)$ of distinct elements of $\Fix_{N_k}(\sigma)$. Note that this formula is almost identical to the formula \eqref{phi_m_t_formula} for $\varphi_{\mathfrak{m}_k}(\mathfrak{t})$ from the proof of Lemma \ref{lem:lim_sph_func_formula}.

Assuming that $(\mathfrak{t},\sigma)\in G_{N_{r}}\subset G_{\widehat{\mathbf{n}}}$ for some $r<k$, let us introduce the multiset $\{\mathfrak{t}_r(x):x\in\Fix_{N_r}(\sigma)\}$ and denote its elements by $t_1,\ldots,t_l$. Then, for all $k>r$ the multiset $\{\mathfrak{t}_k(x):x\in\Fix_{N_k}(\sigma)\}$ consists of $l\cdot N_k/N_r$ elements, namely, it consists of $N_k/N_r$ copies each of the elements $t_1,\ldots,t_l$. Arguing as in the proof of Lemma \ref{lem:lim_sph_func_formula}, we obtain that
\begin{equation*}
\begin{aligned}
\lim_{k\to\infty}
N_k^{-d}\psi_k(\mathfrak{t},\sigma)
&=\frac{1}{\prod_{j\in\mathbb{Z}\setminus\{0\}}d_j!}\cdot\frac{1}{N_r^d}\sum_{1\leq i_1,\ldots,i_d\leq l}t_{i_1}^{M_1}\ldots t_{i_d}^{M_d}=\frac{1}{\prod_{j\in\mathbb{Z}\setminus\{0\}}d_j!}\cdot\prod_{p=1}^{d}\left(\frac{1}{N_r}\sum_{i=1}^{l}t_i^{M_p}\right)=
\\
&=\prod_{j\in\mathbb{Z}\setminus\{0\}}\frac{1}{d_j!}\Bigg(\frac{1}{N_r}\sum_{x\in\Fix_{N_r}(\sigma_r)}\mathfrak{t}_r(x)^{j}\Bigg)^{d_j}=\prod_{j\in\mathbb{Z}\setminus\{0\}}\frac{1}{d_j!}\Bigg(\int_{\Fix(\sigma;\mathbb{X}_{\widehat{\mathbf{n}}})}\mathfrak{t}(x)^{j}d\nu_{\widehat{\mathbf{n}}}(x)\Bigg)^{d_j},
\end{aligned}
\end{equation*}
as claimed.
\end{proof}
\begin{Rem}
As can be seen from the proof, the case $\sigma_k=1_{\mathfrak{S}_{N_k}}\in\mathfrak{S}_{N_k}$ of the lemma above is essentially equivalent to Lemma \ref{lem:lim_sph_func_formula}.
\end{Rem}

\subsection{Characters of $G$ and spherical functions of $(G\times G,\diag(G))$.}\label{sect:characters_spherical_func}

The goal of this subsection is to give proofs of Proposition \ref{prop:char_spherical_double} and Proposition \ref{prop:ind_char_and_spher_func} stated in the introduction which relate the spherical representations of the pair $(G_{\widehat{\mathbf{n}}}\times G_{\widehat{\mathbf{n}}},\diag(G_{\widehat{\mathbf{n}}}))$ and indecomposable characters of $G_{\widehat{\mathbf{n}}}$.

Let us recall the setting of Proposition \ref{prop:char_spherical_double}: we have a topological group $G$ and a normalized indecomposable character $\chi\colon G\to\mathbb{C}$. Then, the GNS (Gelfand--Naimark--Segal) construction gives a unitary representation $(\pi_{\chi},\mathcal{H}_{\chi},\xi_{\chi})$ of $G$, where the unitary operators $\{\pi_\chi(g)\}_{g\in G}$ act in Hilbert space $\mathcal{H}_\chi$, the unit vector $\xi_\chi$ is cyclic, i.e., the linear span of the vectors $\{\pi_\chi(g)\xi_\chi\}_{g\in G}$ is dense in $\mathcal{H}_\chi$, and $\chi(g)=\langle\pi_\chi(g)\xi_\chi,\xi_\chi \rangle$ for all $g\in G$. 

We denote by $\mathcal{B}(\mathcal{H}_{\chi})$ the algebra of all bounded operators in $\mathcal{H}_\chi$. Let $\pi_\chi(G)''$ be the $W^*$-algebra generated by the operators $\{\pi_\chi(g)\}_{g\in G}$. For an indecomposable character $\chi$ the von Neumann algebra  $\pi_\chi(G)''$ is a factor, hence the mapping $\pi_\chi(G)''\ni b\stackrel{\tr}{\mapsto} \langle b\,\xi_\chi,\xi_\chi\rangle\in\mathbb{C}$ is the unique normalized\footnote{It means that ${\tr}(1_G)=1$.} weakly continuous trace on $\pi_\chi(G)''$.
The diagonal subgroup of the double group $G\times G$ is defined as $\diag(G)=\{(g,g): g\in G\}$.

\begin{proof}[Proof of Proposition \ref{prop:char_spherical_double}]
Define the operator $S$ by $\mathcal{H}_\chi\ni b\xi_\chi\stackrel{S}{\mapsto} b^*\xi_\chi$, where $b\in\pi_\chi(G)''$. Since for any $b\in\pi_\chi(G)''$ we have
\begin{equation*}
\|b^*\xi_\chi\|^2=\tr(bb^*)=\tr(b^*b)=\| b\xi_\chi \|^2,
\end{equation*}
the operator $S$ is well defined and extends by continuity to an antilinear isometry of $\mathcal{H}_\chi$, i.e., $S$ satisfies the condition $\langle S\eta, S\zeta\rangle=\langle\zeta,\eta\rangle$ for all $\eta,\zeta\in\mathcal{H}_\chi$. It also follows from the definition that $S^2=I_{\mathcal{H}_{\chi}}$. 

\textbf{The commutant of $\pi_{\chi}(G)$.} Our first goal is to show that $S\pi_{\chi}(G)''S=\pi_{\chi}(G)'$. To simplify the notation, denote the $\textrm{II}_1$-factor $\pi_\chi(G)''$ by $M$. Then, note that for any $a,b,x\in M$ we have 
\begin{equation*}
SbSax\xi_\chi=Sbx^*a^*\xi_\chi=axb^*\xi_\chi~~\text{and}~~aSbSx\xi_\chi=aSbx^*\xi_\chi=axb^*\xi_\chi,
\end{equation*} 
hence the operators $SbS$ and $a$ commute. Therefore, we have the inclusion
\begin{equation}\label{SMS_subset M'}
S\pi_\chi(G)''S=SMS\subset M'=\pi_\chi(G)'=\left\{ v\in \mathcal{B}(\mathcal{H}_{\chi}): va=av ~\text{for all}~ a\in\pi_\chi(G)''\right\}
\end{equation}
and $\xi_\chi$ is a cyclic vector for $M'=\pi_\chi(G)'$.

Now let us prove the opposite inclusion $M'\subset SMS$. To this end, we first prove that the mapping $M'\ni a'\mapsto\langle a'\xi_\chi,\xi_\chi\rangle$ defines a tracial state (or simply trace) on $M'$, that is, the equality $\langle a'b'\xi_\chi,\xi_\chi\rangle=\langle b'a'\xi_\chi,\xi_\chi\rangle$ holds for all $a',b'\in M'$.
Indeed, if $a'=SaS$ and  $b'=SbS$, where $a,b\in M$, then from computations above
\begin{equation*}
\langle a'b'\xi_\chi,\xi_\chi\rangle=\langle b^*a^*\xi_\chi,\xi_\chi\rangle=\tr(b^*a^*)=\tr(a^*b^*)=\langle a^*b^*\xi_\chi,\xi_\chi\rangle=\langle b'a'\xi_\chi,\xi_\chi\rangle,
\end{equation*}
Assume, for instance, that $a'$ is not in $W^*$-algebra $SMS\subset M'=\left\{v\in \mathcal{B}(\mathcal{H}_{\chi}): va=av ~\text{for all}~ a\in M\right\}$. Since $\xi_\chi$ is a cyclic vector for $SMS$, there exist sequences $\{Sa_nS\},\{ Sb_nS\}\subset SMS$ such that
\begin{equation}\label{Cauchy_sequences}
\lim\limits_{n\to\infty}\left\|a'\xi_\chi-Sa_nS\xi_\chi \right\|=0~~\text{and}~~ \lim_{n\to\infty}\left\|b'\xi_\chi-Sb_nS\xi_\chi \right\|=0.
\end{equation}
It follows from this that $\{Sa_n^*S\}$, $\{Sb_n^*S\}$ are also Cauchy sequences (recall that $\|x\xi_{\chi}\|=\|x^*\xi_{\chi}\|$ for all $x\in M$). Thus, for any $x\in M$ we have
\begin{equation*}
\begin{split}
\lim_{n\to\infty}\langle Sa_n^*S\xi_\chi,x\xi_\chi \rangle
&=\lim_{n\to\infty}\langle a_n\xi_\chi,x\xi_\chi\rangle=\lim_{n\to\infty}\langle x^*\xi_\chi,a_n^*\xi_\chi\rangle=
\\
&=\lim\limits_{n\to\infty}\langle x^*\xi_\chi, Sa_nS\xi_\chi\rangle=\langle x^*\xi_\chi, a'\xi_\chi \rangle=\langle(a')^*\xi_\chi,x\xi_\chi\rangle.
\end{split}
\end{equation*}
Therefore,
\begin{equation}\label{Cauchy_sequences_star}
\lim_{n\to\infty}\left\|(a')^*\xi_\chi-Sa_n^*S\xi_\chi \right\|=0 ~~\text{and}~~
\lim\limits_{n\to\infty}\left\|(b')^*\xi_\chi-Sb_n^*S\xi_\chi \right\|=0,
\end{equation}
and hence $(a')^*\xi_\chi=\lim\limits_{n\to\infty}Sa_n^*S\xi_\chi$ and $(b')^*\xi_\chi=\lim\limits_{n\to\infty}Sb_n^*S\xi_\chi$. Finally, on account of \eqref{Cauchy_sequences} and \eqref{Cauchy_sequences_star} we have
\begin{equation*}
\begin{split}
\langle a'b'\xi_\chi, \xi_\chi\rangle
&=\lim_{n\to\infty}\langle Sb_nS\xi_\chi,Sa^*_nS\xi_\chi\rangle=\lim_{n\to\infty}\langle b_n^*\xi_\chi,a_n\xi_\chi\rangle=\lim_{n\to\infty}\langle a_n^*\xi_\chi,b_n\xi_\chi\rangle=
\\
&=\lim_{n\to\infty}\langle Sa_nS\xi_\chi,Sb_n^*S\xi_\chi\rangle=\langle a'\xi_\chi,(b')^*\xi_\chi \rangle=\langle b'a'\xi_\chi,\xi_\chi\rangle.
\end{split}
\end{equation*}
Thus, we have proved that the mapping $M'\ni a'\mapsto \langle a'\xi_\chi,\xi_\chi\rangle$ is indeed a trace on $M'$.

As above, the map $M'\xi_\chi\ni a'\xi_\chi\stackrel{S'}{\mapsto} (a')^*\xi_\chi$ extends by continuity to an antilinear isometry $S'\colon\mathcal{H}_\chi\rightarrow\mathcal{H}_\chi$.
Similarly as before, it is easily checked that $S'M'S'\subset (M')'=M$. Since the equality $(SaS)^*x\xi_\chi=xa\xi_\chi$ holds for all $a,x\in M$, we obtain $S'\left( SaS\right)\xi_\chi=a\xi_\chi$. Therefore, $S'=S^{-1}=S$. Hence, applying \eqref{SMS_subset M'}, we obtain that $S\pi_\chi(G)''S=SMS=M'=\pi_\chi(G)'$. 

\textbf{Construction of the spherical representation.}
To conclude the proof of the proposition it remains to show that if $\pi_\chi$ is a factor representation of $G$, then the operators
\begin{equation*}
\pi_\chi^{(2)}((g_1,g_2))=\pi_\chi(g_1)S\pi_\chi(g_2)S,\quad g_1,g_2\in G,
\end{equation*}
form an irreducible spherical representation of the group $G\times G$, and to compute the spherical function of $\pi_{\chi}^{(2)}$.

We start by defining for each $b\in \pi_\chi(G)''$ the mapping $\mathcal{H}_\chi\ni a\xi_\chi\stackrel{R(b)}{\mapsto}ab^*\,\xi_\chi\in\mathcal{H}_\chi$, where  $a\in \pi_\chi(G)''$. It follows from the equalities $\|ab^*\,\xi_\chi\|_{\mathcal{H}_\chi}^2=\langle ba^*ab^*\,\xi_\chi,\xi_\chi \rangle=\langle b^*b\,a^*\,\xi_\chi,a^*\,\xi_\chi\rangle$ that $R(b)$ extends to a bounded operator in $\mathcal{H}_\chi$, $\|R(b)\|\leq\|b\|$ and $R(b)$ is a unitary operator for a unitary $b$. By definition, $R(b)\in\pi_\chi(G)'$ (in fact $R(b)=SbS$) and the mapping $\pi_\chi(G)''\ni b\mapsto R(b)\in \pi_\chi(G)'$ is an antilinear $\star$-representation of $W^*$-algebra $\pi_\chi(G)''$. By the above, the mapping $G\ni g\stackrel{\pi'_\chi}{\mapsto} R\left( \pi_\chi(g)\right)=\pi'_\chi(g)\in \pi_\chi(G)'$ is a unitary representation of $G$. Therefore, the operators $\pi_\chi^{(2)}((g_1,g_2))=\pi_\chi(g_1)\cdot \pi'_\chi(g_2)=\pi_\chi(g_1)S\pi_\chi(g_2)S$, where $(g_1,g_2)\in G\times G$, define a representation of the group $G\times G$. The irreducibility follows from the fact that $\pi_{\chi}(G)'\cap\pi_{\chi}'(G)'=M'\cap M=\mathbb{C}\cdot I_{\mathcal{H}_{\chi}}$.

The definition of $\pi_\chi^{(2)}$ implies that $\pi_\chi^{(2)}((g,g))\xi_\chi=\xi_\chi$ for all $g\in G$. In other words, the vector $\xi_{\chi}$ is a $\diag(G)$-spherical vector of $\pi_{\chi}^{(2)}$. Moreover,
\begin{equation*}
\langle\pi_{\chi}^{(2)}(g_1,g_2)\xi_{\chi},\xi_{\chi}\rangle=\langle\pi_{\chi}^{(2)}(g_1g_2^{-1},1)\xi_{\chi},\xi_{\chi}\rangle=\chi(g_1g_2^{-1}).
\end{equation*}
Hence, the value of the spherical function of representation $\pi_{\chi}^{(2)}$ at $(g_1,g_2)\in G\times G$ equals $\chi(g_1g_2^{-1})$. This completes the proof of the proposition.
\end{proof}

In the proof of Proposition \ref{prop:ind_char_and_spher_func} we use the constructions from Proposition \ref{prop:char_spherical_double} specialized to the case $G=G_{\widehat{\mathbf{n}}}$.

\begin{proof}[Proof of Proposition \ref{prop:ind_char_and_spher_func}] Recall that the homomorphism $\Gamma\colon G_{\widehat{\mathbf{n}}}\to G_{\widehat{\mathbf{n}}}\times G_{\widehat{\mathbf{n}}}$ is given by $\Gamma(\mathfrak{t}s)=(\mathfrak{t}s,s)$, where $\mathfrak{t}\in\mathbb{T}^{\widehat{\mathbf{n}}}$ and $s\in\mathfrak{S}_{\widehat{\mathbf{n}}}$. Note that the image of $\Gamma$ normalizes the subgroup $\mathbb{T}^{\widehat{\mathbf{n}}}\times\{1_{G_{\widehat{\mathbf{n}}}(\mathbb{T})}\}$ of $G_{\widehat{\mathbf{n}}}(\mathbb{T})\times G_{\widehat{\mathbf{n}}}(\mathbb{T})$. Then, the representation $\pi=\pi^{(2)}_{\chi}\circ\Gamma$ can be restricted to the subspace $[\pi_{\chi}(\mathbb{T}^{\widehat{\mathbf{n}}})\xi_{\chi}]$. We claim that this representation is in fact irreducible and spherical.

Denote $M=\pi_\chi(G_{\widehat{\mathbf{n}}}(\mathbb{T}))''$ and $N=\pi_\chi(\mathbb{T}^{\widehat{\mathbf{n}}})''=\pi( \mathbb{T}^{\widehat{\mathbf{n}}})''$. Recall that then there exists the \emph{conditional expectation} $E_{\mathbb{T}}$ of the factor $M$ onto the abelian subalgebra $N$ which can be defined as the unique linear map $E_{\mathbb{T}}\colon M\to N$ (see \cite[Ch.~V.2, Proposition 2.36]{Takesaki}) which satisfies the following properties:
\begin{itemize}
\item 
$\tr(E(b))=\tr(b)$ for all  $b\in M$,

\item 
$E_{\mathbb{T}}(a)=a$ for all $a\in N$,

\item $E_{\mathbb{T}}(ab)=aE_{\mathbb{T}}(b)$ for all $a\in N$ and $b\in M$,

\item $E_{\mathbb{T}}$ can be extended continuously to the orthogonal projection in the Hilbert space $L^2(M,\tr)$ onto the subspace $L^2(N,\tr)\subset L^2(M,\tr)$.
\end{itemize}

Let $P$ be the orthogonal projection onto $[\pi_{\chi}(\mathbb{T}^{\widehat{\mathbf{n}}})\xi_{\chi}]$ (the closed linear span of $\{\pi_{\chi}(\mathfrak{t})\xi_{\chi}:\mathfrak{t}\in\mathbb{T}^{\widehat{\mathbf{n}}}\}$). We claim that the $W^*$-subalgebra of $\mathcal{B}(\mathcal{H}_{\chi})$ generated by the operators $P\pi(g)P$, $g\in G_{\widehat{\mathbf{n}}}$, contains all operators of the form $P\pi_{\chi}^{(2)}(g,h)P$. 
Indeed, for any $\mathfrak{t}\in\mathbb{T}^{\widehat{\mathbf{n}}}$ we have $\pi_{\chi}^{(2)}(\mathfrak{t},\mathfrak{t})P=P\pi_{\chi}^{(2)}(\mathfrak{t},\mathfrak{t})=P$ since $\pi_{\chi}^{(2)}(\mathfrak{t},\mathfrak{t})$ acts trivially on $[\pi_{\chi}(\mathbb{T}^{\widehat{\mathbf{n}}})\xi]$. Hence, for any $h=(\mathfrak{t},s)\in G_{\widehat{\mathbf{n}}}$ we have
\begin{equation*}
P\pi_{\chi}^{(2)}(g,h)P=P\pi_{\chi}^{(2)}(h,h)P\cdot P\pi_{\chi}(h^{-1}g)P=P\pi_{\chi}^{(2)}(s,s)P\cdot P\pi_{\chi}(h^{-1}g)P=P\pi(s)P\cdot P\pi_{\chi}(h^{-1}g)P.
\end{equation*}
It remains to observe that the operator $P\pi_{\chi}(h^{-1}g)P$ belongs to $N$ since $P\pi_{\chi}(h^{-1}g)P=E_{\mathbb{T}}(\pi_{\chi}(h^{-1}g))P$.

It follows that any operator in the commutant of $P\pi(G_{\widehat{\mathbf{n}}})P$ must belong to the commutant of $P\pi_{\chi}^{(2)}(G_{\widehat{\mathbf{n}}}\times G_{\widehat{\mathbf{n}}})P$ as well. Since the representation $\pi^{(2)}_{\chi}$ of the group $G_{\widehat{\mathbf{n}}}\times G_{\widehat{\mathbf{n}}}$ is irreducible, the $W^*$-subalgebra generated by $P\pi_{\chi}^{(2)}(g,h)P$ is $P\mathcal{B}(\mathcal{H}_{\chi})P$ which can be identified with the $\mathcal{B}(P\mathcal{H}_{\chi})$. Therefore, the restriction of representation $\pi$ to $P\mathcal{H}_{\chi}=[\pi_{\chi}(\mathbb{T}^{\widehat{\mathbf{n}}})\xi_{\chi}]$ is irreducible. Finally, note that $\xi_{\chi}$ is an $\mathfrak{S}_{\widehat{\mathbf{n}}}$-spherical vector. It follows that restrictions of the corresponding spherical function and $\chi$ to $\mathbb{T}^{\widehat{\mathbf{n}}}$ coincide since $\langle\pi(\mathfrak{t})\xi_{\chi},\xi_{\chi}\rangle=\langle\pi_{\chi}(\mathfrak{t})\xi_{\chi},\xi_{\chi}\rangle=\chi(\mathfrak{t})$.
\end{proof}

\begin{Rem}
There is an isometry $\mathcal{H}_{\chi}\to L^2(M,\tr)$ that sends $m\xi_{\chi}$ to $m$ for any $m\in M$ (recall that by construction the unit vector $\xi_{\chi}$ is cyclic and separating). In terms of this identification, Proposition \ref{prop:ind_char_and_spher_func} can restated as follows: the assignment $g\mapsto E_{\mathbb{T}}\pi(g)E_{\mathbb{T}}$ defines an irreducible spherical representation of $G_{\widehat{\mathbf{n}}}$ in $L^2(N,\tr)\subset L^2(M,\tr)$ with spherical vector $I_{\mathcal{H}_{\chi}}$.
\end{Rem}

\end{document}